\documentclass[aos]{imsart}

\RequirePackage{amsthm,amsmath,amsfonts,amssymb}

\usepackage[authoryear]{natbib}
\usepackage[T1]{fontenc}
\usepackage{dsfont}

\usepackage[utf8]{inputenc}
\usepackage{bm}
\usepackage{verbatim}
\usepackage{float}
\usepackage{mwe}
\usepackage{color,soul}
\usepackage{threeparttable}
\usepackage{graphicx}
\usepackage{mwe}
\usepackage{mdframed}
\usepackage{xcolor}

\usepackage{multirow}
\usepackage{makecell}
\usepackage{booktabs}
\usepackage{array}
\usepackage{hhline}

\usepackage{lscape}
\usepackage{mathtools}
\usepackage{bbm,amsfonts}
\usepackage{romanbar}
\usepackage{enumerate}
\usepackage{rotating}
\usepackage{tikz}
\usetikzlibrary{matrix}
\usetikzlibrary{calc,intersections}
\usepackage{mathrsfs}

\usepackage{soul}

\startlocaldefs
\theoremstyle{plain}

\newtheorem{theorem}{Theorem}[section]
\newtheorem{corollary}{Corollary}[section]

\newtheorem{proposition}{Proposition}[section]

\theoremstyle{remark}

\newtheorem{remark}{Remark}[section]

\usepackage{color}

\definecolor{green}{rgb}{0.0, 0.5, 0.0}

\newcommand{\bol}[1]{\mbox{\boldmath$#1$}}

\newcommand{\bSigma}{\mathbf{\Sigma}}

\newcommand{\bmu}{\bol{\mu}}

\newcommand{\btheta}{\bol{\theta}}
\newcommand{\bb}{\mathbf{b}}

\newcommand{\bx}{\mathbf{x}}

\newcommand{\by}{\mathbf{y}}
\newcommand{\bby}{\bar{\mathbf{y}}}
\newcommand{\bbx}{\bar{\mathbf{x}}}

\newcommand{\tbY}{\widetilde{\mathbf{Y}}}

\newcommand{\bH}{\mathbf{H}}

\newcommand{\bJ}{\mathbf{J}}

\newcommand{\bX}{\mathbf{X}}
\newcommand{\bY}{\mathbf{Y}}
\newcommand{\bv}{\mathbf{v}}

\newcommand{\bOne}{\mathbf{1}}

\newcommand{\bI}{\mathbf{I}}

\newcommand{\tr}{\operatorname{tr}}

\newcommand{\bA}{\bol{A}}

\newcommand{\bxi}{\boldsymbol{\xi}}

\newcommand{\bS}{\mathbf{S}}

\newcommand{\bU}{\mathbf{U}}
\newcommand{\bTheta}{\mathbf{\Theta}}
\newcommand{\ta}{\alpha}
\newcommand{\tb}{\beta}

\newcommand{\bi}{\mathbf{1} }

\newcommand{\argmin}{\mathop{\mathrm{argmin}}}

\DeclareMathOperator*{\argmax}{argmax}

\endlocaldefs

\makeatletter

\newcommand*{\addFileDependency}[1]{
\typeout{(#1)}
\@addtofilelist{#1}
\IfFileExists{#1}{}{\typeout{No file #1.}}
}\makeatother

\RequirePackage[colorlinks,citecolor=blue,urlcolor=blue]{hyperref}
\begin{document}

\begin{frontmatter}
\title{Reviving pseudo-inverses:\\ Asymptotic properties of large dimensional Moore-Penrose and Ridge-type inverses with applications}
\runtitle{Asymptotic properties of large inverses}

\begin{aug}
\author[A]{\fnms{Taras}~\snm{Bodnar}\ead[label=e1]{taras.bodnar@liu.se}}
\and
\author[B]{\fnms{Nestor}~\snm{Parolya}\ead[label=e2]{n.parolya@tudelft.nl}}
\address[A]{{Department of Management and Engineering, Link\"{o}ping University, SE-581 83 Link\"{o}ping, Sweden\printead[presep={,\ }]{e1}}}
\address[B]{Department of Applied Mathematics, Delft University of Technology, Mekelweg 4, 2628 CD Delft, The Netherlands\printead[presep={,\ }]{e2}}
\end{aug}

\begin{abstract}
 In this paper, we derive high-dimensional asymptotic properties of the Moore-Penrose inverse and, as a byproduct, of various ridge-type inverses of the sample covariance matrix. In particular, the analytical expressions of the asymptotic behavior of the weighted sample trace moments of generalized inverse matrices are deduced in terms of the partial exponential Bell polynomials which can be easily computed in practice. The existent results for pseudo-inverses are extended in several directions: (i) First, the population covariance matrix is not assumed to be a multiple of the identity matrix; (ii) Second, the assumption of normality is not used in the derivation; (iii) Third, the asymptotic results are derived under the high-dimensional asymptotic regime. Our findings provide universal methodology for construction of fully data-driven improved shrinkage estimators of the precision matrix, optimal portfolio weights and beyond. It is found that the Moore-Penrose inverse acts asymptotically as a certain regularizer of the true covariance matrix and it seems that its proper transformation (shrinkage) performs similarly to or even outperforms the existing benchmarks in many applications, while keeping the computational time as minimal as possible.
\end{abstract}



\end{frontmatter}

\section{Introduction}\label{sec:intro}

The covariance matrix and its inverse, the precision matrix, are present in many different applications. The covariance matrix is typically regarded as a multivariate measure of uncertainty and as a measure of linear dependence structure (see \cite{rencher2012methods}), while the precision matrix is present in the formulae of the weights of optimal portfolios in finance (see \cite{ao2019approaching}, \cite{cai2020high}, \cite{kan2022optimal}, \cite{bodnarokhrinparolya2023}, \cite{lassance2023optimalcombination}), in the expression of the minimum variance filter in signal processing (see \cite{PalomarBook2016}), in high-dimensional time-series analysis (cf., \cite{heiny2019random}, \cite{heiny2021large}), in prediction and test theory in multivariate and high-dimensional statistics (see \cite{chen2010tests}, \cite{cai2011limiting}, \cite{yao_zheng_bai_2015}, \cite{bodnardetteparolya2019}, \cite{shi2022universally}).

In practice, the unknown true covariance matrix is commonly estimated by the sample covariance matrix, while one of the most used estimators of the precision matrix is the inverse of the sample covariance matrix. The sample estimator of the covariance matrix possesses several important properties. First, it is an unbiased estimator of the true covariance matrix. Second, when the dimension of the data-generating model is fixed and the sample size tends to infinity, then both the sample covariance matrix and its inverse are consistent estimators for the population covariance matrix and the population precision matrix, respectively (see \cite{muirhead1990}).

The situation becomes challenging in the high-dimensional setting, i.e., when the model dimension is proportional to the sample size, especially if it is larger than the sample size. It is known as the high-dimensional asymptotic regime or Kolmogorov asymptotics (see \cite{bai2010spectral}). Neither the sample covariance matrix nor its inverse are consistent estimators of the corresponding population quantities without imposing some restrictions on the structure of the true covariance/precision matrix. The issue is even more difficult when the observation vectors are taken from a heavy-tailed distribution (see, e.g., \cite{heiny2022limiting}) and the dimension of the data-generating model is larger than the sample size. In the latter case, the sample covariance matrix is singular and its inverse cannot be constructed (cf., \cite{muirhead1990}, \cite{srivastava2003singular}).

Matrix algebra proposes several ways how a generalized or pseudo-inverse of a singular matrix can be defined (see, e.g., \cite{penrose1955generalized}, \cite{rao1972generalized}, \cite{ben2003generalized}, \cite{wang2018generalized}). Although the generalized inverse is not uniquely determined in general, there exists a specific type of pseudo-inverse matrices, which is uniquely defined. This is the Moore-Penrose inverse, which is also a least squares generalized inverse (see, e.g., \cite{harville1997matrix}). 

Even though the  properties of the Moore-Penrose inverse of deterministic matrices have been studied in the literature (cf., \cite{meyer1973generalized}, \cite{harville1997matrix}), only a few results are available for the Moore-Penrose inverse of the sample covariance matrix which were derived under very strict assumptions imposed on the data-generating model. Regarding the finite-sample properties, the sample covariance matrix has a singular Wishart distribution under the assumption of normality (\cite{srivastava2003singular}). The density function of the Moore-Penrose inverse of a singular Wishart distributed random matrix was derived in \cite{bodnar2008properties}. The resulting expression appears to be very complicated. As such, the moments (i.e., expectations of powers of the matrix) of the Moore-Penrose inverse of the sample covariance matrix cannot even be obtained under the assumption of normality imposed on the data-generating model when the true population covariance matrix is not restricted to be proportional to the identity matrix. Only the expressions of the upper and lower limits for the mean matrix and covariance matrix of the Moore-Penrose inverse of the sample covariance matrix were derived in \cite{imori2020mean} in the general case, while \cite{cook2011mean} provided the exact mean matrix and the covariance matrix in the very restrictive special case when the true covariance matrix is proportional to the identity matrix. Both these results are non-asymptotic and were obtained when the data-generating model follows a multivariate normal distribution. No other results have been derived either for the Moore-Penrose inverse or for the ridge-type inverse in the literature in the non-asymptotic setting to the best of our knowledge, even though both matrices are widely used in practice (see, e.g., \cite{ben2003generalized}, \cite{wang2018generalized}). 
The ridge-type inverse was used for constructing shrinkage estimators of the precision matrix in \cite{kubokawa2008estimation} and \cite{wang2015shrinkage}. While \cite{kubokawa2008estimation} derived a shrinkage estimator by assuming that the observation matrix is drawn from the multivariate normal distribution, \cite{wang2015shrinkage} derived the results in the general case by using the methods of random matrix theory. Lastly, while the limiting spectral distribution and central limit theorem (CLT) for linear spectral statistics of the Moore-Penrose inverse are discussed in \cite{bodnardetteparolya2016}, explicit results involving eigenvectors are not available. To date, this remains the only known asymptotic analysis of the large-sample behavior of the Moore-Penrose inverse of the sample covariance matrix.

We contribute to the existing literature in several directions. First, the asymptotic limits of higher-order weighted empirical trace moments (see, e.g., \cite{golub1994estimates} and \cite{BAI1996}) of Moore-Penrose and ridge-type inverses are derived by imposing no specific distributional assumption on the data-generating model. This addresses a well-known challenge in multivariate statistics, particularly when the data dimension exceeds the sample size. Our general weighting scheme enables insights into the asymptotic behavior of the matrix elements, the interaction between eigenvalues and eigenvectors, and lays the groundwork for developing new shrinkage estimators. Second, we assume only the existence of fourth moments and impose no structural assumptions on the population covariance matrix beyond positive definiteness and boundedness of its spectrum. Third, our results are established under a high-dimensional asymptotic regime, where both the dimension and sample size tend to infinity. Fourth, we extend existing results by deriving asymptotic properties of higher-order weighted trace moments for ridge-type inverses. Fifth, as a by-product of our methodology, we provide closed-form expressions for the asymptotic equivalents of the higher-order weighted trace moments of the sample covariance matrix itself. Finally, the developed theory is applied in the derivation of improved shrinkage estimators of the precision matrix and optimal portfolio weights.

The rest of the paper is structured as follows. In Section \ref{sec:main}, the main theoretical results are presented. The weighted sample trace moments of the sample Moore-Penrose inverse are given in Section \ref{sec:main-MP}, while Section \ref{sec:main-ridge} and Section \ref{sec:main-MP-ridge} present similar findings derived for the ridge-type inverse and the Moore-Penrose-ridge inverse. Section S.3 in the supplement (\cite{BP2025reviving-S}) provides additional results derived for the sample covariance matrix. The obtained theoretical results are used in the derivation of shrinkage estimators of the precision matrix in Section \ref{sec:shrink-prec} and optimal portfolio weights in Section \ref{sec:shrink-gmv}. The finite-sample performance of the theoretical findings is investigated via an extensive simulation study in Section \ref{sec:sim}. Finally, the proofs and some additional simulations are postponed to the supplementary material (\cite{BP2025reviving-S}).

\section{Asymptotic properties of pseudo-inverse matrices}\label{sec:main}
Let $\bY_n=(\by_1,\by_2,...,\by_n)$ be the $p \times n$ observation matrix with $p>n$ and let $\mathbb{E}(\mathbf{y}_i)=\bmu$ and $\mathbb{C}ov(\mathbf{y}_i)=\bSigma$ for $i \in 1,...,n$. Throughout the paper, it is assumed that there exists a $p\times n$ random matrix $\bX_n=(\bx_1,\bx_2,...,\bx_n)$ which consists of independent and identically distributed (i.i.d.) random variables with zero mean and unit variance such that
\begin{equation}\label{obs}
\bY_n \stackrel{d}{=} \bmu \bi_n^\top + \bSigma^{{1}/{2}}\bX_n ,
\end{equation}
where $\bi_n$ denotes the $n$-dimensional vector of ones and the symbol $\stackrel{d}{=}$ denotes the equality in distribution.
Using the observation matrix $\bY_n$, the sample estimator of the covariance matrix $\bSigma$ is given by
\begin{equation}\label{bS}
\bS_n= \frac{1}{n}\bY_n\bY_n^\top -\bby_n\bby_n^\top
\quad \text{with} \quad \bby_n= \frac{1}{n}\bY_n \bi_n.
\end{equation}

In many statistical and machine learning applications, one seeks best possible estimators for certain functions of the high-dimensional precision matrix $\bSigma^{-1}$ using the sample covariance matrix $\bS_n$. This problem becomes particularly challenging for the case $p>n$, due to the singularity of $\bS_n$, as one must somehow ``invert'' a non-invertible matrix. There are plenty of ways to ``invert'' $\bS_n$ when $p>n$. For a while, we assume that one of the methods has been chosen and denote the corresponding generalized inverse by $\bS^{\#}_n(t)$ where $t$ is a tuning parameter. Since none of these inverses presents a consistent estimator of $\bSigma^{-1}$ in the high-dimensional setting, much of the research focuses on finding an appropriate matrix- or vector-valued function $f(\bS_n^{\#}(t))$ that  asymptotically minimizes the distance to the target transformation of the precision matrix $g(\bSigma^{-1})$ for some prespecified function $g$ (see, e.g., \cite{Lam2020} and references therein). This can be formalized in the following way
\begin{eqnarray}\label{intro_obj}
\argmin\limits_{f\in\mathcal{C}_f}||f(\bS^{\#}_n(t))-g(\bSigma^{-1})||^q_{l_q}   
\end{eqnarray}
with $\mathcal{C}_f$ being a class of functions of interest and $||\cdot||_{l_q}$ denoting the $l_q$-norm. For example, by taking the quadratic norm $l_2$ and $\mathcal{C}_f$ as a class of linear functions, which act on the eigenvalues of $\bS_n$, one can be interested in the estimation of $\bSigma^{-1}$ itself. In that case, the well-established linear shrinkage techniques presented, e.g., in \cite{lw2004}, \cite{BodnarGuptaParolya2014} and \cite{BodnarGuptaParolya2016} among many others, can be recovered. Similarly, nonlinear shrinkage estimators were derived in \cite{lw12}, \cite{lw20} and \cite{lwQIS2020} for a nonlinear function $f$. Also, the functions $f$ and $g$ can be vector valued, which happens often in financial applications, e.g., shrinkage estimation of the optimal portfolio weights (see \cite{frahm2010}, \cite{golosnoy2007multivariate}, \cite{bodnar2018estimation} among others). Taking a closer look at the aforementioned papers and the optimization problem \eqref{intro_obj} itself, we deal very often with the moments of the specific estimator of $\bSigma^{-1}$. More precisely, one has to establish the asymptotic behavior of the weighted trace moments alike $\tr\left[(\bS^{\#}_n(t))^m\bTheta\right]$ with $m=1, 2, 3, \ldots$ for some initial estimator $\bS^{\#}_n(t)$ of the precision matrix and a weighting matrix $\bTheta$ possibly dependent on $\bSigma$. More details can be found in Section \ref{sec:app}, which is concerned with the application of the theoretical results of this paper to the shrinkage estimation techniques.

In this section, we investigate the properties of the weighted sample trace moments of three pseudo-inverses of the sample covariance matrix $\bS_n$ under the high-dimensional asymptotic regime. Namely, we deal with $\text{tr}\left[(\bS_n^{\#}(t))^m\bTheta\right]$ for $m=1,2,\ldots$ where $\bS_n^{\#}(t)$ is one of the pseudo-inverses considered: (i) Moore-Penrose inverse $\bS_n^+$ in Section \ref{sec:main-MP}, (ii) ridge-type inverse $\bS_n^-(t)$ in Section \ref{sec:main-ridge}, and (iii) Moore-Penrose-ridge inverse $\bS_n^\pm(t)$ in Section \ref{sec:main-MP-ridge}. Both the ridge-type inverse and the Moore-Penrose-ridge inverse are functions on a tuning parameter $t$ whose optimal value depends on the application at hand (see Section \ref{sec:shrink-prec} and Section \ref{sec:shrink-gmv} for details). 

Theoretical results are derived by imposing the following three assumptions.
\begin{itemize}
    \item[\textbf{(A1)}] $\bSigma$ is a nonrandom positive definite matrix, which satisfies
\[0 < \inf_p \lambda_{\min}(\Sigma) \leq \sup_p \lambda_{\max}(\Sigma) < \infty.\]
\item[\textbf{(A2)}] The elements of $\bX_n$ have bounded $4+\varepsilon$ moments for some $\varepsilon>0$.
\item[\textbf{(A3)}] $\bTheta$ is symmetric and
\begin{itemize}
    \item[] \hspace{0.3cm}(i) it has finite rank $k$, i.e., $\bTheta=\sum_{i=1}^k\btheta_i\btheta_i^\top$, and $\sup\limits_p\btheta_i^\top \btheta_i<\infty$ for all $i=1,...,k$;
    \item[] \hspace{0.3cm}(ii) or $\lambda_{max}(\bTheta)=o(1)$ as $p \to \infty$, otherwise.
\end{itemize}
\end{itemize}

\vspace{0.1cm}
The first assumption ensures that the smallest and the largest eigenvalues of the population covariance matrix are uniformly bounded in $p$ away from zero and infinity. As such, the only source of the singularity of the sample covariance matrix $\bS_n$ is the lack of data, that is, $n<p$. This is a classical technical assumption imposed in random matrix theory (see, e.g., \cite{pan2014comparison}, \cite{lw20}). Assumption \textbf{(A2)} imposes no distributional assumptions on the data-generating model and presents the classical conditions used in random matrix theory (cf., \cite{bai2010spectral}, \cite{wang2015shrinkage}, \cite{ledoit2021shrinkage}).

\begin{remark}[Discussion on Assumption (A3)]
While the matrix $\bTheta$ can be chosen with some flexibility, its definition inherently includes a normalization factor. For instance, in the case of sample trace moments of a generalized inverse, the choice $\bTheta = \bI_p/p$ leads to all eigenvalues being equal to $1/p$. When $\bTheta$ has finite rank, a bounded Euclidean norm is assumed for each vector $\btheta_i$, $i=1, \ldots, k$. Consequently, if these vectors are not sparse, their entries are expected to be relatively small.
It is important to highlight that the assumption on $\bTheta$ is often of a technical nature. In many applications, expressions involving $\bS_n^{\#}(t)$ appear as ratios (see Section \ref{sec:app}), and adjustments in the definition of $\bTheta$ can be made accordingly. Moreover, Assumption \textbf{(A3)} may be replaced by a slightly stronger but more convenient condition: that $\bTheta$ has bounded trace norm. This alternative assumption requires that
\begin{itemize}
\item[\textbf{(A3$^\prime$)}] $\sup\limits_p \tr\left[(\bTheta^\top\bTheta)^{1/2}\right] < \infty$,
\end{itemize}
and is often easier to verify in practice.
If $\bTheta$ is not symmetric, it should be replaced by its symmetrized version $(\bTheta + \bTheta^\top)/2$. Alternatively, bilinear forms can be addressed by using the identity 
$$\btheta\boldsymbol{\eta}^\top + \boldsymbol{\eta}\btheta^\top = ((\btheta + \boldsymbol{\eta})(\btheta + \boldsymbol{\eta})^\top - (\btheta - \boldsymbol{\eta})(\btheta - \boldsymbol{\eta})^\top)/2$$ 
for any vectors $\btheta$ and $\boldsymbol{\eta}$ in $\mathbb{R}^p$. Therefore, Assumption \textbf{(A3)} is quite general and can either be readily verified or satisfied by normalizing $\bTheta$ via $\tr\left[(\bTheta^\top\bTheta)^{1/2}\right]$. Indeed, if \textbf{(A3)} is not met, then the quantities of interest, $\text{tr}\left[(\bS_n^{\#}(t))^m\bTheta\right]$, either diverge or exhibit persistent fluctuations (e.g., asymptotic normality). The latter case will be explored in future research.
\end{remark}

The typical examples of the weighting matrix $\bTheta$ include several key instances that are crucial in random matrix theory and statistical applications:

\begin{itemize}
    \item Setting $\bTheta$ equal to $\frac{1}{p}\bI_p$ and $m = 1$ leads to the Stieltjes transform restricted to positive real line, which is a concept often associated with classical random matrix theory. This transform is fundamental in understanding the spectral distribution of different types of random matrices (see \cite{bai2010spectral}).
    \item When $\bTheta$ is defined as $\btheta\boldsymbol{\eta}^\top$, it pertains to bilinear or quadratic forms that involve eigenvectors. This approach focuses on the interaction between eigenvalues and eigenvectors.
    \item Choosing $\bTheta$ as $\frac{1}{p}g(\bSigma)$ for some  $g(\cdot)$ corresponds to the Ledoit-Péché functionals, which involve nonlinear shrinkage formulas. These objects are vital for their ability to adjust covariance matrix estimators in the high-dimensional setting (see \cite{ledoitpeche2011}).
\end{itemize}

The significance of exploring these different forms of $\bTheta$ includes the construction of new shrinkage estimators for cases where the ratio of the dimensionality to the sample size exceeds one, which will be elaborated upon later in Section \ref{sec:shrink-prec} and Section \ref{sec:shrink-gmv}. The specification of $\bTheta$ as $\btheta\btheta^\top$ unveils information about the eigenvectors and arises frequently in portfolio analysis (see \cite{bodnarokhrinparolya2023}). Finally, utilizing $\bTheta$ as $\mathbf{e}_i\mathbf{e}_j^\top$, where $\mathbf{e}_i$ is a unit vector with one in the i-th position, helps in deriving the asymptotics of the matrix entries. 

In the formulation of the theoretical results, we use the partial exponential Bell polynomials, which are defined by (see \cite{bell1927partition, bell1934exponential})
\begin{equation}\label{Bell-pol}
    B_{m,k}(x_{1},x_{2}, ... ,x_{m-k+1})=\sum \frac{m!}{j_{1}!j_{2}!... j_{m-k+1}!}     \left(\frac{x_{1}}{1!}\right)^{j_{1}}
\left(\frac{x_{2}}{2!}\right)^{j_{2}}
... \left(\frac{x_{m-k+1}}{(m-k+1)!}\right)^{j_{m-k+1}},
\end{equation} 
where the sum is taken over all sequences $j_1, ... , j_{m-k+1}$ of non-negative integers such that $\sum_{l=1}^{m-k+1}j_l=k$ and $\sum_{l=1}^{m-k+1} l j_l=m$. In practice, the Bell polynomials can be easily computed in the R-package {\color{blue}\textit{kStatistics}}, see also \cite{di2008unifying}.

Further, throughout the paper, we will need a specific function $v(t): \mathbb{R}^+\to\mathbb{R}$, which satisfies the following asymptotic equation
\begin{equation}\label{th2-vt}
\frac{1}{p}\text{tr}\left[\left(v(t)\bSigma+\bI_p\right)^{-1}\right]=\frac{c_n-1+tv(t)}{c_n},
\end{equation}
with $c_n=p/n$. It is easy to verify that there exists the limit of $v(t)$ as $t\to0$ and, consequently, we denote it by $v(0)=\lim\limits_{t\to0}v(t)$. Moreover, from \eqref{th2-vt}, it is straightforward to find that $v(0)$ is a unique solution of the equation
\begin{equation}\label{th1-v0}
    \frac{1}{p}\text{tr}\left[\left(v(0)\bSigma+\bI_p\right)^{-1}\right]=\frac{c_n-1}{c_n} \quad \text{with}\quad c_n=\frac{p}{n}.
\end{equation}
Although the function $v(t)$ is defined as a solution of a nonlinear equation \eqref{th2-vt} and it is not available in closed form, we show that it is a decreasing function in Proposition \ref{th2a} whose proof is given in the supplement.
\begin{proposition}\label{th2a}
Let $\bY_n$ fulfill the stochastic representation \eqref{obs}. Then, under Assumptions \textbf{(A1)}-\textbf{(A3)} , $v(t)$ is strictly decreasing for $t \ge 0$ and $c_n >1$.
\end{proposition}

All results in this paper are expressed in terms of the function $v(t)$, its value at zero $v(0)$, and its derivatives $v^{(1)}(t), v^{(2)}(t), \ldots$, along with their respective values at zero, $v^{(1)}(0), v^{(2)}(0), \ldots$, and so on. Furthermore, we demonstrate how these quantities can be consistently estimated in the high-dimensional asymptotic regime where $p/n \to c > 1$ as $n \to \infty$. Accordingly, we now present the consistent estimators of the function $v(t)$, its derivatives, and their values at zero for $m=0, 1, 2, \ldots$ expressed as
\begin{eqnarray}\label{hv0-all1}
\hat{v}^{(m)}(t)&=& (-1)^{m} m! c_n \left(\frac{1}{p}\text{tr}\left[ (\bS_n^{-}(t))^{m+1}\right]-t^{-(m+1)} \frac{c_n-1}{c_n}\right),\\
\hat{v}^{(m)}(0)&=& (-1)^{m} m! c_n \frac{1}{p}\text{tr}\left[ (\bS_n^+)^{m+1}\right],\label{hv0-all}
\end{eqnarray}
with $\hat{v}^{(0)}(t)= \hat{v}(t)$ and $\hat{v}^{(0)}(0)= \hat{v}(0)$. Here, $\bS_n^+$ and $\bS_n^{-}(t)$ are the Moore–Penrose and the ridge-type inverse, defined at the beginning of Sections \ref{sec:main-MP} and \ref{sec:main-ridge}, respectively. The expressions \eqref{hv0-all1} and \eqref{hv0-all} are deduced in Corollaries \ref{cor0} and \ref{cor0Ridge}. 

Throughout the paper, we will assume that the dimension $p$ is a function of the sample size $n$, i.e., $p = p(n)$. All limits are for $n\to\infty$, unless explicitly stated otherwise. At the beginning of the supplement, an abstract {\it sketch of the proof} is presented, which applies to all theorems proved in Section \ref{sec:main}. 

\subsection{Weighted moments of the sample Moore-Penrose inverse}\label{sec:main-MP}

The Moore-Penrose inverse of $\bS_n$ is uniquely defined as the matrix $\bS_n^+$ which fulfills the following four conditions:
\begin{enumerate}
    \item[] (i)~ $\bS_n^+ \bS_n\bS_n^+ =\bS_n^+$, \qquad\qquad (ii) $\bS_n \bS_n^+\bS_n =\bS_n$,
    \item[] (iii) $(\bS_n^+ \bS_n)^\top =\bS_n^+ \bS_n$,
    \qquad~  (iv) $(\bS_n \bS_n^+)^\top =\bS_n\bS_n^+ $.
\end{enumerate}

Theorem \ref{th1} presents the behavior of the weighted sample trace moments of the Moore-Penrose inverse, $\text{tr}( (\bS_n^+)^m\bTheta)$, under the high-dimensional asymptotic regime. 

\begin{theorem}\label{th1}
Let $\bY_n$ fulfill the stochastic representation \eqref{obs}. 
Then, under Assumptions \textbf{(A1)}-\textbf{(A3)}, it holds for $m=1, 2, \ldots$ that
\begin{equation}\label{th1_eq1}
\left|\emph{tr} \left[ (\bS_n^+)^m\bTheta\right]-s_m(\bTheta)\right| \stackrel{a.s.}{\rightarrow} 0\quad\text{for} \quad p/n \rightarrow c \in (1,\infty) 
\quad \text{as} \quad n \rightarrow \infty,
\end{equation}
where 
\begin{equation}\label{th1-sm}
    s_m(\bTheta)= \sum_{k=1}^m \frac{(-1)^{m+k+1} k!}{m!}d_k(\bTheta)
    B_{m,k}\left(v^{(1)}(0),...,v^{(m-k+1)}(0)\right),
\end{equation}
with
\begin{eqnarray}\label{th1-dk}
d_k(\bTheta)&=&\emph{tr}\left\{\left(v(0)\bSigma+\bI_p\right)^{-1}\left[\bSigma\left(v(0)\bSigma+\bI_p\right)^{-1}\right]^{k}\bTheta\right\}, \quad k=1,2,..., 
\end{eqnarray}
$v(0)$ the unique solution of the equation \eqref{th1-v0} and its derivatives satisfy
\begin{equation}\label{th1-v0pr}
v^{(1)}(0)=-\frac{1}{\frac{1}{v(0)^2}-c_n  \frac{1}{p}\emph{tr}\left\{\left[\bSigma\left(v(0)\bSigma+\bI_p\right)^{-1}\right]^2\right\}}, 
\end{equation}
$v^{(2)}(0)$,...,$v^{(m)}(0)$ are computed recursively by
\begin{equation}\label{th1-v0pr-m}
v^{(m)}(0)=-v^{(1)}(0)\sum_{k=2}^m (-1)^{k} k! h_{k+1}
 B_{m,k}\left(v^{(1)}(0),...,v^{(m-k+1)}(0)\right)
 ~~\text{with}
\end{equation}
\begin{eqnarray}\label{th1-hk}
  h_k=\frac{1}{[v(0)]^{k}}-c_n\frac{1}{p}\emph{tr}\left\{\left[\bSigma\left(v(0)\bSigma+\bI_p\right)^{-1}\right]^{k}\right\}, \quad k=1,2,....  
\end{eqnarray}
\end{theorem}

In Corollary \ref{cor1}, we compute closed-form expressions in case $m=1,2,3,4$.

\begin{corollary}\label{cor1}
Let $\bY_n$ fulfill the stochastic representation \eqref{obs}. Then, under Assumptions \textbf{(A1)}-\textbf{(A3)} we get \eqref{th1_eq1} with
\begin{eqnarray*}
s_1(\bTheta)&=& -v^{(1)}(0) d_1(\bTheta), \quad
s_2(\bTheta)= \frac{1}{2}v^{(2)}(0) d_1(\bTheta)-[v^{(1)}(0)]^2 d_2(\bTheta),\\
s_3(\bTheta)&=&-\frac{1}{6}v^{(3)}(0) d_1(\bTheta)+v^{(1)}(0)v^{(2)}(0) d_2(\bTheta)
-[v^{(1)}(0)]^3 d_3(\bTheta),\\
s_4(\bTheta)&=&\frac{1}{24}v^{(4)}(0) d_1(\bTheta)- \frac{1}{12}\left(4v^{(1)}(0)v^{(3)}(0)+3[v^{(2)}(0)]^2\right) d_2(\bTheta)\\
&&+\frac{3}{2}[v^{(1)}(0)]^2 v^{(2)}(0)  d_3(\bTheta)-[v^{(1)}(0)]^4 d_4(\bTheta),
\end{eqnarray*}
for $p/n \rightarrow c \in (1,\infty)$ as $n \rightarrow \infty$ 
where $v(0)$ is the unique solution of \eqref{th1-v0}, 
\begin{eqnarray*}
v^{(1)}(0)&=&  -h^{-1}_2,\quad
v^{(2)}(0)=-2[v^{(1)}(0)]^3 h_3 ,\quad
v^{(3)}(0)=-6[v^{(1)}(0)]^2 v^{(2)}(0) h_3 +6[v^{(1)}(0)]^4 h_4 ,\\
v^{(4)}(0)&=&
-\left(8 [v^{(1)}(0)]^2 v^{(3)}(0) + 6v^{(1)}(0)[v^{(2)}(0)]^2 \right) h_3 
+36[v^{(1)}(0)]^3 v^{(2)}(0) h_4 
-24[v^{(1)}(0)]^5 h_5 ,
\end{eqnarray*}
and $d_k(\bTheta)$ and $h_k$ are defined in \eqref{th1-dk} and \eqref{th1-hk}, respectively.
\end{corollary}

The scenario with $m=1$ illustrates that the Moore-Penrose inverse inherently functions as a particular form of regularization. Specifically, it exhibits asymptotic behavior similar/proportional to $$\left(v(0)\bSigma+\bI_p\right)^{-1}\bSigma\left(v(0)\bSigma+\bI_p\right)^{-1}=\frac{1}{v(0)^2}\left(\bSigma+v^{-1}(0)\bI_p\right)^{-1}\bSigma\left(\bSigma+v^{-1}(0)\bI_p\right)^{-1}.$$ This observation prompts the exploration of an alternative type of inverse in Section \ref{sec:main-MP-ridge}.

Corollary \ref{cor0} provides the results for a special case $\bTheta=({1}/{p})\bI_p$ with all eigenvalues equal to $1/p$, Consequently, it fulfills the condition about the weighting matrix $\bTheta$ stated in Assumption \textbf{(A3)}.

\begin{corollary}\label{cor0}
Let $\bY_n$ fulfill the stochastic representation \eqref{obs}.
Then, under Assumptions \textbf{(A1)}-\textbf{(A2)} it holds for $m=1,2,...$ that
\begin{equation}\label{cor0_eq1}
\left|\frac{1}{p}\emph{tr}\left[ (\bS_n^+)^m\right]-\frac{(-1)^{m-1}v^{(m-1)}(0)}{(m-1)! c_n}\right| \stackrel{a.s.}{\rightarrow} 0\quad\text{for} \quad p/n \rightarrow c \in (1,\infty) 
\quad \text{as} \quad n \rightarrow \infty,
\end{equation}
where
$v^{(m)}(0)$ is defined in \eqref{th1-v0}, \eqref{th1-v0pr}, and \eqref{th1-v0pr-m} for $m=0,1,...$ with $v^{(0)}(0)=v(0)$.
\end{corollary}

The results of Corollary \ref{cor0} have a number of important applications. In particular, they allow us to approximate $v(0)$ and $v^{(m)}(0)$, $m=1,2,...$ even though $v(0)$ is defined as a solution of the nonlinear equation \eqref{th1-v0} which can be solved only numerically and the definition of $v^{(m)}(0)$, $m=1,2,...$, depends on $v(0)$. Using the findings of Corollary \ref{cor0}, we get the closed-form expressions of consistent estimators for $v^{(m)}(0)$, $m=0,1,...$ expressed with the formula  \eqref{hv0-all}, namely
\begin{equation*}
\hat{v}^{(m)}(0)= (-1)^{m} m! c_n \frac{1}{p}\text{tr}\left[ (\bS_n^+)^{m+1}\right].
\end{equation*}

In Figure \ref{fig:MP}, the finite-sample performance of the estimators in \eqref{hv0-all} is depicted when $m=0$ and $m=1$ under the assumption of the normal distribution and $t$-distribution for $p/n \in (1,5]$ and $n\in\{100,250,500\}$. We observe that the estimators of $v(0)$ and $v^{(1)}(0)$ converge to their true values already for small sample size $n=100$ when the concentration ratio $c_n$ is larger than 2. When $c_n \in (1,2]$, then a larger sample size is needed. These findings hold independently of the distribution used to generate the elements of $\bX_n$, while the convergence is slower for the $t$-distribution. Finally, we note that the convergence is faster for $v(0)$.

\vspace{-1cm}
\begin{figure}[h!t]
\begin{tabular}{ll}
\hspace{-0.5cm}\includegraphics[width=7cm]{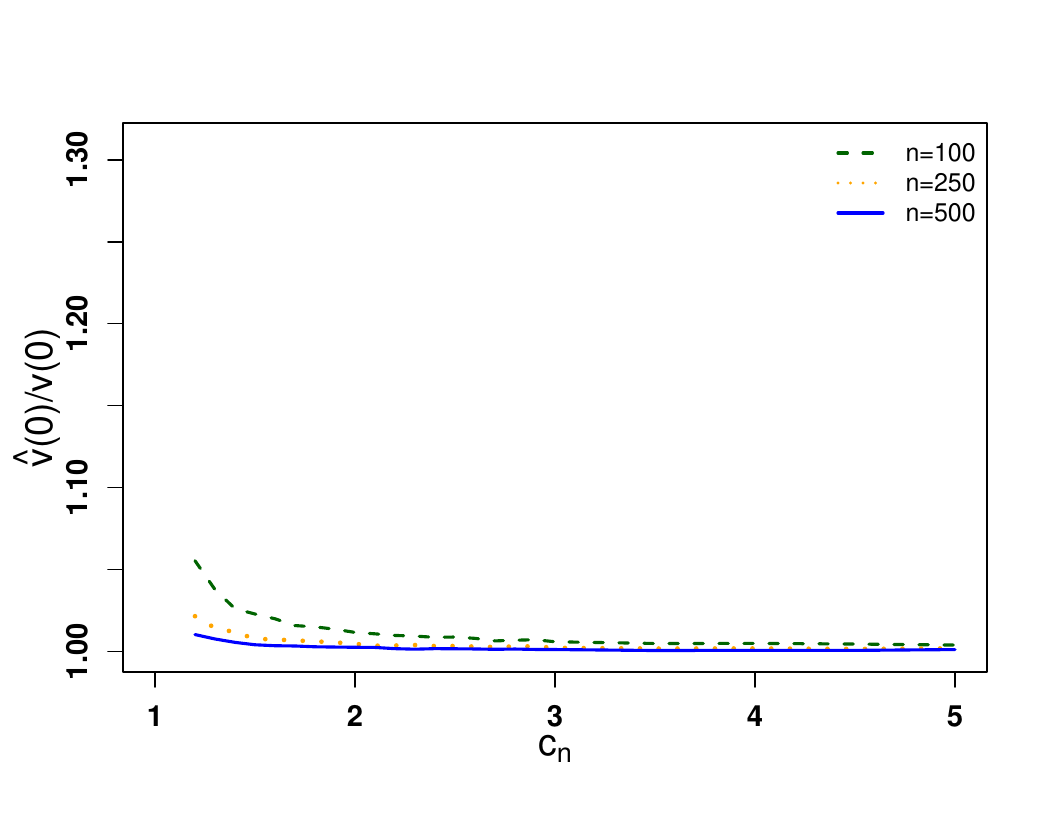}&
\hspace{-0.5cm}\includegraphics[width=7cm]{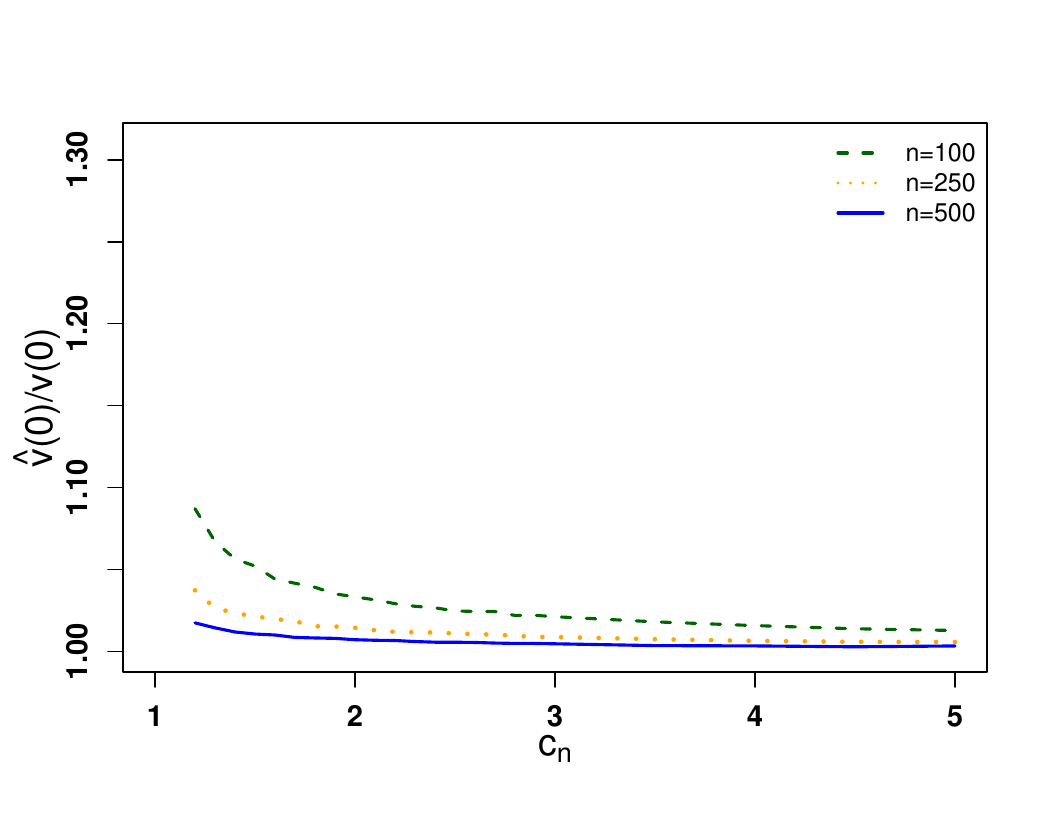}\\[-1.1cm]
\hspace{-0.5cm}\includegraphics[width=7cm]{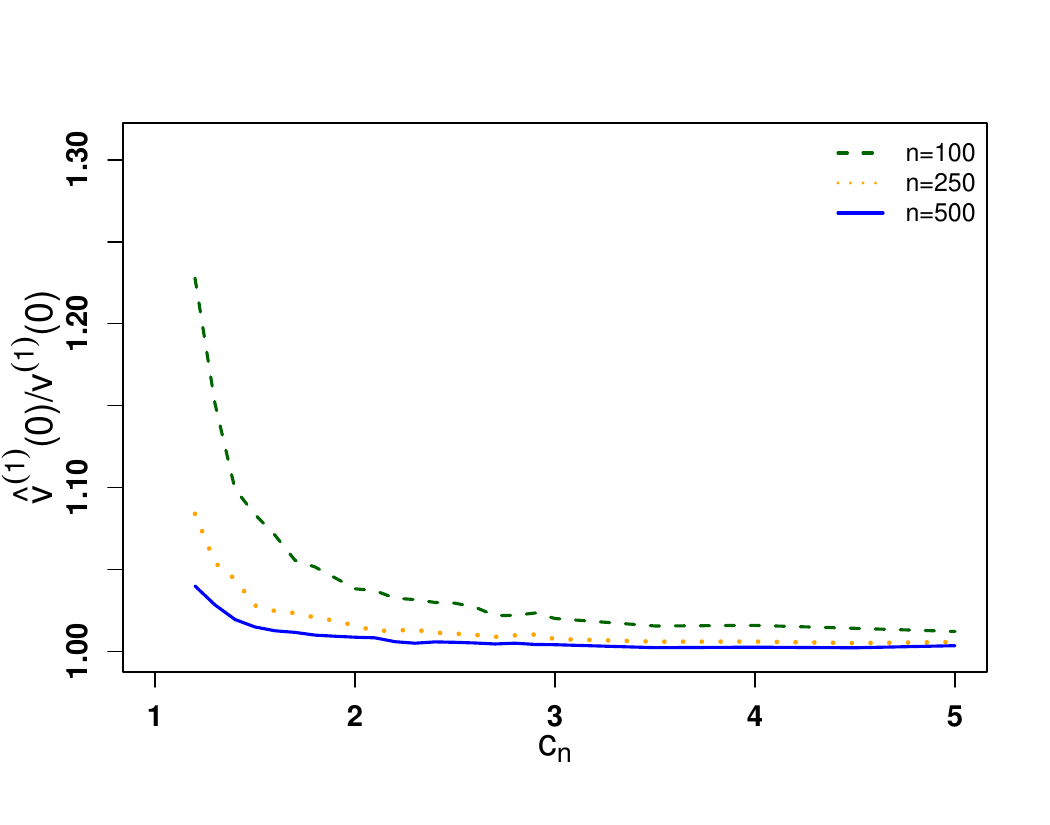}&
\hspace{-0.5cm}\includegraphics[width=7cm]{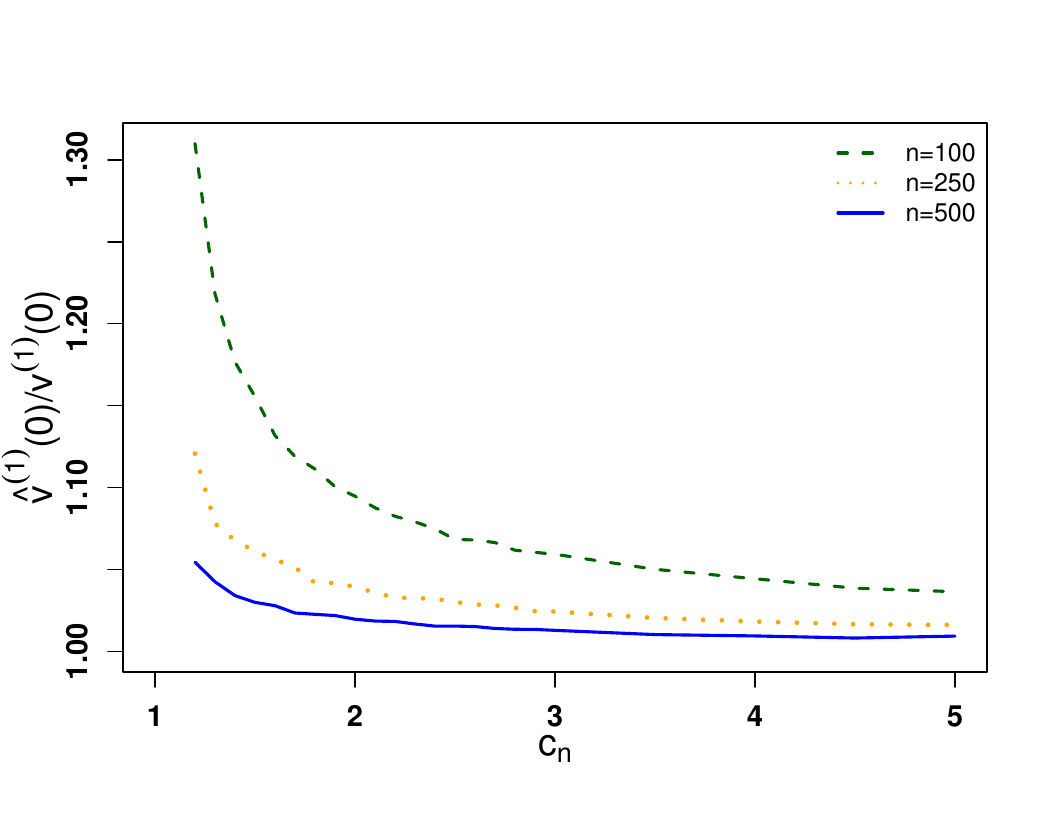}\\[-0.7cm]
\end{tabular}
 \caption{Finite-sample performance of the estimators for $v(0)$ and $v^{(1)}(0)$ when $c_n \in (1,5]$, $n\in\{100,250,500\}$, and the elements of $\bX_n$ drawn from the normal (first column) and scale $t$-distribution (second column).}
\label{fig:MP}
 \end{figure}

Next, we present the general results of Theorem \ref{th1} for $\bSigma=\bI_p$.

\begin{corollary}\label{cor2}
Let $\bY_n$ fulfill the stochastic representation \eqref{obs} with $\bSigma=\bI_p$. Then, under Assumptions \textbf{(A2)}-\textbf{(A3)} it holds that
\begin{equation}\label{cor2-eq1}
\left|\emph{tr}( (\bS_n^+)^{m}\bTheta)- \frac{(-1)^{m-1}v^{(m-1)}(0)}{(m-1)! c_n} \emph{tr}(\bTheta)\right| \stackrel{a.s.}{\rightarrow} 0
\end{equation}
for $p/n \rightarrow c \in (1,\infty)$ as $n \rightarrow \infty$ with
{\small\begin{equation}\label{cor2-eq2}
v(0)=\frac{1}{c_n-1}, \quad
v^{(1)}(0)=\frac{-c_n}{(c_n-1)^3}~~\quad\text{and}
\end{equation}
\begin{equation}\label{cor2-eq3}
v^{(m)}(0)= \frac{c_n}{(c_n-1)^3}\sum_{k=2}^m (-1)^{k} k!(c_n-1)^{(k+1)}\left(1-c_n^{-k}\right) B_{m,k}\left(v^{(1)}(0),...,v^{(m-k+1)}(0)\right).
\end{equation}
}
\end{corollary}

As a special case of Corollary \ref{cor2}, we get
\begin{eqnarray*}
&&\left|\text{tr}( \bS_n^+\bTheta)-\frac{1}{(c_n-1)c_n} \text{tr}\left(\bTheta\right)\right| \stackrel{a.s.}{\rightarrow} 0,\qquad\left|\text{tr}( (\bS_n^+)^2\bTheta)-\frac{1}{(c_n-1)^3}\text{tr}\left(\bTheta\right)\right| \stackrel{a.s.}{\rightarrow} 0,\\
&&\left|\text{tr}( (\bS_n^+)^3\bTheta)-\frac{c_n+1}{(c_n-1)^5}\text{tr}\left(\bTheta\right)\right| \stackrel{a.s.}{\rightarrow} 0,\qquad\left|\text{tr}( (\bS_n^+)^4\bTheta)-\frac{c_n^2+3c_n+1}{(c_n-1)^7}\text{tr}\left(\bTheta\right)\right| \stackrel{a.s.}{\rightarrow} 0,
\end{eqnarray*}
for $p/n \rightarrow c \in (1,\infty)$ as $n \rightarrow \infty$.

The results of Corollary \ref{cor2} extend the previous findings of \cite{cook2011mean} and \cite{imori2020mean}. Under the additional assumption that the elements of the matrix $\bY_n$ are normally distributed, we get that $(n-1)\bS_n$ has a $p$-dimensional singular Wishart distribution with $n-1$ degrees of freedom and identity covariance matrix (see, e.g., \cite{srivastava2003singular}, \cite{bodnar2008properties}). \cite{cook2011mean} derived the exact expression of the mean matrix of the singular Wishart distribution under the assumption that the population covariance matrix is proportional to the identity matrix. They proved that
\[\mathbb{E}\left(
\bS_n^+\right)
=(n-1)\mathbb{E}\left(
((n-1)\bS_n)^+\right)
=\frac{(n-1)^2}{p(p-n)}\bI_p.
\]
Furthermore, it holds that $\dfrac{(n-1)^2}{p(p-n)} \to \dfrac{1}{c(c-1)}$ for $p/n \to c  \in (1,\infty)$ as $n \to \infty$. On the other side, the application of Corollary \ref{cor2} leads to the same result when one defines $\bTheta=\mathbf{e}_i\mathbf{e}_j^\top$, $i,j=1,...,p$, where $\mathbf{e}_i$ is the vector of zeros except the $i$-th element which is one. 

\cite{imori2020mean} extended the results of \cite{cook2011mean} by considering a singular Wishart distribution with an arbitrary covariance matrix and deriving the lower and upper limits of the mean matrix and covariance matrix. The findings of Corollary \ref{cor2} provide further generalization by establishing the limiting behaviour of the elements of the Moore-Penrose inverse of the sample covariance matrix. Moreover, the results of Theorem \ref{th1}, Corollary \ref{cor1}, Corollary \ref{cor0}, and Corollary \ref{cor2} are derived in the general case without the assumption of normality imposed on the elements of $\bY_n$.

\subsection{Weighted moments of the sample ridge-type inverse}\label{sec:main-ridge}

The proof of Theorem \ref{th1} can be adjusted to derive the weighted trace moments of the ridge-type inverse of the sample covariance matrix, expressed as 
\begin{equation}\label{S_n-ridge}
\bS_n^{-}(t)=(\bS_n+t\bI_p)^{-1}  \quad \text{for} \quad t > 0.
\end{equation}
The findings are summarized in Theorem \ref{th2}. In contrast to the statement of Theorem \ref{th1}, the results of Theorem \ref{th2} are derived for $c \in (0,+\infty)$.

\begin{theorem}\label{th2}
Let $\bY_n$ fulfill the stochastic representation \eqref{obs}. Then, under Assumptions \textbf{(A1)}-\textbf{(A3)} for any $t>0$ and $m=0, 1, 2, \ldots$, it holds that  
\begin{equation}\label{th2_eq1}
\left|\emph{tr}((\bS_n^{-}(t))^{m+1}\bTheta)-\tilde{s}_{m+1}\left(t,\bTheta\right)\right| \stackrel{a.s.}{\rightarrow} 0\quad\text{for} \quad p/n \rightarrow c \in (0, +\infty) 
\quad \text{as} \quad n \rightarrow \infty,
\end{equation}
where 
\begin{eqnarray}\label{th2-tsm}
    \tilde{s}_{m+1}\left(t,\bTheta\right)&=&  
 \sum_{l=1}^m t^{-(m-l)-1} \sum_{k=1}^l \frac{(-1)^{l+k} k!}{l!}d_k\left(t,\bTheta\right)B_{l,k}\left(v^{(1)}(t),v^{(2)}(t),...,v^{(l-k+1)}(t)\right)\nonumber\\
&+&t^{-m-1} d_0\left(t,\bTheta\right) \qquad\text{with}
\end{eqnarray}
\begin{eqnarray}\label{th2-dk}
d_k\left(t,\bTheta\right)&=&\emph{tr}\left\{\left(v(t)\bSigma+\bI_p\right)^{-1}\left[\bSigma\left(v(t)\bSigma+\bI_p\right)^{-1}\right]^{k}\bTheta\right\}, \quad k=0,1,2,...,
\end{eqnarray}
$v(t)$ is the unique solution of the equation \eqref{th2-vt} and its derivatives satisfy
\begin{equation}\label{th2-vtpr}
v^{(1)}(t)=-\frac{1}{v(t)^{-2}-c_n  \frac{1}{p}\emph{tr}\left\{\left[\bSigma\left(v(t)\bSigma+\bI_p\right)^{-1}\right]^2\right\}}, 
\end{equation}
$v^{(2)}(t)$,...,$v^{(m)}(t)$ are computed recursively by
\begin{equation}\label{th2-vtpr-m}
v^{(m)}(t)=-v^{(1)}(t) \sum_{k=2}^m (-1)^{k} k!h_{k+1}(t)
B_{m,k}\left(v^{(1)}(t),...,v^{(m-k+1)}(t)\right),\quad\text{with}
\end{equation}
\begin{eqnarray}\label{th2-hk}
  h_k(t)=[v(t)]^{-k}-c_n\frac{1}{p}\emph{tr}\left\{\left[\bSigma\left(v(t)\bSigma+\bI_p\right)^{-1}\right]^{k}\right\}, \quad k=1,2,....
\end{eqnarray}
\end{theorem}
 
The almost sure convergence stated in Theorem \ref{th2} remains valid and, moreover, it is uniform on the larger domain $\mathbbm{C} \setminus \mathbbm{R}^{-}$. This extension follows from the Weierstrass convergence theorem \cite{ahlfors1953}. Alternatively, the same conclusion can be justified using arguments found in \cite{baisil2004}, \cite{rubio2011spectral}, \cite{rubio2012performance}, or \cite{bodnarokhrinparolya2023}, which rely on Vitali’s theorem concerning the uniform convergence of sequences of uniformly bounded holomorphic functions to a holomorphic limit (see \cite{rudin1987real}, \cite{hille2002analytic}).

Corollary \ref{cor3} provides the formulae for $\tilde{s}_{m}\left(t,\bTheta\right)$ for $m=0,1,2,3$.

\begin{corollary}\label{cor3}
Let $\bY_n$ fulfill the stochastic representation \eqref{obs}. Then, under Assumptions \textbf{(A1)}-\textbf{(A3)} for any $t>0$, we get \eqref{th1_eq1} with
\begin{eqnarray*}
\tilde{s}_1\left(t,\bTheta\right)&=& t^{-1} d_0(t,\bTheta), \quad
\tilde{s}_2\left(t,\bTheta\right)= t^{-2} d_0(t,\bTheta) + t^{-1} v^{(1)}(t) d_1(t,\bTheta),\\
\tilde{s}_3\left(t,\bTheta\right)&=& t^{-3} d_0(t,\bTheta) + t^{-2} v^{(1)}(t) d_1(t,\bTheta)+t^{-1}\left\{[v^{(1)}(t)]^2 d_2(t,\bTheta)-\frac{1}{2}v^{(2)}(t) d_1(t,\bTheta)\right\},\\
\tilde{s}_4\left(t,\bTheta\right)&=& t^{-4} d_0(t,\bTheta) + t^{-3} v^{(1)}(t) d_1(t,\bTheta)+t^{-2}\left\{[v^{(1)}(t)]^2 d_2(t,\bTheta)-\frac{1}{2}v^{(2)}(t) d_1(t,\bTheta)\right\}\\
&&+t^{-1}\left\{\frac{1}{6}v^{(3)}(t) d_1(t,\bTheta)-v^{(1)}(t)v^{(2)}(t) d_2(t,\bTheta)+[v^{(1)}(t)]^3 d_3(t,\bTheta)\right\},
\end{eqnarray*}
for $p/n \rightarrow c \in (0,\infty)$ as $n \rightarrow \infty$ where $d_k\left(t,\bTheta\right)$ and $h_k(t)$ are defined in \eqref{th2-dk} and \eqref{th2-hk}, $v(t)$ is the unique solution of \eqref{th2-vt} and
\begin{eqnarray*}
v^{(1)}(t)&=& -h_2(t)^{-1} , \; v^{(2)}(t)=-2[v^{(1)}(t)]^3 h_3(t)
,\; v^{(3)}(t)=-6[v^{(1)}(t)]^2 v^{(2)}(t) h_3(t) +6[v^{(1)}(t)]^4 h_4(t).
\end{eqnarray*}
\end{corollary}

It is interesting to note that $\tilde{s}_{m+1}\left(t,\bTheta\right)$ in \eqref{th2-tsm} can be recursively computed by
\begin{equation}\label{tsm-rec}
\tilde{s}_{m+1}\left(t,\bTheta\right)= \frac{1}{t}\tilde{s}_{m}\left(t,\bTheta\right)+\frac{1}{t} \sum_{k=1}^m \frac{(-1)^{m+k} k!}{m!}d_k\left(t,\bTheta\right)B_{m,k}\left(v^{(1)}(t),v^{(2)}(t),...,v^{(m-k+1)}(t)\right).
\end{equation}
Using \eqref{tsm-rec}, we present the results in a special case for $\bTheta=({1}/{p})\bI_p$ in Corollary \ref{cor0Ridge}.

\begin{corollary}\label{cor0Ridge}
Let $\bY_n$ fulfill the stochastic representation \eqref{obs}. Then, under Assumptions \textbf{(A1)}-\textbf{(A2)} for any $t>0$, it holds for $m=0,1,...$ that
\begin{equation*}
\left|\frac{1}{p}\emph{tr}\left[(\bS_n^{-}(t))^{m+1}\right]-t^{-(m+1)}\frac{c_n-1}{c_n}- \frac{(-1)^{m}v^{(m)}(t)}{m! c_n}\right| \stackrel{a.s.}{\rightarrow} 0\;\text{for} \; p/n \rightarrow c \in (0,\infty) 
\, \text{as} \, n \rightarrow \infty,
\end{equation*}
where
$v^{(m)}(t)$ is defined in \eqref{th2-vt}, \eqref{th2-vtpr}, and \eqref{th2-vtpr-m} for $m=0,1,...$ with $v^{(0)}(t)=v(t)$.
\end{corollary}

Several interesting results follow from the statement of Corollary \ref{cor0Ridge}. The sample trace moments $\text{tr}( (\bS_n^{-}(t))^{m+1})$ are not defined for $t=0$, while their limiting values can be well approximated by the expressions of the form $at^{-(m+1)} +b(t)$, where the constant $a$ depends on $c_n$ only. Moreover, equations \eqref{th2-vt}, \eqref{th2-vtpr}, and \eqref{th2-vtpr-m} in Theorem \ref{th2} which are used in the computations of $v^{(m)}(t)$, $m=0,1,...$, are well defined at $t=0$ and they coincide with the corresponding expressions \eqref{th1-v0}, \eqref{th1-v0pr}, and \eqref{th1-v0pr-m} derived in Theorem \ref{th1} for the Moore-Penrose inverse. 

Another important application of Corollary \ref{cor0Ridge} provides consistent estimators for $v^{(m)}(t)$, $m=0,1,...$, for $t>0$ despite these functions are not defined in the closed form in Theorem \ref{th2}. Their consistent estimators are given by formula \eqref{hv0-all1}, namely
\begin{equation*}
\hat{v}^{(m)}(t)= (-1)^{m} m! c_n \left(\frac{1}{p}\text{tr}\left[ (\bS_n^{-}(t))^{m+1}\right]
-t^{-(m+1)} \frac{c_n-1}{c_n}\right).
\end{equation*}
 
Figure \ref{fig:Ridge} depicts the estimators of $v(t)$ and $v^{(1)}(t)$, normalized by the corresponding true values for $t \in [0.1,5]$. Two data-generating models are considered, which are the normal distribution and the scaled $t$-distribution with five degrees of freedom. The covariance matrix $\bSigma$ is set following the procedure described in Section \ref{sec:sim}. In the figure, we observe that $\hat{v}(t)$ and $\hat{v}^{(1)}(t)$ converge to $v(t)$ and $v^{(1)}(t)$, respectively, already for $n=100$. The results are similar for all considered values of the concentration ratio $c_n$ and the distributional models used to generate $\bX_n$. It is interesting that for small values of $t$, the convergence is slower. Also, the convergence is faster when data are generated from the normal distribution.

\vspace{-1cm}
\begin{figure}[h!t]
\centering
\begin{tabular}{cc}
\hspace{-0.5cm}\includegraphics[width=7cm]{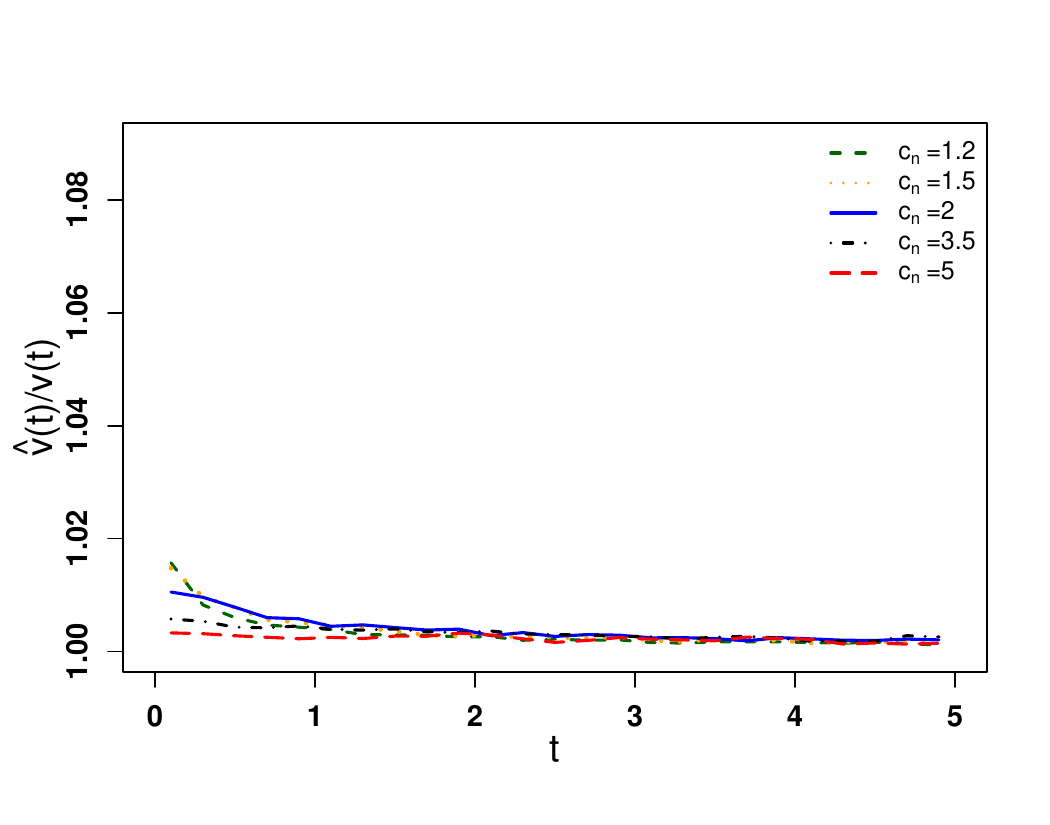}&
\hspace{-0.5cm}\includegraphics[width=7cm]{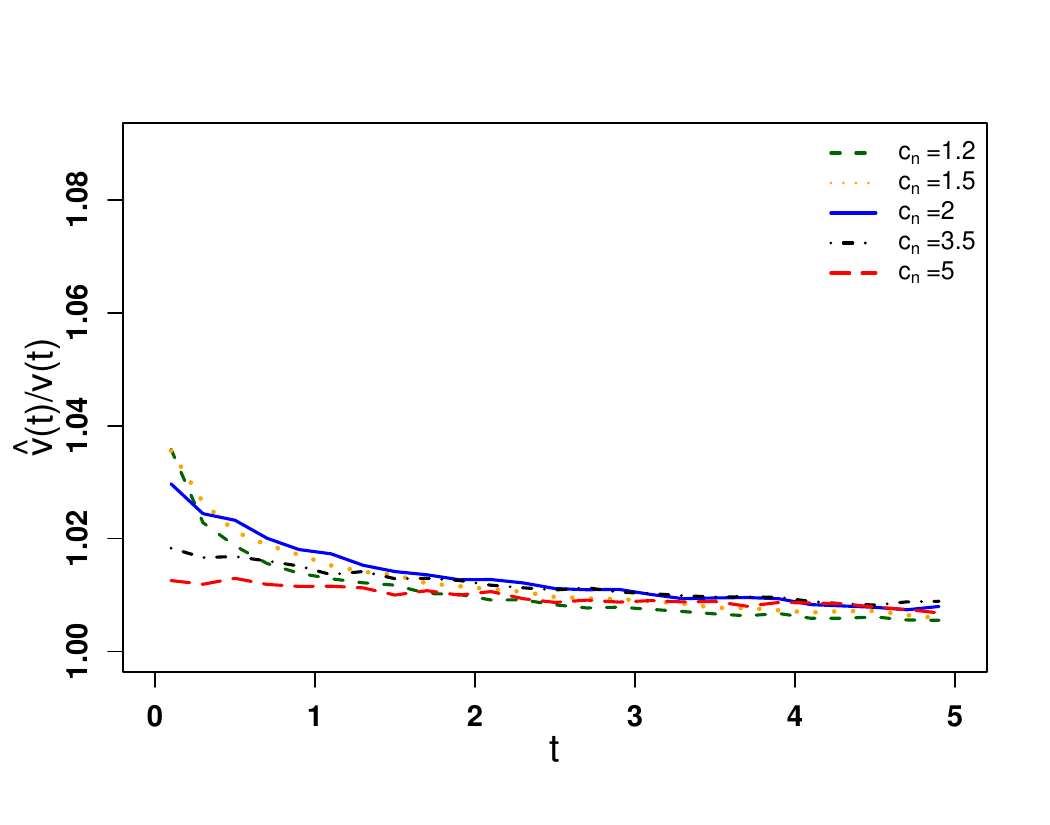}\\[-1.1cm]
\hspace{-0.5cm}\includegraphics[width=7cm]{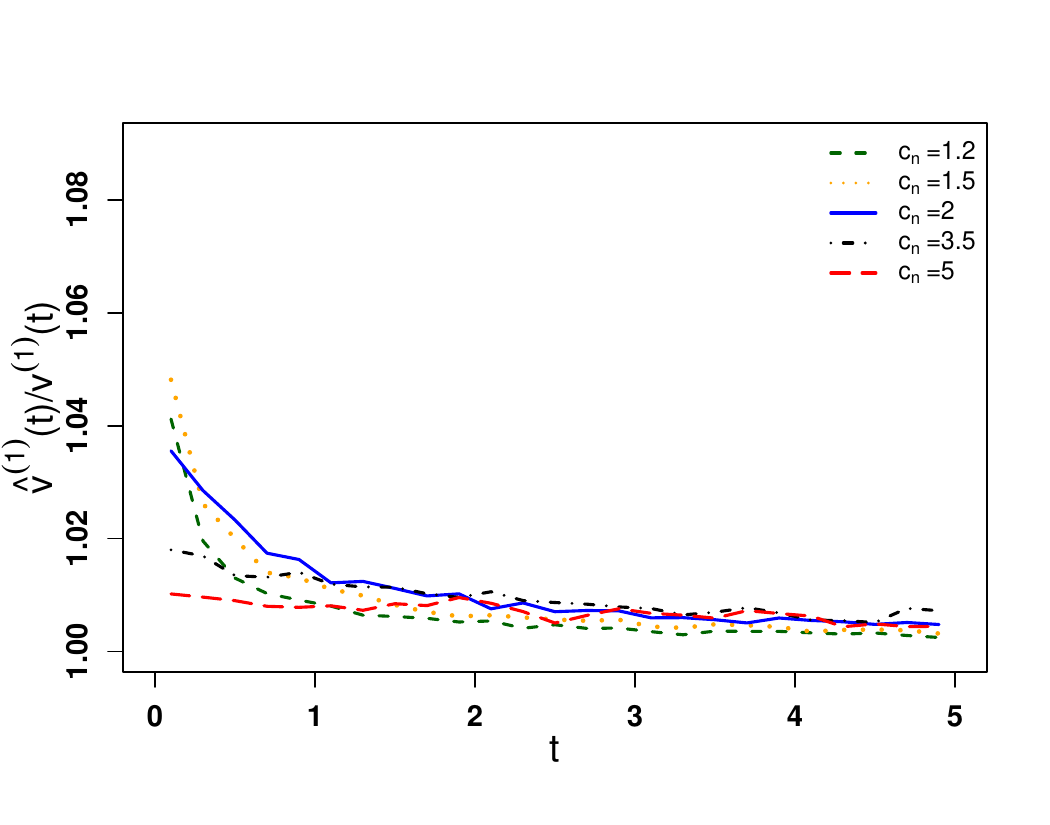}&
\hspace{-0.5cm}\includegraphics[width=7cm]{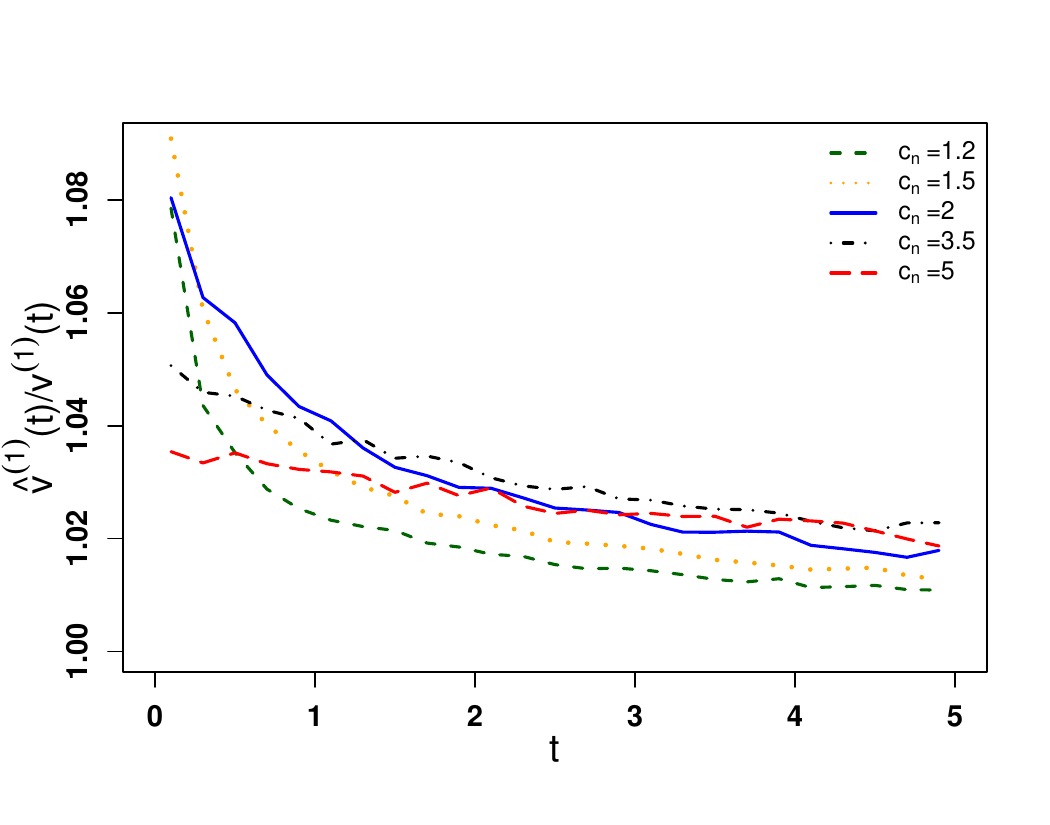}\\[-0.7cm]
\end{tabular}
 \caption{Finite-sample performance of the estimators for $v(t)$ and $v^{(1)}(t)$ when $t \in [0.1,5]$, $c_n \in \{1.2,1.5,2,3.5,5\}$, $n=100$, and the elements of $\bX_n$ drawn from the normal (first column) and scale $t$-distribution (second column).}
\label{fig:Ridge}
 \end{figure}

Corollary \ref{cor3a} presents the findings under the additional assumption that the true population covariance matrix is the identity matrix.

\begin{corollary}\label{cor3a}
Let $\bY_n$ fulfill the stochastic representation \eqref{obs} with $\bSigma=\bI_p$. Then, under Assumptions \textbf{(A2)}-\textbf{(A3)} for any $t>0$, it holds that
\begin{equation}\label{cor3a-eq1}
\left|\emph{tr}((\bS_n^{-}(t))^{m+1}\bTheta)-\left( t^{-(m+1)}\frac{c_n-1}{c_n}+ \frac{(-1)^{m}v^{(m)}(t)}{m! c_n} \right)\emph{tr}(\bTheta)\right| \stackrel{a.s.}{\rightarrow} 0
\end{equation}
for $p/n \rightarrow c \in (0, +\infty)$ as $n \rightarrow \infty$ with
{\scriptsize
\begin{eqnarray}
&v(t)=\frac{2}{(c_n-1+t)+\sqrt{(c_n-1+t)^2+4t}}, \quad
v^{(1)}(t)=\frac{-1}{v(t)^{-2}-c_n(v(t)+1)^{-2}},\label{cor3a-eq2}\\
&~~v^{(m)}(t)= -v^{(1)}(t) \sum_{k=2}^m (-1)^{k} k!\left(v(t)^{-(k+1)}-c_n(v(t)+1)^{-(k+1)}\right) B_{m,k}\left(v^{(1)}(t),...,v^{(m-k+1)}(t)\right).\label{cor3a-eq3}    
\end{eqnarray}
}
\end{corollary}

\subsection{Weighted moments of the sample Moore-Penrose-ridge inverse}\label{sec:main-MP-ridge}

Taking into account the properties of the Moore-Penrose and ridge-type inverses, one can take the advantage of both estimators and build a new inverse for any $t>0$ by
\begin{eqnarray}
   \bS^{\pm}_n(t)=(\bS_n+t\bI_p)^{-1}\bS_n(\bS_n+t\bI_p)^{-1}\,.
\end{eqnarray}
This matrix inherits the properties of the Moore-Penrose inverse for $t\to 0$ because by the spectral decomposition, we have 
\begin{eqnarray*}
\lim\limits_{t\to0}\bS^{\pm}_n(t)&=&\lim\limits_{t\to0}(\bS_n+t\bI_p)^{-1}\bS_n(\bS_n+t\bI_p)^{-1}=\lim\limits_{t\to0}\bU\text{diag}\left\{\frac{d_1}{(d_1+t)^2},\ldots,\frac{d_p}{(d_p+t)^2}\right\}\bU^\top=\bS_n^+,
\end{eqnarray*}
where $\bS_n = \bU \, \text{diag}\left\{d_1, \ldots, d_p\right\} \bU^\top$ is the spectral decomposition of $\bS_n$, with $d_1, d_2, \ldots, d_p$ denoting the eigenvalues of $\bS_n$, and $\bU$ the corresponding matrix of eigenvectors.

On the other side, it is related to the ridge inverse since
\begin{eqnarray}\label{Sn-pm-t}
    \bS^{\pm}_n(t)&=&
    (\bS_n+t\bI_p)^{-1}-t(\bS_n+t\bI_p)^{-2}.
\end{eqnarray}
In fact, we have already all of the instruments to find the asymptotic behavior for the weighted sample trace moments of $\bS^{\pm}(t)$. Indeed, using binomial theorem we get
\begin{eqnarray}
   (\bS^{\pm}_n(t))^m&=&\left((\bS_n+t\bI_p)^{-1}-t(\bS_n+t\bI_p)^{-2}\right)^m=\sum\limits_{k=0}^m \binom{m}{k}(\bS_n+t\bI_p)^{-(m-k)}(-1)^{k}t^k(\bS_n+t\bI_p)^{-2k}\nonumber\\
    && \hspace{1cm}=\sum\limits_{k=0}^m (-1)^{k}t^k\binom{m}{k}(\bS_n+t\bI_p)^{-(m+k)}\,.
\end{eqnarray}
Thus, in order to find the limit of $\text{tr}\left( (\bS^{\pm}_n(t))^m\bTheta\right)$ we only need to know the limits of the functionals $\text{tr}\left((\bS_n+t\bI_p)^{-(m+k)}\bTheta\right)$, which have already been derived in Theorem \ref{th2}. As a result, we get the following theorem.

\begin{theorem}\label{th3MPR}
    Let $\bY_n$ fulfill the stochastic representation \eqref{obs}. Then, under Assumptions \textbf{(A1)}-\textbf{(A3)} for any $t>0$ and $m=1,2,\ldots,$ it holds that
\begin{equation*}
\left|\emph{tr}((\bS_n^{\pm}(t))^{m}\bTheta)-\grave{s}_{m}\left(t,\bTheta\right)\right| \stackrel{a.s.}{\rightarrow} 0\quad\text{for} \quad p/n \rightarrow c \in (0, +\infty) 
\quad \text{as} \quad n \rightarrow \infty\quad\text{with}
\end{equation*}
\begin{eqnarray}\label{th2-gsm}
\grave{s}_{m}\left(t,\bTheta\right)&=&  
 \sum\limits_{k=0}^m (-1)^{k}t^k\binom{m}{k} \tilde{s}_{m+k}(t,\bTheta)\,,
\end{eqnarray}
where $\tilde{s}_{m+k}(t,\bTheta)$ can be computed using \eqref{th2-tsm} for $k=0, 1, \ldots, m$ and $m=1,2,\ldots$.
\end{theorem}

In Corollary \ref{cor-MPRidge2} we present another expression of $\grave{s}_{m}(t,\bTheta)$ which follows from the recursive computation of $\tilde{s}_{m+k}(t,\bTheta)$ as given in \eqref{tsm-rec}. The proof of Corollary \ref{cor-MPRidge2} is given in the supplement.

\begin{corollary}\label{cor-MPRidge2}
    Let $\bY_n$ fulfill the stochastic representation \eqref{obs}. Then, under Assumptions \textbf{(A1)}-\textbf{(A3)} for any $t>0$, it holds that
\begin{eqnarray}\label{cor-gsm_eq1}
\grave{s}_{m}\left(t,\bTheta\right)&=& -D_m(t,\bTheta) + 
 \sum\limits_{k=2}^m (-1)^{k}\binom{m}{k}
 \sum\limits_{j=1}^{k-1} t^j D_{m+j}(t,\bTheta)\,, ~\text{where}
\end{eqnarray}
\begin{equation}\label{cor-gsm_eq2}
D_m(t,\bTheta)=\sum_{k=1}^m \frac{(-1)^{m+k} k!}{m!}d_k\left(t,\bTheta\right)B_{m,k}\left(v^{(1)}(t),v^{(2)}(t),...,v^{(m-k+1)}(t)\right).
\end{equation}
\end{corollary}

Interestingly, the summands which are unbounded for $\tilde{s}_m(t,\bTheta)$ as $t\to 0$ are all canceled in the limiting behavior of for $\grave{s}_m(t,\bTheta)$. Since all derivatives of $v(t)$ and the functions $d_k(t,\bTheta)$ for $k\ge 1$ are all well-defined and bounded in $t=0$, it holds that the limits of $\grave{s}_m(t,\bTheta)$ are also well defined at zero and, moreover,
\begin{equation}\label{tto0}
\grave{s}_m(0,\bTheta)=-D_m(0,\bTheta)=s_m(\bTheta).    
\end{equation}
Hence, the limiting behavior of $\grave{s}_m(0,\bTheta)$ coincides with the one, obtained in the case of the Moore-Penrose inverse.

In Corollary \ref{cor3MPR}, the results for $m=1$ and $m=2$ are summarized.

\begin{corollary}\label{cor3MPR}
Let $\bY_n$ fulfill the stochastic representation \eqref{obs}. Then, under Assumptions \textbf{(A1)}-\textbf{(A3)} for any $t>0$ we get \eqref{th1_eq1} with
\begin{eqnarray}\label{grave_s1}
\grave{s}_1\left(t,\bTheta\right)&=& -v^{(1)}(t)d_1\left(t,\bTheta\right),\\
\grave{s}_2\left(t,\bTheta\right)&=& -\left\{[v^{(1)}(t)]^2 d_2(t,\bTheta)-\frac{1}{2}v^{(2)}(t) d_1(t,\bTheta)\right\} \nonumber\\
&+&t\left\{\frac{1}{6}v^{(3)}(t) d_1(t,\bTheta)-v^{(1)}(t)v^{(2)}(t) d_2(t,\bTheta)+[v^{(1)}(t)]^3 d_3(t,\bTheta)\right\},\label{grave_s2}
\end{eqnarray}
for $p/n \rightarrow c \in (0,+\infty)$ as $n \rightarrow \infty$ where $d_k\left(t,\bTheta\right)$ and $h_k(t)$ are defined in \eqref{th2-dk} and \eqref{th2-hk}, and $v^{(m)}(t)$, $m=0,1,...$, are defined in\eqref{th2-vt}, \eqref{th2-vtpr}, and \eqref{th2-vtpr-m} with $v^{(0)}(t)=v(t)$.
\end{corollary}

Next, we present the results for two other important special cases $\bTheta=(1/p)\bI_p$ and $\bSigma=\bI_p$ in Corollary \ref{cor0MPRidge} and Corollary \ref{cor3aMPR}, respectively. The proof of Corollary \ref{cor0MPRidge} is given in the supplement.

\begin{corollary}\label{cor0MPRidge}
    Let $\bY_n$ fulfill the stochastic representation \eqref{obs}. Then, under Assumptions \textbf{(A1)}-\textbf{(A2)} for any $t>0$, it holds for $m=1,2,...$ that
\begin{equation*}
\left|\frac{1}{p}\emph{tr}\left[(\bS_n^{\pm}(t))^{m}\right]-\frac{(-1)^{m-1}}{c_n}\sum\limits_{k=0}^m \binom{m}{k}\frac{v^{(m+k-1)}(t)}{(m+k-1)!}t^k\right| \stackrel{a.s.}{\rightarrow} 0\;\text{for} \; p/n \rightarrow c \in (0, +\infty) 
\; \text{as} \; n \rightarrow \infty,
\end{equation*}
where $v^{(m)}(t)$ is defined in \eqref{th2-vt}, \eqref{th2-vtpr}, and \eqref{th2-vtpr-m} for $m=0,1,...$ with $v^{(0)}(t)=v(t)$.
\end{corollary}

\begin{corollary}\label{cor3aMPR}
Let $\bY_n$ fulfill the stochastic representation \eqref{obs} with $\bSigma=\bI_p$. 
Then, under Assumptions \textbf{(A2)}-\textbf{(A3)} for any $t>0$, it holds that
\begin{equation}\label{cor3a-eq1MPR}
\left|\emph{tr}((\bS_n^{\pm}(t))^{m+1}\bTheta)-\frac{(-1)^{m-1}}{c_n}\sum\limits_{k=0}^m \binom{m}{k}\frac{v^{(m+k-1)}(t)}{(m+k-1)!}t^k\emph{tr}(\bTheta)\right| \stackrel{a.s.}{\rightarrow} 0
\end{equation}
for $p/n \rightarrow c \in (0, +\infty)$ as $n \rightarrow \infty$ with $v(t)$ and $v^{(m)}(t)$ defined in \eqref{cor3a-eq2} and \eqref{cor3a-eq3}.
\end{corollary}

The results of Section \ref{sec:main} show that the three generalized inverses have different limiting behavior. This is easily visualized with the findings presented in Corollary \ref{cor2}, Corollary \ref{cor3a}, and Corollary \ref{cor3aMPR} when $\bSigma=\bI_p$. Namely, it holds that
\begin{eqnarray*}
&&\left|\text{tr}( \bS_n^+\bTheta)-\frac{1}{(c_n-1)c_n} \text{tr}\left(\bTheta\right)\right| \stackrel{a.s.}{\rightarrow} 0, \\
&&\left|\text{tr}(\bS_n^-(t)\bTheta)-\left(\frac{v(t)}{c_n}+\frac{1}{t}\frac{c_n-1}{c_n}\right)\text{tr}\left(\bTheta\right)\right| \stackrel{a.s.}{\rightarrow} 0,\\
&&\left|\text{tr}(\bS_n^\pm(t)\bTheta)-\left(\frac{v(t)}{c_n}+t\frac{v^{(1)}(t)}{c_n}\right)\text{tr}\left(\bTheta\right)\right| \stackrel{a.s.}{\rightarrow} 0
\end{eqnarray*}
for $p/n \rightarrow c \in (1, +\infty)$ as $n \rightarrow \infty$ where
\[v(t)=\frac{2}{(c_n-1+t)+\sqrt{(c_n-1+t)^2+4t}}, \quad
v^{(1)}(t)=\frac{-1}{v(t)^{-2}-c_n(v(t)+1)^{-2}}.\]
Interestingly, all three generalized inverses consistently estimate non-diagonal elements of the precision matrix elementwise, which follows directly by setting
\begin{equation}\label{bTheta-cov}
\bTheta=\frac{1}{2}(\mathbf{e}_i\mathbf{e}_j^\top+\mathbf{e}_j\mathbf{e}_i^\top)
=\frac{1}{4}\left((\mathbf{e}_i+\mathbf{e}_j)(\mathbf{e}_i+\mathbf{e}_j)^\top
-(\mathbf{e}_i-\mathbf{e}_j)(\mathbf{e}_i-\mathbf{e}_j)^\top\right),
\end{equation}
where $1\le i<j\le p$ and $\mathbf{e}_i$ is a unit vector with one in the i-th position. 

Another notable observation concerns the case $t \to 0$. When $p>n$, then the application of \eqref{cor2-eq2} shows that the Moore-Penrose inverse and the Moore-Penrose-ridge inverse exhibit the same limiting behavior, while the ridge-type inverse does not. This is not surprising, as it follows directly from their definitions: $\lim\limits_{t \to 0} \bS^{\pm}_n(t) = \bS^{+}_n$, while $\lim\limits_{t \to 0} \bS^{-}_n(t) = \infty$. The last two equalities also hold for $\bSigma \neq \bI_p$ following \eqref{tto0}. Thus, the Moore-Penrose-ridge estimator $\bS^{\pm}_n(t)$ effectively approximates the Moore-Penrose inverse for small values of $t$. This feature is particularly useful when only mild regularization is needed. For instance, when the concentration ratio $c_n = p/n$ is close to one, a regime in which the Moore-Penrose inverse $\bS^+_n$ is known to be numerically unstable. This insight supports the practical relevance of the Moore-Penrose-ridge inverse $\bS^{\pm}_n$ in scenarios where $p$ and $n$ are nearly equal.

For $p<n$, the sample covariance matrix $\bS_n$ is nonsingular. In this case, the application of Theorem S.2.1 
and Corollary S.2.4 
from the supplement (\cite{BP2025reviving-S}), and \eqref{Sn-pm-t} yields for $\bSigma=\bI_p$ that
\begin{eqnarray*}
&&\left|\text{tr}( \bS_n^{-1}\bTheta)-\frac{1}{1-c_n} \text{tr}\left(\bTheta\right)\right| \stackrel{a.s.}{\rightarrow} 0, \quad\left|\text{tr}(\bS_n^-(t)\bTheta)-\frac{1}{t+w(t)}\text{tr}\left(\bTheta\right)\right| \stackrel{a.s.}{\rightarrow} 0,\\
&&\left|\text{tr}(\bS_n^\pm(t)\bTheta)-\left(\frac{1}{t+w(t)}-t\frac{w^{(1)}(t)+1}{(t+w(t))^2}\right)\text{tr}\left(\bTheta\right)\right| \stackrel{a.s.}{\rightarrow} 0
\end{eqnarray*}
for $p/n \rightarrow c \in (0, 1)$ as $n \rightarrow \infty$ where
\[w(t)=\frac{(1-c_n-t)+\sqrt{(1-c_n-t)^2+4t}}{2}, \quad
w^{(1)}(t)=\frac{c_nw(t)}{(t+w(t))^{2}+tc_n},\]
with $w(0)=1-c_n$ and $w^{(1)}(0)=c_n/(1-c_n)$. Consequently, in contrast to the case $p>n$, all three inverses exhibit the same limiting behavior. Furthermore, the application of \eqref{bTheta-cov} yields that all three inverses consistently estimate non-diagonal elements of the precision matrix elementwise, similarly to the case $p>n$. Finally, the comparison of formulae after Corollary \ref{cor2} and at the end of Section S.2 
in the supplement (\cite{BP2025reviving-S}) leads to the conclusion that $\text{tr}((\mathbf{S}^{+}_n)^m\mathbf{\Theta})$ and $\text{tr}((\mathbf{S}^{-1}_n)^m\mathbf{\Theta})$ have the same asymptotic behavior up to the sign of $c_n-1$ when $m=2,3,4$. 


\section{Applications}\label{sec:app}
The derived theoretical results have a lot of applications in high-dimensional statistics. Maybe, the most obvious one would be the approximation of the weighted functionals $\tr\left(f(\bS_n^{\#}(t))\bTheta\right)$ by Taylor expansion up to a specific order $m$, assuming that the test function $f$ is smooth enough and $\bS_n^{\#}(t)$ denotes an estimator of the precision matrix.  The most prominent examples, however, are the construction of fully data-driven shrinkage estimators for the precision matrix  and for the weights of the global minimum variance portfolio in the case of a singular sample covariance matrix. Later on, we choose $\bS_n^{\#}(t)$ to be one of the three generalized inverses considered in Section \ref{sec:main}, i.e., $\bS_n^{\#}(t) \in 
\{\bS_n^+,\bS_n^-(t),\bS_n^\pm(t)\}$.

\subsection{Shrinkage estimator of the high-dimensional precision matrix}\label{sec:shrink-prec}
Following \cite{BodnarGuptaParolya2016} the general linear shrinkage estimator for the precision matrix is given by
 \begin{equation}\label{gse}
\widehat{\boldsymbol{\Pi}}_{GSE}=\alpha_n\bS_n^{\#}(t)+\beta_n\boldsymbol{\Pi}_0\,,
\end{equation}
where $\boldsymbol{\Pi}_0$ is a target matrix, which typically is chosen to be equal to the identity matrix when no \emph{a priori} information about the precision matrix $\bSigma^{-1}$ is available. Potentially, the estimator \eqref{gse} is nonlinear since all the three pseudo-inverses of $\bS_n$ are nonlinear functions of $\bSigma$ and also they are nonlinear functions of the eigenvalues of $\bS_n$. For example, the Moore-Penrose inverse inverts the non-zero eigenvalues and keeps zero eigenvalues untouched. A similar structure as \eqref{gse} obeys the non-linear quadratic shrinkage estimator of \cite{lwQIS2020} with specific $\alpha_n$, $\beta_n$ and the data-driven equivariant (with same eigenvectors as $\bS_n$) target $\boldsymbol{\Pi}_0$, which come from the nonlinear Ledoit-P\'ech\'e shrinkage formula \cite[see][]{ledoitpeche2011}. A proper shrinkage of $\bS_n^{\#}(t)$ could result in a better estimator of $\bSigma^{-1}$. Moreover, it is interesting how it compares to the well-established nonlinear shrinkage technique. The choice of $\boldsymbol{\Pi}_0$ is free to some extent but we will concentrate us on $\boldsymbol{\Pi}_0=\bI_p$ in simulations to have a fair comparison with common rotation-equivariant estimators \cite[see, e.g.,][]{lw2004, lw12, bodnar2014strong}.

The optimal shrinkage intensities $\alpha^{*}_n$ and $\beta^{*}_n$ are determined by minimizing the Frobenius (quadratic) loss\footnote{This loss function is a bit different than given in \cite{BodnarGuptaParolya2016}. The reason for that is the discussion in Section 5 of \cite{lwQIS2020} about the singular case, i.e., $\bSigma^{-1}$ should be avoided in the loss function because of potential numerical instabilities.} for a given nonrandom target matrix $\boldsymbol{\Pi}_0$ given by \cite[see][]{haff1979, KrishGupta1989, YangBerger1994, wang2015shrinkage, BodnarGuptaParolya2016} 

\begin{equation}\label{risk}
 L^2_{F;n}=||\widehat{\boldsymbol{\Pi}}_{GSE}\bSigma-\bI_p||_F^2\,.
\end{equation}
The solution is summarized in Theorem \ref{lem_prec_loss}.

\begin{theorem}\label{lem_prec_loss}
For any fixed value of $t>0$, the loss function $L^2_{F;n}$ is minimized at
{\small
\begin{eqnarray}\label{beta-gen}
&&\alpha_n^*(\bS_n^{\#}(t))=
\dfrac{\emph{tr}(\bS_n^{\#}(t)\bSigma)-\emph{tr}\left(\dfrac{\bSigma\boldsymbol{\Pi}_0}{||\bSigma\boldsymbol{\Pi}_0||_F}\right)  \emph{tr}\left(\bS_n^{\#}(t)\dfrac{{\bSigma^2}\boldsymbol{\Pi}_0}{||\bSigma\boldsymbol{\Pi}_0||_F}\right)}{||\bS_n^{\#}(t)\bSigma||^2_F
  -\Bigl[\emph{tr}\left(\bS_n^{\#}(t)\dfrac{{\bSigma^2}\boldsymbol{\Pi}_0}{||\bSigma\boldsymbol{\Pi}_0||_F}\right)\Bigr]^2}
 \end{eqnarray}
and
\begin{eqnarray}\label{alfa-gen} 
&& \beta_n^*(\bS_n^{\#}(t))
=\dfrac{\emph{tr}\left(\dfrac{\bSigma\boldsymbol{\Pi}_0}{||\bSigma\boldsymbol{\Pi}_0||_F}\right)||\bS_n^{\#}(t)\bSigma||^2_F
 -\emph{tr}(\bS_n^{\#}(t)\bSigma)\emph{tr}\left(\bS_n^{\#}(t)\dfrac{\bSigma^2\boldsymbol{\Pi}_0}{||\bSigma\boldsymbol{\Pi}_0||_F}\right)}{||\bS_n^{\#}(t)\bSigma||^2_F
 -\Bigl[\emph{tr}\left(\bS_n^{\#}(t)\dfrac{\bSigma^2\boldsymbol{\Pi}_0}{||\bSigma\boldsymbol{\Pi}_0||_F}\right)\Bigr]^2}||\bSigma\boldsymbol{\Pi}_0||^{-1}_F\,.
 \end{eqnarray}
}
Additionally, when $\bS_n^{\#}(t) \in \{\bS_n^{-}(t),\bS_n^{\pm}(t)\}$ the optimal value of $t$ is found by maximizing
\begin{eqnarray}\label{risk-L-F2}
L^2_{F;n,2}(\bS_n^{\#}(t))
&=& \frac{\left[\emph{tr}(\bS_n^{\#}(t)\bSigma)
- \emph{tr}\left(\bS_n^{\#}(t)\dfrac{{\bSigma^2}\boldsymbol{\Pi}_0}{||\bSigma\boldsymbol{\Pi}_0||_F}\right)\emph{tr}\left(\dfrac{\bSigma\boldsymbol{\Pi}_0}{||\bSigma\boldsymbol{\Pi}_0||_F}\right)\right]^2
}{||\bS_n^{\#}(t)\bSigma||^2_F - \Bigl[\emph{tr}\left(\bS_n^{\#}(t)\dfrac{{\bSigma^2}\boldsymbol{\Pi}_0}{||\bSigma\boldsymbol{\Pi}_0||_F}\right)\Bigr]^2} .
\end{eqnarray}
\end{theorem}

The proof of Theorem \ref{lem_prec_loss} is given in the supplement. It is important to note that after dividing by $p$ the numerator and denominator in the formulas for $\alpha^*_{n}(\bS_n^{\#}(t))$, $\beta^*_{n}(\bS_n^{\#}(t))$ and $L^2_{F;n,2}(\bS_n^{\#}(t))$, no additional assumption on $\boldsymbol{\Pi}_0$ is needed since all ratios in \eqref{beta-gen}, \eqref{alfa-gen} and \eqref{risk-L-F2} are self-normalized. For example, by taking $$\bTheta=\dfrac{\bSigma\boldsymbol{\Pi}_0}{\sqrt{p}||\bSigma\boldsymbol{\Pi}_0||_F}$$ in Corollary \ref{cor1} and using Jensen's inequality, we get that $\tr\left(\bTheta\right)\leq 1$. Similarly, for $\bTheta=\frac{1}{p}\bSigma^2$ it holds that
$ \frac{1}{p}\tr(\bSigma^2)\leq \frac{p \lambda^2_{\max}(\bSigma)}{p}= \lambda^2_{\max}(\bSigma)<\infty.$
In the same way, it can be shown that  $\bTheta=\frac{1}{p}\bSigma
\quad \text{and}\quad \bTheta=\dfrac{\bSigma^2\boldsymbol{\Pi}_0}{\sqrt{p}||\bSigma\boldsymbol{\Pi}_0||_F}$ possess bounded trace norms. The latter properties verify Assumption \textbf{(A3)}.

Since $\alpha_n^*(\bS_n^{\#}(t))$, $\beta_n^*(\bS_n^{\#}(t))$ and $L^2_{F;n,2}(\bS_n^{\#}(t))$ depend on the unknown population covariance matrix $\bSigma$, they cannot be computed in practice and thus we refer to $\alpha_n^*(\bS_n^{\#}(t))$ and $\beta_n^*(\bS_n^{\#}(t))$ as the \emph{oracle} estimators. In the high-dimensional singular setting, i.e., $p>n$, \cite{BodnarGuptaParolya2016} found the almost sure asymptotic equivalents for $\alpha_n^*(\bS_n^+)$ and  $\beta_n^*(\bS_n^+)$ and estimate them consistently only in the case of $\bSigma=\sigma^2\bI_p$. The results of Section \ref{sec:main} allow us to establish the limiting behaviour of $\alpha_n^*(\bS_n^{\#}(t))$ and $\beta_n^*(\bS_n^{\#}(t))$, and to estimate consistently their limiting values $\alpha^*$ and  $\beta^*$ in a very general case. It is enough to find the exact limits of $\frac{1}{p}||\bS_n^{\#}\bSigma||^2_F
=\frac{1}{p}\tr\left([\bS_n^{\#}(t)]^2\bSigma^2\right)$, $\frac{1}{p}\text{tr}(\bS_n^{\#}(t)\bSigma)$ and $\frac{1}{p}\text{tr}\left(\bS_n^{\#}(t)\dfrac{{\bSigma^2}\boldsymbol{\Pi}_0}{||\bSigma\boldsymbol{\Pi}_0||_F}\right)$ by using Corollary \ref{cor1}, Corollary \ref{cor3} and Corollary \ref{cor3MPR},
and to estimate them consistently, together with some other quantities which depend on the population covariance matrix, not on its inverse. Thus, it is straightforward to determine the almost sure equivalents of the optimal shrinkage intensities. We present the results in the following subsections for all three pseudo-inverses.

\subsubsection{Results for the Moore-Penrose inverse}\label{sec:sh-prec-MP}

In the case of the Moore-Penrose inverse, the optimization problem and the resulting optimal shrinkage intensities do not depend on the tuning parameter $t$. This fact simplifies considerably the construction of the optimal shrinkage estimator. Let $\alpha_{MP;n}^*=\alpha_n^*(\bS_n^{+})$ and $\beta_{MP;n}^*=\beta_n^*(\bS_n^{+})$. The asymptotic deterministic equivalents to $\alpha_{MP;n}^*$ and $\beta_{MP;n}^*$ are derived in Theorem S.5.1 
from the supplementary material (see \cite{BP2025reviving-S}) by using the results presented in Corollary \ref{cor1}.

As a result, consistent estimators $\hat{\alpha}_{MP}^*$ and $\hat{\beta}_{MP}^*$ are obtained, which are summarized in Theorem \ref{th5}. Interestingly, all formulas are given in the closed form and, as such, no numerical computation is needed to construct consistent estimators for $\alpha_{MP}^*$ and $\beta_{MP}^*$. 

\begin{theorem}\label{th5} 
Let $\bY_n$ fulfill the stochastic representation \eqref{obs}. Then, under Assumptions \textbf{(A1)}-\textbf{(A2)}, consistent estimators for $\alpha_{MP}^*$ and $\beta_{MP}^*$ are given by
{\footnotesize
\begin{eqnarray}\label{halp_prec-opt}
  \hat{\alpha}_{MP}^*&=& \dfrac{\hat{d}_1\left(\frac{1}{p}\bSigma\right) \hat{q}_2\left(\frac{1}{p}\boldsymbol{\Pi}_0^2\right) -\hat{d}_1\left(0,\frac{1}{p}\bSigma^2\boldsymbol{\Pi}_0\right) \hat{q}_1\left(\frac{1}{p}\boldsymbol{\Pi}_0\right)}
{-\dfrac{1}{\hat{h}_2}\left(\hat{d}_2\left(\frac{1}{p}\bSigma^2\right)-\hat{d}_1\left(\frac{1}{p}\bSigma^2\right)\dfrac{\hat{h}_3}{\hat{h}_2}\right)\hat{q}_2\left(\frac{1}{p}\boldsymbol{\Pi}_0^2\right)-\dfrac{1}{\hat{h}_2}\hat{d}_1^2\left(0,\frac{1}{p}\bSigma^2\boldsymbol{\Pi}_0\right)},\\[0.3cm]
\hat{\beta}_{MP}^*&=&\!\!\dfrac{\left(\hat{d}_2\left(\frac{1}{p}\bSigma^2\right)-\hat{d}_1\left(\frac{1}{p}\bSigma^2\right)\dfrac{\hat{h}_3}{\hat{h}_2}\right)\hat{q}_1\left(\frac{1}{p}\boldsymbol{\Pi}_0\right)
+\hat{d}_1\left(\frac{1}{p}\bSigma\right)\hat{d}_1\left(0,\frac{1}{p}\bSigma^2\boldsymbol{\Pi}_0\right)}
{\left(\hat{d}_2\left(\frac{1}{p}\bSigma^2\right)-\hat{d}_1\left(\frac{1}{p}\bSigma^2\right)\dfrac{\hat{h}_3}{\hat{h}_2}\right)\hat{q}_2\left(\frac{1}{p}\boldsymbol{\Pi}_0^2\right)+\hat{d}_1^2\left(0,\frac{1}{p}\bSigma^2\boldsymbol{\Pi}_0\right)}
\,, \label{hbet_prec-opt}
  \end{eqnarray}
}
with
{\scriptsize
\begin{eqnarray}
\hat{d}_1\left(\frac{1}{p}\bSigma\right)&=&\frac{1}{\hat{v}(0)}\left(\frac{1}{c_n \hat{v}(0)}-\hat{d}_1\left(\frac{1}{p}\bI_p\right)\right), \label{th5-eq1}
\\
\hat{d}_1\left(\frac{1}{p}\bSigma^2\right)&=& \frac{1}{[\hat{v}(0)]^2}\left(\frac{1}{p}\tr[\bS_n]+\hat{d}_1\left(\frac{1}{p}\bI_p\right)-\frac{2}{c_n \hat{v}(0)}\right),\label{th5-eq2}\\
\hat{d}_1\left(0,\frac{1}{p}\bSigma^2\boldsymbol{\Pi}_0\right)&=&
\frac{1}{[\hat{v}(0)]^2}\left(\frac{1}{p}\tr[\boldsymbol{\bS_n\Pi}_0]+\hat{d}_1\left(\frac{1}{p}\boldsymbol{\Pi}_0\right)\right)-\frac{2}{[\hat{v}(0)]^3}\left(\frac{1}{p}\tr[\boldsymbol{\Pi}_0]-
\hat{d}_0\left(0,\frac{1}{p}\boldsymbol{\Pi}_0\right)\right),\label{th5-eq4}\\
\hat{d}_2\left(\frac{1}{p}\bSigma^2\right)&=& \frac{1}{\hat{v}(0)} \hat{d}_1\left(\frac{1}{p}\bSigma^2\right) - \frac{1}{[\hat{v}(0)]^2} \left(\hat{d}_1\left(\frac{1}{p}\bSigma\right)-\hat{d}_2\left(\frac{1}{p}\bI_p\right)\right),\label{th5-eq5}
\end{eqnarray}
}
where $\hat{v}(0)$, $\hat{h}_2$, $\hat{h}_3$, $\hat{d}_1\left(\frac{1}{p}\bI_p\right)$, $\hat{d}_1\left(\frac{1}{p}\boldsymbol{\Pi}_0\right)$, $\hat{d}_2\left(\frac{1}{p}\bI_p\right)$, $\hat{d}_0\left(0,\frac{1}{p}\boldsymbol{\Pi}_0\right)$, $\hat{q}_1\left(\frac{1}{p}\boldsymbol{\Pi}_0\right)$ and $\hat{q}_2\left(\frac{1}{p}\boldsymbol{\Pi}_0^2\right)$ are given in \eqref{hv0-all} and in (S.38), (S.39), (S.40), (S.41), (S.42), (S.45), (S.47), and (S.48) in the supplement (\cite{BP2025reviving-S}), 
respectively.
\end{theorem}

The proof of Theorem \ref{th5} is given in the supplementary material. If $\boldsymbol{\Pi}_0=\bI_p$, then $d_0\left(0,\frac{1}{p}\bI_p\right)=\frac{c_n-1}{c_n}$ and the results of Theorem \ref{th5} can be significantly simplified. 
Using Theorem \ref{th5}, a bona fide shrinkage estimator for the precision matrix is deduced as
 \begin{equation}\label{BF-prec}
\widehat{\boldsymbol{\Pi}}_{MP;BF}=\hat{\alpha}_{MP}^*\bS_n^{+}+\hat{\beta}_{MP}^*\boldsymbol{\Pi}_0.
\end{equation}

\subsubsection{Results for the ridge-type inverse}\label{sec:sh-prec-Ridge}
Similarly, a shrinkage estimator for the precision matrix can be constructed by using the ridge-type inverse of the sample covariance matrix. First, we consider the case when the tuning parameter $t$ is fixed to some preselected value $t_0$ and then extend the results by minimizing the loss function \eqref{risk} additionally with respect to the tuning parameter $t$. In the former case the first two summands in \eqref{risk} are considered and the solution is given by
\begin{equation*}
\alpha_{R;n}^*(t_0)=\alpha_n^*(\bS_n^{-}(t_0))
\quad \text{and} \quad
\beta_{R;n}^*(t_0)=\beta_n^*(\bS_n^{-}(t_0)),
\end{equation*}
while in the latter case $L^2_{R;n,2}(t)=L^2_{F;n,2}(\bS_n^{-}(t))$ is additionally maximized with respect to the tuning parameter $t$, which is then used in the determination of the optimal shrinkage intensities. 


In Theorem S.6.1 from \cite{BP2025reviving-S}, the deterministic equivalents of $\alpha_{R;n}^*(t_0)$ and $\beta_{R;n}^*(t_0)$ in the high-dimensional setting are presented by using Corollary \ref{cor3}, while their consistent estimators are given in Theorem \ref{th_prec_R-est}. Moreover, Theorem S.6.2 from the supplementary material (\cite{BP2025reviving-S}) presents the deterministic asymptotic equivalent $L^2_{R;2}$ to $L^2_{R;n,2}(t)$ by applying the results of Corollary \ref{cor3}, which is  consistently estimated in the second statement of Theorem \ref{th_prec_R-est} (see, the formula \eqref{L_prec-opt-R-hat}).
\begin{theorem}\label{th_prec_R-est} Let $\bY_n$ fulfill the stochastic representation \eqref{obs}. Then, under Assumptions \textbf{(A1)}-\textbf{(A2)} for any $t_0>0$, consistent estimators for $\alpha_R^*(t_0)$ and $\beta_R^*(t_0)$ are given by
{\scriptsize
\begin{eqnarray}
\label{alp_prec-opt-R-hat}
&\hat{\alpha}_R^*(t_0)= \dfrac{\hat{d}_0\left(t_0,\frac{1}{p}\bSigma\right)\hat{q}_2\left(\frac{1}{p}\boldsymbol{\Pi}_0^2\right) - \hat{d}_0\left(t_0,\frac{1}{p}\bSigma^2\boldsymbol{\Pi}_0\right) \hat{q}_1\left(\frac{1}{p}\boldsymbol{\Pi}_0\right)}{\left(t_0^{-1}\hat{d}_0\left(t_0,\frac{1}{p}\bSigma^2\right)+\hat{v}^{(1)}(t_0)\hat{d}_1\left(t_0,\frac{1}{p}\bSigma^2\right) \right)\hat{q}_2\left(\frac{1}{p}\boldsymbol{\Pi}_0^2\right)
  -t_0^{-1}\hat{d}^2_0\left(t_0,\frac{1}{p}\bSigma^2\boldsymbol{\Pi}_0\right) },\\[0.2cm]
&\hat{\beta}_R^*(t_0)=\dfrac{\left(t_0^{-1}\hat{d}_0\left(t_0,\frac{1}{p}\bSigma^2\right)+\hat{v}^{(1)}(t_0)\hat{d}_1\left(t_0,\frac{1}{p}\bSigma^2\right)\right)
\hat{q}_1\left(\frac{1}{p}\boldsymbol{\Pi}_0\right)
-t_0^{-1}\hat{d}_0\left(t_0,\frac{1}{p}\bSigma\right)\hat{d}_0\left(t_0,\frac{1}{p}\bSigma^2\boldsymbol{\Pi}_0\right)}
{\left(t_0^{-1}\hat{d}_0\left(t_0,\frac{1}{p}\bSigma^2\right)+\hat{v}^{(1)}(t_0)\hat{d}_1\left(t_0,\frac{1}{p}\bSigma^2\right) \right)\hat{q}_2\left(\frac{1}{p}\boldsymbol{\Pi}_0^2\right)
  -t_0^{-1}\hat{d}^2_0\left(t_0,\frac{1}{p}\bSigma^2\boldsymbol{\Pi}_0\right) }\,, \label{bet_prec-opt-R-hat}
  \end{eqnarray}
}
with
{\small
\begin{eqnarray}
\hat{d}_0\left(t_0,\frac{1}{p}\bSigma\right) &=& \frac{1}{c_n \hat{v}(t_0)}-\frac{t_0}{c_n},\label{hd0-Sigma}\\
\hat{d}_0\left(t_0,\frac{1}{p}\bSigma^2\right) &=&\frac{1}{\hat{v}(t_0)}\left( \frac{1}{p}\emph{tr}\left[\bS_n\right]-\frac{1}{c_n\hat{v}(t_0)}+\frac{t_0}{c_n}\right), \label{hd0-Sigma2}\\
\hat{d}_0\left(t_0,\frac{1}{p}\bSigma^2\boldsymbol{\Pi}_0\right) 
&=&\frac{1}{\hat{v}(t_0)} \frac{1}{p}\emph{tr}\left[\bS_n\boldsymbol{\Pi}_0\right]-\frac{1}{[\hat{v}(t_0)]^2}\left(\frac{1}{p}\emph{tr}\left[\boldsymbol{\Pi}_0\right]
-\hat{d}_0\left(t_0,\frac{1}{p}\boldsymbol{\Pi}_0\right)\right),\label{hd0-Sigma2Pi}\\
\hat{d}_1\left(t_0,\frac{1}{p}\bSigma^2\right) &=&\frac{1}{[\hat{v}(t_0)]^2}\left(\frac{1}{p}\emph{tr}\left[\bS_n\right]+ \hat{d}_1\left(t_0,\frac{1}{p}\bI_p\right)-\frac{2}{c_n\hat{v}(t_0)}+\frac{2t_0}{c_n}\right),\label{hd1-Sigma2}
\end{eqnarray}
}
where $\hat{v}(t_0)$, $\hat{v}^{(1)}(t_0)$, $\hat{d}_0\left(t_0,\frac{1}{p}\boldsymbol{\Pi}_0\right)$, $\hat{d}_1\left(t_0,\frac{1}{p}\bI_p\right)$,
$\hat{q}_1\left(\frac{1}{p}\boldsymbol{\Pi}_0\right)$ and $\hat{q}_2\left(\frac{1}{p}\boldsymbol{\Pi}_0^2\right)$ are given in \eqref{hv0-all1} and in (S.45), (S.46), (S.47), and (S.48) from the supplement (\cite{BP2025reviving-S}), 
respectively. 

Moreover, for any $t>0$, it holds that
\begin{eqnarray*}\frac{1}{p}\left|\hat{L}^2_{R;2}(t) - L^2_{R;2}(t)\right|\stackrel{a.s.}{\rightarrow} 0,~~
   \text{for} \quad p/n \rightarrow c \in (1,\infty) 
~ \text{as} ~~ n \rightarrow \infty \quad\text{with}
   \end{eqnarray*}
{\footnotesize
\begin{eqnarray}
\label{L_prec-opt-R-hat}
\hat{L}^2_{R;2}(t) &=&\frac{1}{\hat{q}_2\left(\dfrac{1}{p}\boldsymbol{\Pi}_0^2\right)}\frac{\left[\hat{d}_0\left(t,\dfrac{1}{p}\bSigma\right)\hat{q}_2\left(\dfrac{1}{p}\boldsymbol{\Pi}_0^2\right)
- \hat{d}_0\left(t,\dfrac{1}{p}\bSigma^2\boldsymbol{\Pi}_0\right) \hat{q}_1\left(\frac{1}{p}\boldsymbol{\Pi}_0\right)\right]^2
}{\left(\hat{d}_0\left(t,\frac{1}{p}\bSigma^2\right)+t\hat{v}^{(1)}(t)\hat{d}_1\left(t,\frac{1}{p}\bSigma^2\right) \right)\hat{q}_2\left(\frac{1}{p}\boldsymbol{\Pi}_0^2\right)
  -\hat{d}^2_0\left(t,\frac{1}{p}\bSigma^2\boldsymbol{\Pi}_0\right) },
\end{eqnarray}
}
where $\hat{v}^{(1)}(t)$, 
$\hat{q}_1\left(\frac{1}{p}\boldsymbol{\Pi}_0\right)$, $\hat{q}_2\left(\frac{1}{p}\boldsymbol{\Pi}_0^2\right)$, $\hat{d}_0\left(t,\dfrac{1}{p}\bSigma\right)$, $\hat{d}_0\left(t,\dfrac{1}{p}\bSigma^2\right)$, $\hat{d}_0\left(t,\dfrac{1}{p}\bSigma^2\boldsymbol{\Pi}_0\right)$ and $\hat{d}_1\left(t,\dfrac{1}{p}\bSigma^2\right)$ are given in \eqref{hv0-all1}, \eqref{hd0-Sigma}, \eqref{hd0-Sigma2}, \eqref{hd0-Sigma2Pi}, and \eqref{hd1-Sigma2}, and in (S.47) and (S.48) from \cite{BP2025reviving-S}. 
\end{theorem}

The proofs of both statements of Theorem \ref{th_prec_R-est} follow from the theoretical results derived in Section \ref{sec:main-ridge} and the consistent estimators are presented in Section S.4 of the supplement (\cite{BP2025reviving-S}). The bona fide shrinkage estimator for the precision matrix under the application of the ridge-type inverse is given by
 \begin{equation}\label{BF-prec-R}
\widehat{\boldsymbol{\Pi}}_{R;BF}(t_0)=\hat{\alpha}_R^*(t_0)\bS_n^{-}(t_0)+\hat{\beta}_R^*(t_0)\boldsymbol{\Pi}_0 \quad \text{for some $t_0>0$.}
\end{equation}
It is noted that, here we have additionally to optimize the function \eqref{risk} for the tuning parameter $t$, or, equivalently due to Theorem \ref{lem_prec_loss}, maximize the function $L^2_{F;n,2}(\bS_n^{-}(t))$ from \eqref{risk-L-F2}. 
 The results of the second statement of Theorem \ref{th_prec_R-est} are very important from the practical viewpoint. In particular, they allow us to find the optimal value of the tuning parameter $t$ by maximizing $\hat{L}^2_{R;2}(t)$, which can be easily computed, instead of $L^2_{R;n,2}(t)$ that depends on the unknown population covariance matrix. Both results of Theorem \ref{th_prec_R-est} ensure that the tuning parameter found by maximizing $\hat{L}^2_{R;2}(t)$ will be closed when the original loss function is optimized. Furthermore, since $\hat{L}^2_{R;2}(t)$ is available in the closed form, no cross-validation is needed which by its definition reduces the sample size. 

Let $t^*=\argmax_{t>0} \hat{L}^2_{R;2}(t)$. Then, the bona-fide optimal shrinkage estimator of the precision matrix with the ridge-type estimator is expressed as
 \begin{equation}\label{BF-prec-R-L}
\widehat{\boldsymbol{\Pi}}_{R;BF}(t^*)=\hat{\alpha}_R^*(t^*)\bS_n^{-}(t^*)+\hat{\beta}_R^*(t^*)\boldsymbol{\Pi}_0 ,
\end{equation}
where $\hat{\alpha}_R^*(t^*)$ and $\hat{\beta}_R^*(t^*)$ are given in \eqref{alp_prec-opt-R-hat} and \eqref{bet_prec-opt-R-hat}, respectively. Finally, we note that for practical purposes one can consider to replace $t$ by $u$ with $t=\tan (u)$ and to maximize $\hat{L}^2_{R;2}(\tan (u))$ over the finite interval $(0,\pi/2)$.

\subsubsection{Results for the Moore-Penrose-ridge inverse}\label{sec:sh-prec-MP-Ridge}

When the shrinkage estimator of the precision matrix is based on the Moore-Penrose-ridge inverse, we proceed in the same way as in the case of the ridge inverse. First, we investigate the limiting behavior of the optimal shrinkage intensities 
\begin{equation*}
\alpha_{MPR;n}^*(t_0)=\alpha_n^*(\bS_n^{\pm}(t_0))
\quad \text{and} \quad
\beta_{MPR;n}^*(t_0)=\beta_n^*(\bS_n^{\pm}(t_0)),
\end{equation*}
for a fixed value of the tuning parameter $t_0$ and then, present the procedure for how the optimal value of the tuning parameter should be chosen.  Since this procedure mimics the one derived for the ridge inverse the results are moved to Section S.7 in \cite{BP2025reviving-S} (see Theorem S.7.1, Theorem S.7.2, Theorem S.7.3, and Theorem S.7.4 for details).

\subsection{Shrinkage estimator for minimum variance portfolio}\label{sec:shrink-gmv}
{ Following \cite{bodnar2018estimation}, the general linear shrinkage estimator for the global minimum variance (GMV) portfolio is constructed by
 \begin{equation}\label{gsegmv}
\widehat{\mathbf{w}}_{GSE}=\alpha_n  \mathbf{w}_{\bS_n^{\#}(t)}+(1-\alpha_n)\mathbf{b}\,,
\end{equation}
where $\mathbf{b}$ is a target portfolio, i.e., $\mathbf{b}^\top\bOne=1$, which typically is chosen to be equal to the naive portfolio $\frac{\bOne}{p}$ when no \emph{a priori} information about the true weights of the GMV portfolio
\begin{equation}\label{trueGMV}
    \mathbf{w}_{GMV}=\frac{\bSigma^{-1}\bOne}{\bOne^\top\bSigma^{-1}\bOne}
\end{equation} is available. Let $\bS_n^{\#}(t)$ denote an estimator of the precision matrix. Then, the traditional (sample) estimator of $\mathbf{w}_{GMV}$ is given by
\begin{eqnarray}
  \mathbf{w}_{\bS_n^{\#}(t)}=\frac{\bS_n^{\#}(t)\bOne}{\bOne^\top\bS_n^{\#}(t)\bOne}
\end{eqnarray}
is obtained by replacing $\bSigma^{-1}$ with $\bS_n^{\#}(t)$ in \eqref{trueGMV}. 

Due to \cite{bodnar2018estimation} and \cite{frahm2010}, the shrinkage estimator of the form \eqref{gsegmv} dominates the sample estimator $ \mathbf{w}_{\bS_n^{\#}(t)}$ for $c<1$ and $\bS_n^{\#}(t)=\bS_n^{-1}$. Moreover, the estimator proposed by \cite{bodnar2018estimation} is proved to be asymptotically dominant over \cite{frahm2010} in case $c<1$ and becomes superior if $c$ gets closer to one (see \cite{bodnaroutofsample2023}). 
Nevertheless, only the case $c<1$ was efficiently covered in the aforementioned works. In case $c>1$, \cite{bodnar2018estimation} provided the approximate formulae for the shrinkage intensities in the case of Moore-Penrose inverse. However, those formulae appear to be accurate only if $c<2$ and close to $1$, otherwise they could be far away from the optimal ones. The theoretical results of Section \ref{sec:main} provide us tools to deal with this problem efficiently.


The aim is to estimate the shrinkage intensity $\alpha_n$ and the ridge parameter $t$ from the following normalized optimization problem (see, e.g., \cite{BPTJMLR2024})
    \begin{equation}\label{normalization}
    \min_{\alpha_n, t} \frac{\widehat{\mathbf{w}}_{GSE}^\top~\bSigma ~\widehat{\mathbf{w}}_{GSE}}{\bb^\top \bSigma \bb}.
\end{equation}
In the numerator, we have the usual performance measure, the so-called out-of-sample variance  defined by $\widehat{\mathbf{w}}_{GSE}^\top\bSigma \widehat{\mathbf{w}}_{GSE}$ for the portfolio with weights $\widehat{\mathbf{w}}_{GSE}$, and the denominator $\bb^\top \bSigma \bb$ is merely for technical reasons, i.e., it is a normalization, which makes the loss bounded for increasing portfolio dimension $p$. The following theorem gives the optimal values of $\alpha_n$ and $t$.

\begin{theorem}\label{lem_gmv_loss}
  For given $t$, the out-of-sample variance \eqref{normalization} is minimized with respect to $\alpha$ at
\begin{equation}\label{alp_ni_star-lambda}
    \alpha_n^*(t) = 
    \frac{
        \bb^\top \bSigma \left(\bb -  \mathbf{w}_{\bS_n^{\#}(t)}\right)
    }{
        \left(\bb -  \mathbf{w}_{\bS_n^{\#}(t)}\right)^\top \bSigma\left(\bb- \mathbf{w}_{\bS_n^{\#}(t)}\right)
    },
\end{equation}
while the optimal $t$ is found by maximizing 
\begin{equation}\label{Li2-lambda}
    L_{n;2}(t)=\frac{1}{\bb^\top\bSigma \bb}
        \frac{\left(\bb^\top \bSigma \left(\bb- \mathbf{w}_{\bS_n^{\#}(t)}\right)\right)^2}
        {\left(\bb- \mathbf{w}_{\bS_n^{\#}(t)}\right)^\top \bSigma\left(\bb- \mathbf{w}_{\bS_n^{\#}(t)}\right)}.
\end{equation}  
\end{theorem}

The proof of Theorem \ref{lem_gmv_loss} is given in the supplementary material. Similarly as before, to find the value $t^*$ together with $\alpha_n^*(t)$, which minimizes the loss function, we proceed in three steps. First, we find the deterministic equivalent to $L_{n;2}(t)$, and estimate it consistently in the second step. Finally, we minimize the obtained consistent estimator in the last step. This procedure has to be done for the  generalized inverses $\bS_n^{\#}(t)\in \{\bS_n^-(t),\bS_n^\pm(t)\}$ except $\bS_n^{\#}(t)=\bS_n^+$ because the latter one is independent of the tuning parameter $t$. Thus, it is enough to find the asymptotic equivalent\footnote{Note that the asymptotic shrinkage intensity $\alpha^*$ clearly depends on $n$, due to the dependence of both $\bb$ and $\bSigma$ on $p = p(n)$. However, this dependence was intentionally omitted in the notation to avoid overwhelming the reader and to keep the presentation of the results as clear and simple as possible.}, i.e., $\alpha^*$, for the optimal shrinkage intensity $\alpha^*_n(t)\equiv \alpha^*_n $ and to estimate it consistently. The results regarding the asymptotic equivalent $\alpha^*$ in case of the Moore-Penrose inverse $\bS_n^+$ are summarized in Theorem S.8.1 in \cite{BP2025reviving-S}, while Theorem \ref{MV_GMV_bona} provides the bona-fide estimator, i.e., a consistent estimator of $\alpha^*$.



\begin{theorem}\label{MV_GMV_bona}
     Let $\bY_n$ possess the stochastic representation as in \eqref{obs}. Then, under the assumption \textbf{(A1)} and \textbf{(A2)}, a consistent estimator of $\alpha^*$ is given by
\begin{equation}\label{alpha-bon}
\hat{\alpha}^*=
  \frac{p\bb^\top\bS_n\bb-\frac{\hat{d}_1\left(\bOne\bb^\top\bSigma  \right)}{\hat{d}_1\left(\frac{\bOne\bOne^\top}{p}\right)}}{p\bb^\top\bS_n \bb-2\frac{\hat{d}_1\left(\bOne\bb^\top\bSigma \right)}{\hat{d}_1\left( \frac{\bOne\bOne^\top}{p}\right)}+\frac{\hat{d}_3\left(\frac{\bOne\bOne^\top}{p}\right)}{\hat{d}_1^2\left(\frac{\bOne\bOne^\top}{p}\right)}},
\end{equation}
where
\begin{eqnarray}
 \hat{d}_1(\bOne\bb^\top\bSigma)   &=& \frac{1}{\hat{v}(0)}\left[\frac{1}{\hat{v}(0)}\left(1-\hat{d}_0(0,\bOne\bb^\top) \right) - \hat{d}_1(\bOne\bb^\top) \right]\,
\end{eqnarray}
with $\hat{v}(0)$, $\hat{d}_1\left( \frac{\bOne\bOne^\top}{p}\right)$, $\hat{d}_1\left( \bOne\bb^\top\right)$, $d_3(\frac{\bOne\bOne^\top}{p})$, and $\hat{d}_0(0, \bOne\bb^\top)$ given in \eqref{hv0-all} and in (S.41), (S.44),
and (S.45) 
from the supplement (\cite{BP2025reviving-S}), respectively.
\end{theorem}

In the same way as in Section \ref{sec:shrink-prec}, the shrinkage estimators for the GMV portfolio weights can be constructed by using the ridge-type inverse and the Moore-Penrose-ridge inverse. As mentioned above, one should consider first the case when the tuning parameter $t$ is fixed to some preselected value $t_0$ and then extend the results by minimizing the loss function \eqref{Li2-lambda} additionally with respect to the tuning parameter $t$. The case of the ridge inverse has already been covered recently in \cite{BPTJMLR2024}. The case of Moore-Penrose-ridge inverse can be done similarly, however with very tedious computations. 

}


\section{Finite sample performance}\label{sec:sim}

In this section, we study the finite-sample performance of the considered estimators via simulations. Two stochastic models will be considered. In the first scenario, we assume that the elements of the matrix $\bX_n$ are standard normally distributed, while they are assumed to be scaled $t$-distributed with $5$ degrees of freedom in the second scenario. The scale factor in the case of the $t$-distribution is set to $\sqrt{3/5}$ to ensure that the variances of the elements of $\bX_n$ are all one. Finally, the mean vector $\bmu$ is taken to be zero. The eigenvectors of the population covariance matrix are drawn from the Haar distribution (see, e.g., \cite{muirhead1990}), while its eigenvalues are chosen in the following way: $20\%$ of eigenvalues are equal to one, 40\% of eigenvalues equal to three, and the rest equal to ten. Finally, we set $n\in\{100,250\}$ and $c \in (1,5]$. In the case of the ridge estimator and the Moore-Penrose-ridge estimator, we consider $t \in [0.1,10]$. Within the simulation study, we compare the performance of the introduced shrinkage estimators with the existing benchmark approaches. The figures with the computation times are provided in the supplement.

\subsection{Results for the estimation of precision matrix}\label{sec:sim_prec}
In this section, we compare the derived shrinkage approaches for the estimation of the precision matrix (see Section \ref{sec:shrink-prec}) with the Moore-Penrose inverse and several benchmark methods, which are used in the statistical literature. These approaches are the \textbf{empirical Bayes} estimator of \cite{kubokawa2008estimation}, the \textbf{optimal ridge} estimator of \cite{wang2015shrinkage} and the \textbf{inverse nonlinear (NL) shrinkage} estimator of \cite{lw20}. More details about the considered benchmark approaches can be found in Section S.9.1 of the supplement (\cite{BP2025reviving-S}). Besides that, we consider the {\bf oracle nonlinear shrinkage} estimator, which is also introduced in Section S.9.1 of the supplement. Although the latter estimator cannot be computed in practice since it depends on the unobservable population covariance matrix, it will be used in the simulation study as the global benchmark strategy. The target matrix is set to the identity matrix, i.e., $\boldsymbol{\Pi}_0=\bI_p$. In fact, the linear shrinkage estimator \eqref{gse} becomes purely a function of the eigenvalues of the sample covariance $\bS_n$ for the considered target matrix and, thus, the oracle nonlinear shrinkage estimator is the best one can do in that specific situation.

Next, we compare the three shrinkage approaches suggested in Section \ref{sec:shrink-prec} with the considered benchmark methods. As a performance measure, we use the percentage relative improvement in average loss (PRIAL) defined by
 \begin{equation}\label{PRIAL}
 \text{PRIAL}(\widehat{\boldsymbol{\Pi}})=\left(1-\dfrac{\mathbbm{E}||\widehat{\boldsymbol{\Pi}}\bSigma-\bI_p||^2_F}{\mathbbm{E}||\bS_n^+\bSigma-\bI_p||_F^2}\right)\cdot100\%\,,
 \end{equation}
where $\widehat{\boldsymbol{\Pi}}$ is an estimator of the precision matrix $\bSigma^{-1}$. By definition, the PRIAL provides the percentage improvement of each strategy in comparison to the one based on the Moore-Penrose inverse. Larger values of the PRIAL indicate better performance. Moreover, it holds that PRIAL($\bSigma^{-1}$)$=100\%$.

Figure \ref{fig:shrinkage-prec} depicts the values of the PRIAL, computed for the introduced three shrinkage approaches and for the two benchmarks: the optimal ridge estimator and the oracle nonlinear shrinkage estimator. The results of the empirical Bayes estimator and of the inverse nonlinear shrinkage estimator are considerably worse and are available in Section S.9.2 of the supplement (see Figure S.1 in \cite{BP2025reviving-S}).

In Figure \ref{fig:shrinkage-prec}, we observe that the ridge shrinkage and the optimal ridge estimators outperform the other competitors with the ridge shrinkage estimator having a slightly better performance. This observation holds for the two considered models of the data-generating process and for the selected values of $n$ and $c_n$. The Moore-Penrose shrinkage and the Moore-Penrose ridge shrinkage estimators are ranked in the third and fourth places with the Moore-Penrose ridge shrinkage estimator performing better in almost all cases with the exception of large values of $c_n$ when $n=250$. These two estimators are followed by the empirical Bayes ridge which performs better than the inverse nonlinear shrinkage estimator (see Figure S.1 in \cite{BP2025reviving-S}). However, it has to be noted that the considered inverse nonlinear shrinkage estimator was derived under the other loss function, which explains the considerable differences between the oracle and bona fide nonlinear shrinkage estimators. Finally, it is remarkable that both the ridge shrinkage and the optimal ridge attain the global benchmark, which is the oracle nonlinear shrinkage estimator and, as such, they can be considered optimal for the chosen data-generating models. Although the Moore-Penrose shrinkage estimator performs slightly worse than the shrinkage approaches based on the ridge estimators, it is present in closed-form and requires no numerical approximation in practice. This fact is illustrated in Figure S.2 of \cite{BP2025reviving-S}, which compares the computational time of the estimators. Among them, the MP shrinkage estimator is the fastest, second only to the Moore-Penrose inverse itself.

\vspace{-1cm}
\begin{figure}[h!t]
\centering
\begin{tabular}{cc}
\hspace{-0.5cm}\includegraphics[width=7cm]{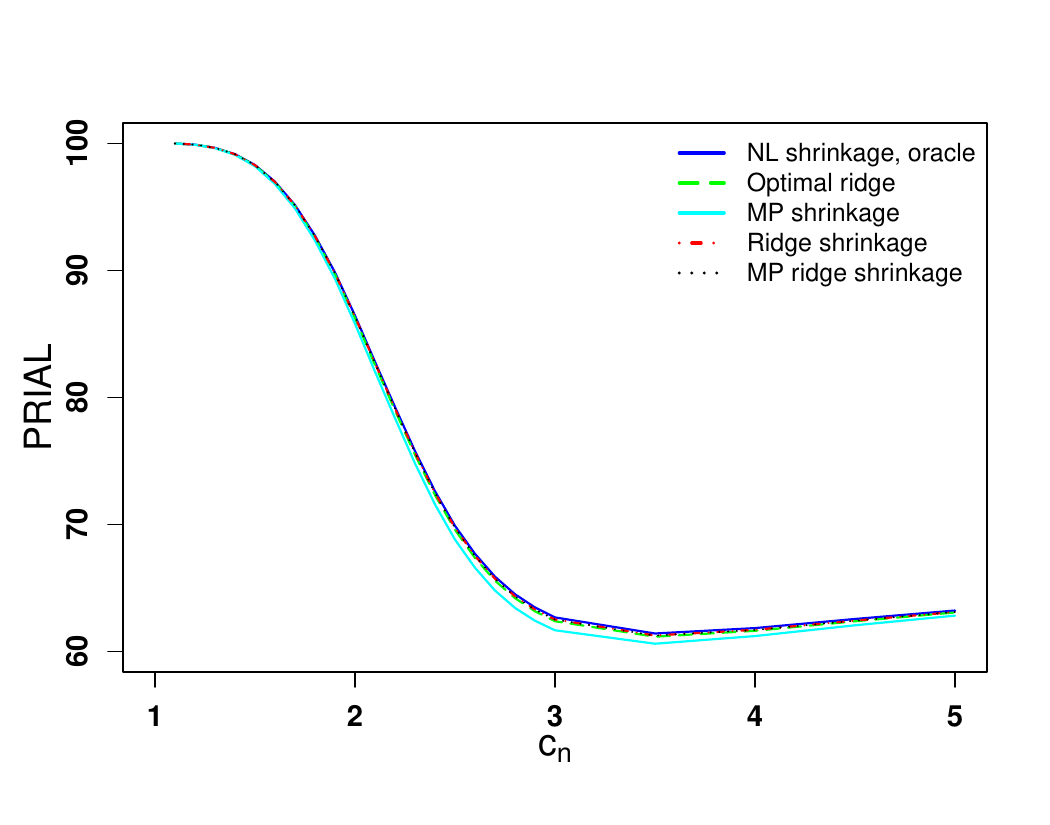}&
\hspace{-0.5cm}\includegraphics[width=7cm]{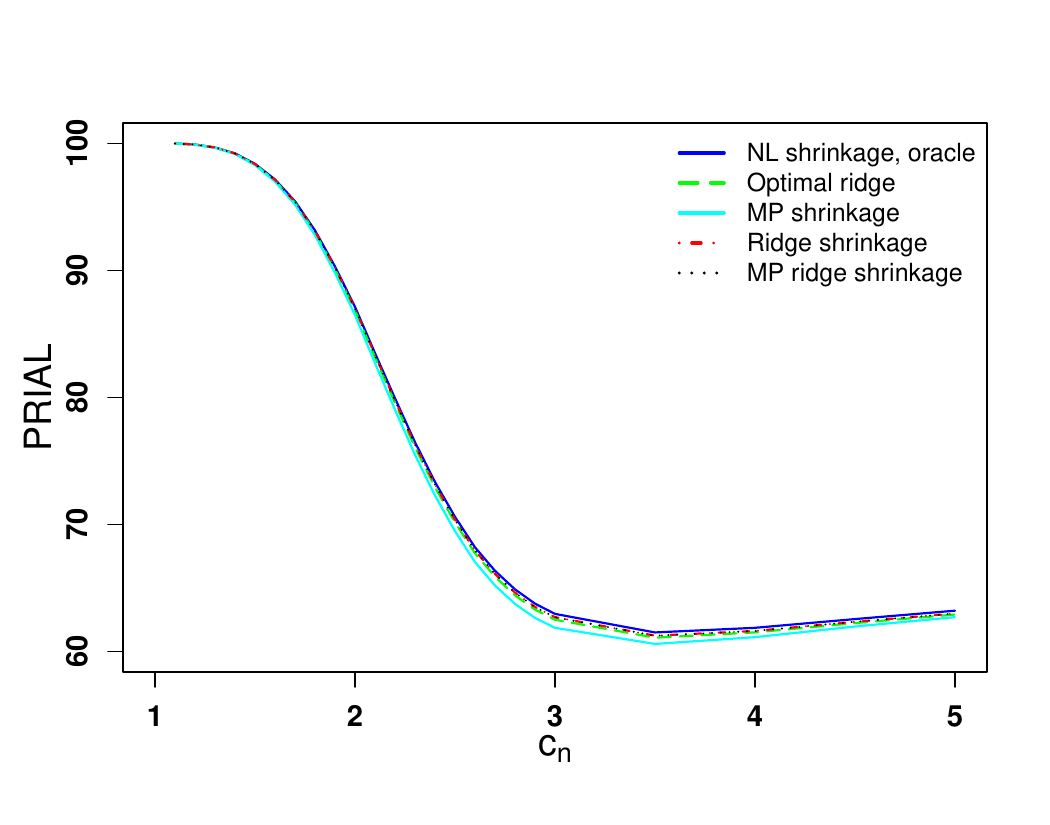}\\[-1.2cm]
\hspace{-0.5cm}\includegraphics[width=7cm]{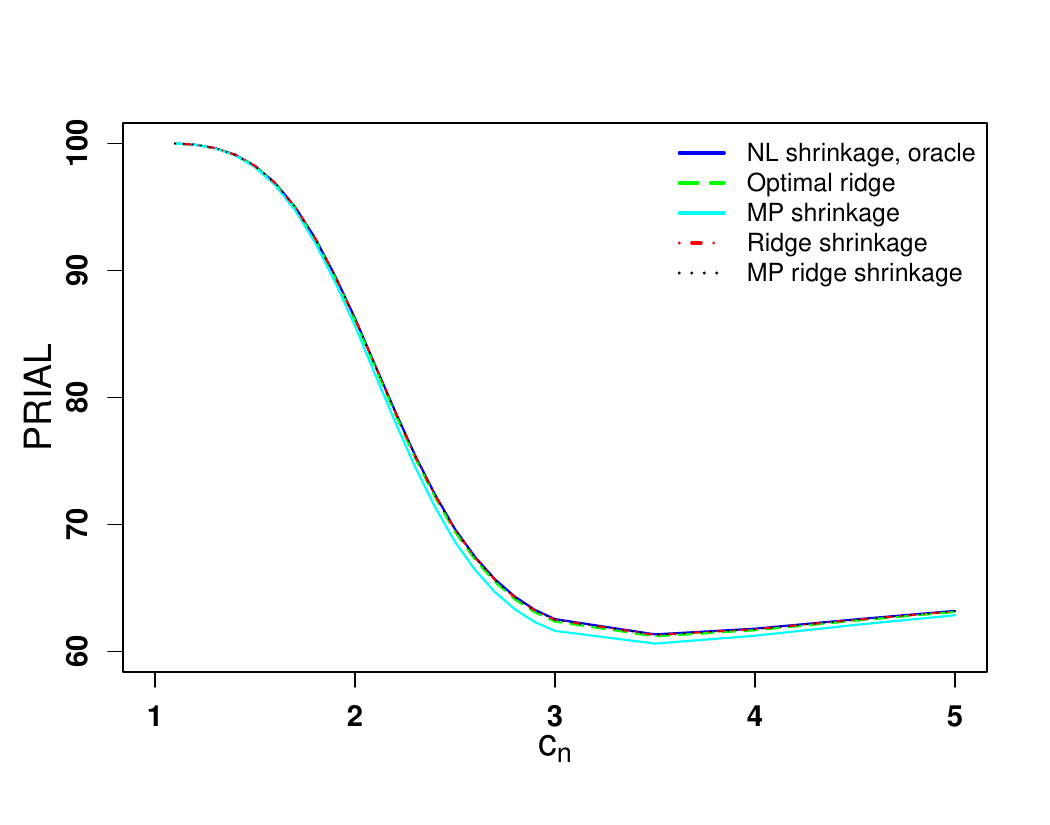}&
\hspace{-0.5cm}\includegraphics[width=7cm]{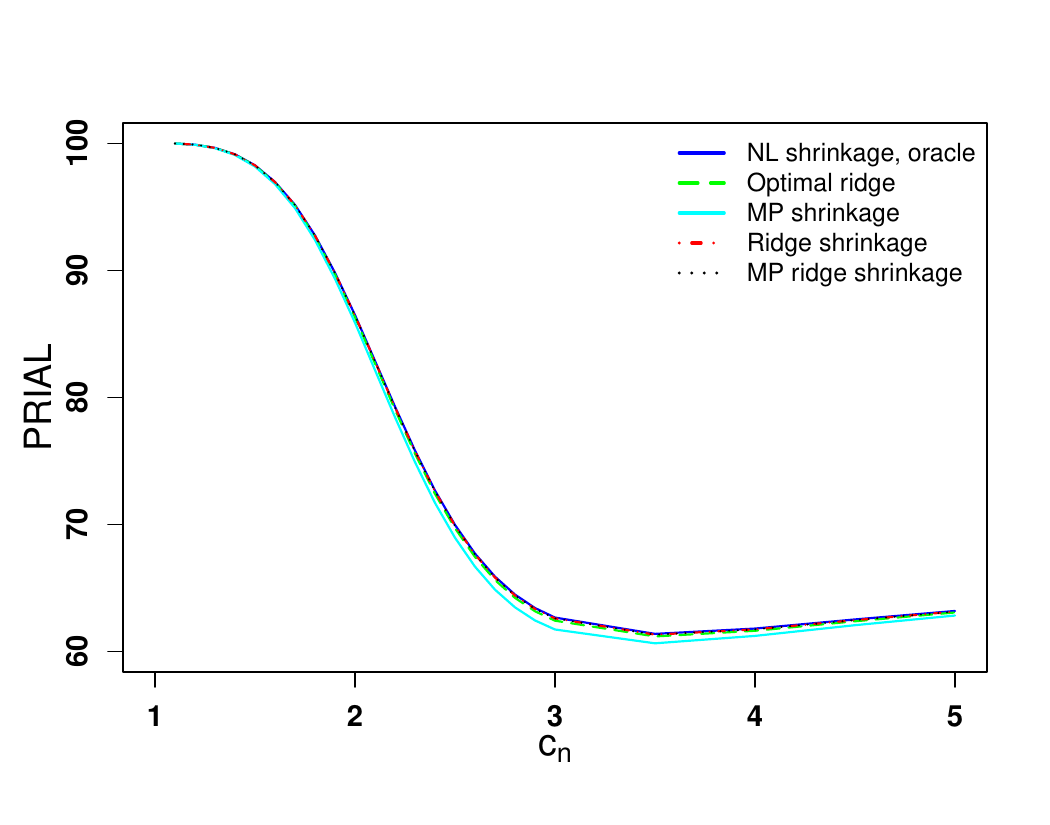}\\[-0.8cm]
\end{tabular}
 \caption{PRIAL for $c_n \in (1,5]$, $n=100$ (first row) and $n=250$ (second row) when the elements of $\bX_n$ are drawn from the normal distribution (first column) and scale $t$-distribution (second column).}
\label{fig:shrinkage-prec}
 \end{figure}
 
\subsection{Results for the estimation of the global minimum variance portfolio}\label{sec:sim_gmv}
In this section, we compare the derived shrinkage approaches for the GMV portfolio for estimating the GMV portfolio (see Section \ref{sec:shrink-gmv}) with the \textbf{traditional} sample estimator based on the Moore-Penrose inverse, the \textbf{reflexive inverse estimator} introduced in \cite{bodnar2018estimation}, and the {\bf double shrinkage estimator} or {\bf ridge shrinkage estimator} recently derived in \cite{BPTJMLR2024}. All benchmarks are described in Section S.9.3 of the supplement (\cite{BP2025reviving-S}). In the simulation study, the target portfolio $\bb$ is set to the equally weighted portfolio.

As a performance measure, we use the relative out-of-sample variance (rOSV) defined as 
\begin{equation}\label{rosvar} \text{rOSV}(\widehat{\mathbf{w}}) = \frac{\widehat{\mathbf{w}}^\top\bSigma\widehat{\mathbf{w}}}{V_{GMV}}-1 ,
\end{equation}
where $\widehat{\mathbf{w}}$ is any estimator of the GMV portfolio, and $V_{GMV}$ is the variance of the true (population) GMV portfolio given by $V_{GMV}=\frac{1}{\bOne^\top\bSigma^{-1}\bOne}$. By definition, rOSV is always greater than zero and equals zero only if $\widehat{\mathbf{w}} = {\mathbf{w}}_{GMV}$ given by \eqref{trueGMV}. Therefore, a smaller rOSV indicates better performance of the portfolio estimator.

Figure \ref{fig:shrinkage-gmv} presents the rOSV values computed for the proposed Moore-Penrose portfolio shrinkage approach and the two benchmarks: the  double shrinkage estimator and the reflexive inverse estimator. Results for the traditional estimator, which are available in Section S.9.4 (see Figure S.4 in \cite{BP2025reviving-S}), are omitted here due to its poor performance.

\begin{figure}[h!t]
\centering
\begin{tabular}{cc}
\hspace{-0.5cm}\includegraphics[width=7cm]{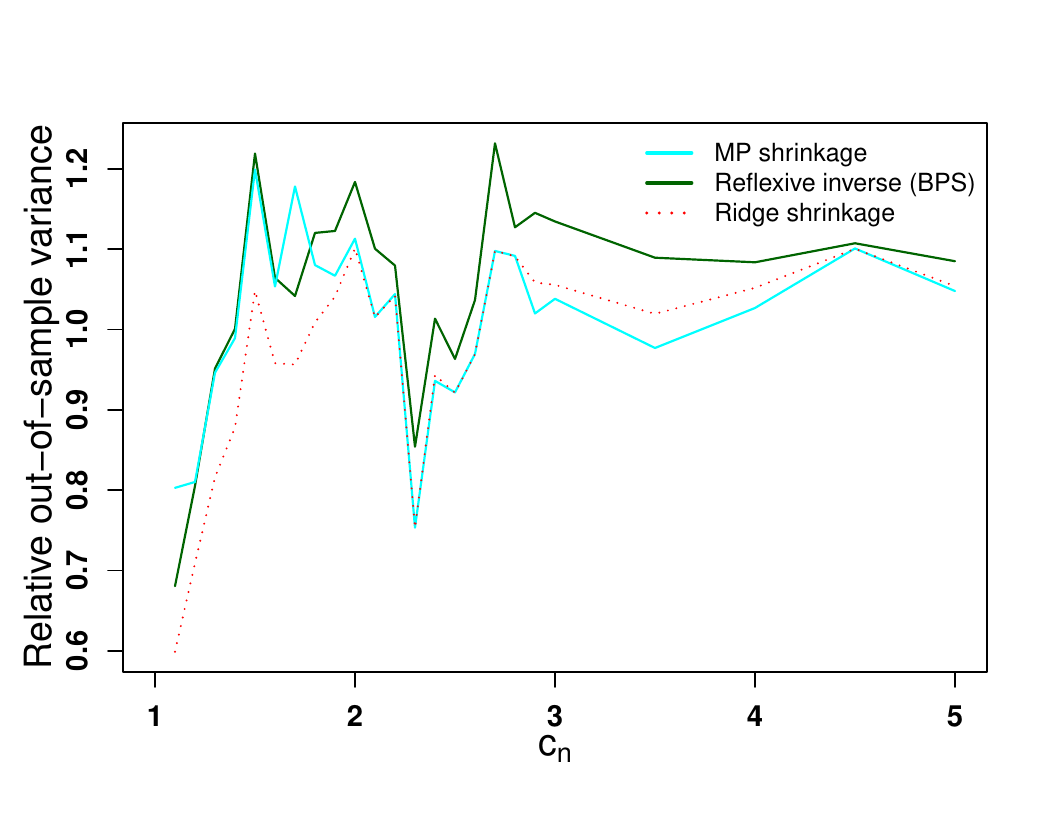}&
\hspace{-0.5cm}\includegraphics[width=7cm]{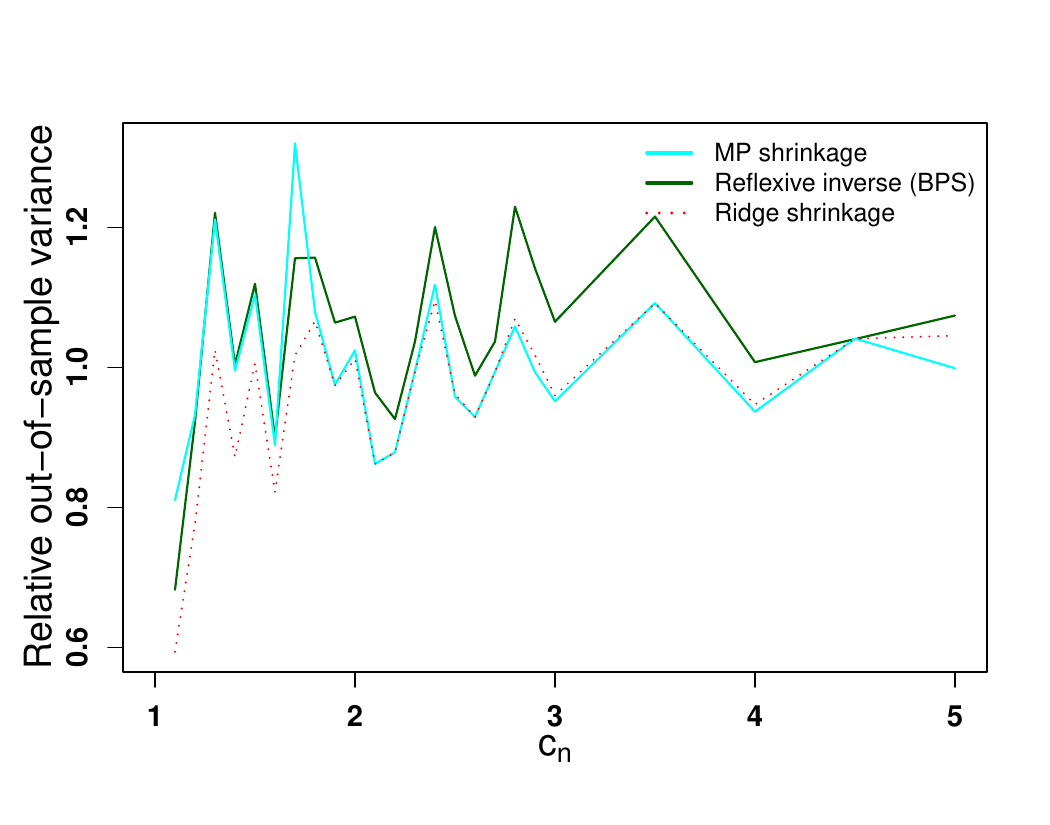}\\[-1.3cm]
\hspace{-0.5cm}\includegraphics[width=7cm]{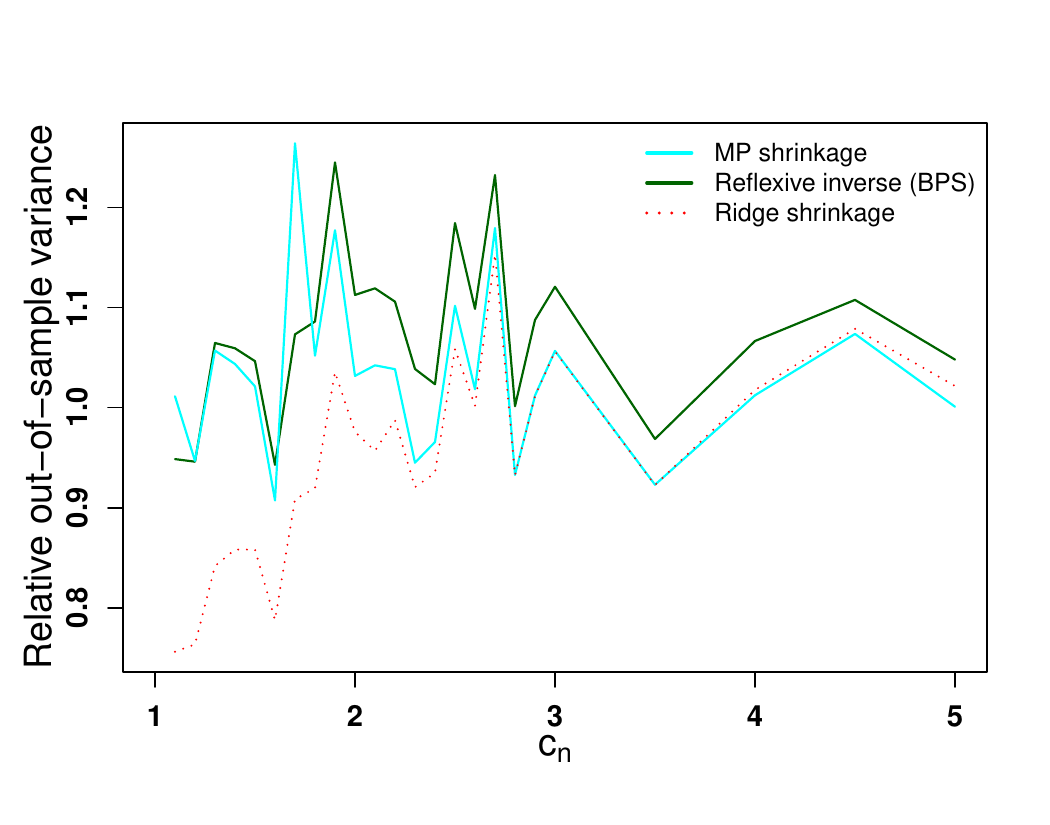}&
\hspace{-0.5cm}\includegraphics[width=7cm]{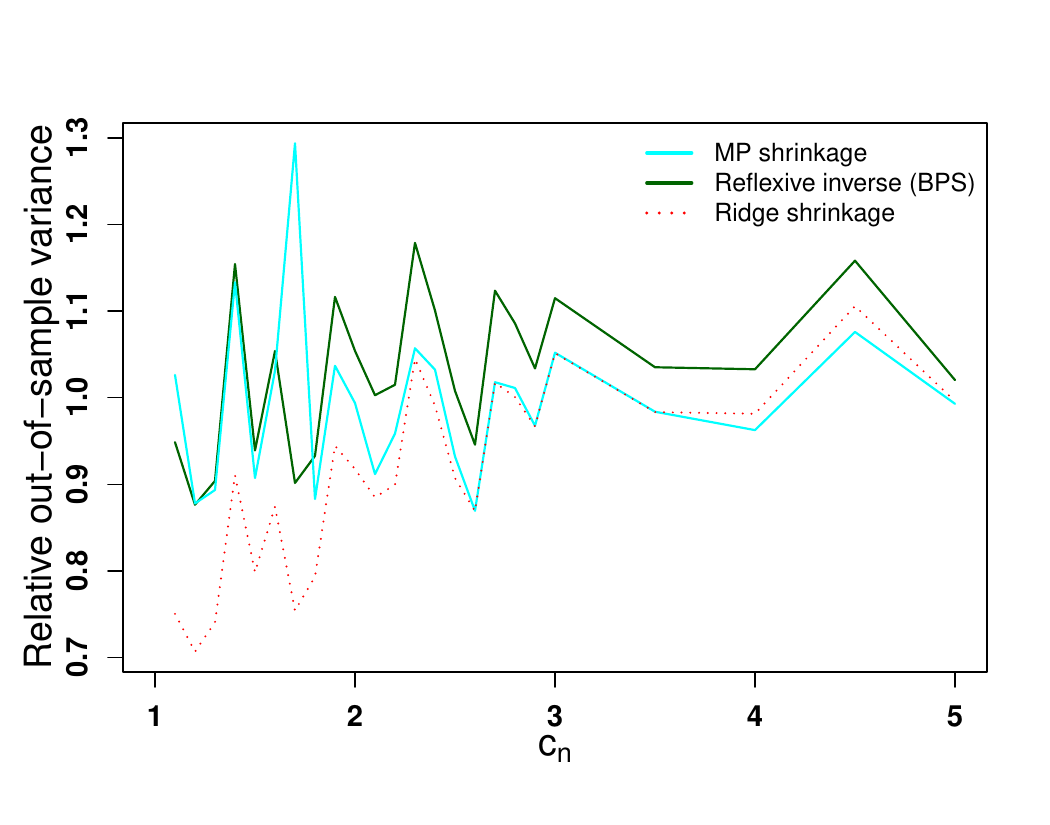}\\[-0.9cm]
\end{tabular}
 \caption{rOSV for $c_n \in (1,5]$, $n=100$ (first row) and $n=250$ (second row) when the elements of $\bX_n$ are drawn from the normal distribution (first column) and scale $t$-distribution (second column).}
\label{fig:shrinkage-gmv}
 \end{figure}
From Figure \ref{fig:shrinkage-gmv}, we observe that all approaches exhibit similar behavior, with the proposed MP shrinkage and double (ridge) shrinkage performing slightly better as the ratio $c=p/n$ increases. The reflexive inverse estimator, as shown in \cite{bodnar2018estimation}, serves as a good approximation to the Moore-Penrose inverse only when $c$ is less than two, a result that is supported here. Additionally, the recently proposed double shrinkage approach by \cite{BPTJMLR2024} outperforms other strategies for smaller values of the concentration ratio $c$.

It is worth noting, however, that the proposed MP shrinkage approach avoids the need for numerical optimization of the penalty parameter, unlike the double shrinkage strategy, which requires a quite involved nonlinear optimization routine. The computational advantages of the MP shrinkage estimator are depicted in Figure S.5 of \cite{BP2025reviving-S}.




\section*{Acknowledgement}
The authors would like to thank Professor Enno Mammen, Professor Hans-Georg M\"{u}ller, the Associate Editor, and the two anonymous Reviewers for their constructive comments that improved the quality of this paper. We gratefully acknowledge the comments from the participants at the conference 'International Conference on Matrix Analysis and its Applications' (Bendlewo) 2023,  'Computational and Methodological Statistics' (CMStatistics) 2024 (Berlin), 'Stochastic models, Statistics and their Applications' (SMSA) 2024 (Delft), the 11th Tartu Conference on Multivariate Statistics (Tartu), the 11th World Congress in Probability and Statistics (Bochum) and the 2nd Joint Conference on Statistics and Data Science in China. The authors are also thankful to Prof. Holger Dette, Prof. Raymond Kan, Prof. Olivier Ledoit, Prof. Mark Podolskij, Dr. habil. Maryna Prus, Prof. Dietrich von Rosen, and Prof. Jianfeng Yao for fruitful discussions, and Dr. Alexis Derumigny for creating and developing the R package \texttt{UniversalShrink}, used extensively in this work. Taras Bodnar was partly supported by the Swedish Research Council (VR) via the grant 2025-04704 {\it ''Reviving generalized inverses: High-dimensional statistical analysis with insufficient data''}.

\begin{frontmatter}
   \title{Supplement to ``Reviving pseudo-inverses:\\ Asymptotic properties of large dimensional Moore-Penrose and Ridge-type inverses with applications''}
 \runtitle{Supplement: Asymptotic properties of large inverses}
 \runauthor{T. Bodnar and N. Parolya}

 \begin{aug}
\author[A]{\fnms{Taras}~\snm{Bodnar}\ead[label=e1]{taras.bodnar@liu.se}}
  \and
 \author[B]{\fnms{Nestor}~\snm{Parolya}\ead[label=e2]{n.parolya@tudelft.nl}}
 \address[A]{{Department of Management and Engineering, Link\"{o}ping University, SE-581 83 Link\"{o}ping, Sweden\printead[presep={,\ }]{e1}}}
 \address[B]{Department of Applied Mathematics, Delft University of Technology, Mekelweg 4, 2628 CD Delft, The Netherlands\printead[presep={,\ }]{e2}}
 \end{aug}

 \end{frontmatter}


\makeatletter
 \renewcommand*\l@subsection{\@dottedtocline{2}{1.8em}{3.2em}}
\makeatother

\tableofcontents
\def\theequation{S.\arabic{equation}}
\def\thefigure{S.\arabic{figure}}
\def\thesection{S.\arabic{section}}
\def\thesubsection{S.\arabic{section}.\arabic{subsection}}
\setcounter{equation}{0}

\section{Proofs of the main theorems and corollaries}

Throughout the supplement, we use the notation $A_n \stackrel{d.a.s.}{\rightarrow} B_n$ to denote the convergence in difference almost surely, that is, $A_n - B_n \stackrel{a.s.}{\rightarrow} 0$. All limits are for $n\to\infty$ unless explicitly stated otherwise. An abstract {\it sketch of the proof}, which applies to all theorems proved in Section 2 of the paper, can be summarized in the following way:
\begin{description}
    \item[\bf Step 1: Centering.] In the first step, we show that the asymptotic behaviour of the weighted trace moments of the Moore-Penrose inverse (and other inverses) with centering by the sample mean is the same as without centering. This is achieved by rewriting the sample covariance matrix in terms of a 'new' observation matrix $\tbY_n = \bY_n\bH_n$, where $\bH_n$ is a specific orthogonal matrix of rank $n-1$.
    
    \item[\bf Step 2: Finding a nonconventional moment generating function.] In the second step, we use the well-known Woodbury matrix identity (matrix inversion lemma, see, e.g., \cite{hornjohn1985}) to derive a nonconvential moment generating function for the weighted traces of the considered generalized inverses.
    
    \item[\bf Step 3: Rewriting the nonconventional moment generating function.] Since the dependence on the true covariance matrix $\bSigma$ is complex, we rewrite the moment generating function in a form that facilitates its high-dimensional analysis. This is done by algebraic manipulation of the weighted traces and by 'pushing' the dependence on $\bSigma$ into the weighting matrix $\bTheta$, thereby reducing the problem to a classical setup with an isotropic case, i.e., $\bSigma = \bI_p$.
    
    \item[\bf Step 4: Generalized Marchenko-Pastur theorem.] With the nonconventional moment generating function in a form where the dependence on $\bSigma$ is absorbed by the weighting matrix, we apply the generalized Marchenko-Pastur result (see \cite{rubio2011spectral} and Lemma 6 in \cite{BPTJMLR2024}) to obtain the asymptotic behaviour of the nonconventional moment generating function for $p/n \to c > 1$ as $p, n \to \infty$. A similar approach is used for the case $c < 1$, although the function must be represented differently to avoid singularity at zero.
    
    \item[\bf Step 5: Fa\'a di Bruno formula.] The partial derivatives of the limiting nonconventional moment generating function and its uniform convergence yield all the moments. This step involves applying the Fa\'a di Bruno formula (see Section 1.2 in \cite{krantz2002primer}), which introduces the partial exponential Bell polynomials from (4).
    
    \item[\bf Step 6: Playing with Bell polynomials.] As the general formula for the asymptotic moments is complicated, we exploit properties of the partial exponential Bell polynomials to derive the final closed-form expressions. These polynomials are essential to obtain explicit results.
\end{description}

\subsection{Proof of Proposition 2.1}
\begin{proof}
Let $t_1 <t_2$. The application of (5) yields
\begin{equation}\label{app-th2a}
\frac{1}{p}\text{tr}\left[\left(v(t_1)\bSigma+\bI_p\right)^{-1}\right]-\frac{1}{p}\text{tr}\left[\left(v(t_2)\bSigma+\bI_p\right)^{-1}\right]=\frac{t_1v(t_1)-t_2v(t_2)}{c_n}.
\end{equation}

Assume that the statement of the theorem is wrong for the considered $t_1$ and $t_2$, i.e., $v(t_1) \le v(t_2)$. Then, the left hand-side of \eqref{app-th2a} is nonnegative. On the other side, the right hand-side of \eqref{app-th2a} is nonnegative only if
\[\frac{t_1}{t_2} \ge \frac{v(t_2)}{v(t_1)},\]
which contradicts that $v(t_1) \le v(t_2)$. The theorem is proved. 
\end{proof}

\subsection{Proof of Theorem 2.1}
\begin{proof}

Let $\bJ_n=\bI_n-\frac{1}{n}\bi_n \bi_n^\top$. Then, $\bJ_n$ is a projection matrix of rank $n-1$. As such, all $n-1$ nonzero eigenvalues of $\bJ_n$ are equal to one, and its singular value decomposition is expressed as
\begin{equation*}
  \bJ_n=\bH_n\bH_n^\top,  
\end{equation*}
where $\bH_n$ is a $n\times (n-1)$ orthogonal matrix, i.e., $\bH_n^\top\bH_n=\bI_{n-1}$. Then,
\begin{equation}\label{bS-bH}
\bS_n=\frac{1}{n} \bY_n \bJ_n \bY_n^\top
= \frac{1}{n} \tbY_n  \tbY_n^\top
\quad \text{with} \quad
\tbY_n=\bY_n\bH_n.
\end{equation}
Moreover, since $\bH_n$ is orthogonal and the rank of $\bY_n$ is equal to $n$, the application of Sylvester's rank inequality leads to the conclusion that the rank of $\tbY_n$ is $n-1$ (see Section 4.3.3 in \cite{lutkepohl1997handbook}). Hence, we get
\begin{equation}\label{lem1_tbS}
\bS_n^+=\left(\frac{1}{n}\tbY_n\tbY_n^\top\right)^+
=\frac{1}{\sqrt{n}}\tbY_n\left(\frac{1}{n}\tbY_n^\top\tbY_n\right)^{-2}
\frac{1}{\sqrt{n}}\tbY_n^\top
\end{equation}
and, similarly,
\begin{equation}\label{lem1_tbS2}
(\bS_n^+)^m=\frac{1}{\sqrt{n}}\tbY_n\left(\frac{1}{n}\tbY_n^\top\tbY_n\right)^{-(m+1)}\frac{1}{\sqrt{n}}\tbY_n^\top
\quad \text{for} \quad m=2,3,...
\end{equation}

It holds that
\begin{eqnarray*}
&& \text{tr}\left[\frac{1}{\sqrt{n}}\tbY_n\left(\frac{1}{n}\tbY_n^\top\tbY_n\right)^{-(m+1)}\frac{1}{\sqrt{n}}\tbY_n^\top\bTheta\right]\\
&=&\frac{(-1)^m}{m!}\left.\dfrac{\partial^m}{\partial t^m}\text{tr}\left[\frac{1}{\sqrt{n}}\tbY_n\left(\frac{1}{n}\tbY_n^\top\tbY_n+t\bI_n\right)^{-1}\frac{1}{\sqrt{n}}\tbY_n^\top\bTheta\right]\right|_{t=0} 
\end{eqnarray*}
for $m=1,2,...$.

The application of the Woodbury formula (matrix inversion lemma, see, e.g., \cite{hornjohn1985}), leads to
\begin{eqnarray*}
&&
\frac{1}{\sqrt{n}}\tbY_n\left(\frac{1}{n}\tbY_n^\top\tbY_n+t\bI_n\right)^{-1}\frac{1}{\sqrt{n}}\tbY_n^\top=\bI_p-t\left(\frac{1}{n}\tbY_n\tbY_n^\top+t\bI_p\right)^{-1}\\
&=&
\bI_p-t\bSigma^{-1/2}\left(\frac{1}{n}\bX_n\bX_n^\top+t\bSigma^{-1}-\bbx_n\bbx_n^\top\right)^{-1}\bSigma^{-1/2},
\end{eqnarray*}
with $\bbx_n=\frac{1}{n}\bX_n \bOne_n$ where we use that
\begin{equation*}
\frac{1}{n}\tbY_n\tbY_n^\top=\bS_n
=\frac{1}{n}\bSigma^{1/2}\bX_n\bX_n^\top \bSigma^{1/2}-\bSigma^{1/2}\bbx_n\bbx_n^\top \bSigma^{1/2}.
\end{equation*}

Hence,
\begin{eqnarray*}
&&\text{tr}((\bS_n^+)^{m}\bTheta)= \frac{(-1)^{m+1}}{m!}\left.\dfrac{\partial^m}{\partial t^m}\,t\,\text{tr}\left[\left(\frac{1}{n}\bX_n\bX_n^\top+t\bSigma^{-1}\right)^{-1}
\bSigma^{-1/2}\bTheta\bSigma^{-1/2}\right]\right|_{t=0}\\
&+&\frac{(-1)^{m+1}}{m!}\left.\dfrac{\partial^m}{\partial t^m}t\frac{\bbx_n^\top\left(\frac{1}{n}\bX_n\bX_n^\top+t\bSigma^{-1}\right)^{-1} \bSigma^{-1/2}\bTheta\bSigma^{-1/2} \left(\frac{1}{n}\bX_n\bX_n^\top+t\bSigma^{-1}\right)^{-1}\bbx_n}{1-\bbx_n^\top\left(\frac{1}{n}\bX_n\bX_n^\top+t\bSigma^{-1}\right)^{-1}\bbx_n}
\right|_{t=0}\,.
\end{eqnarray*}

Following \citet[Eq. (2.28)]{pan2014comparison}, the quantity
\begin{align*}
    \frac{1}{1 - \bbx_n^\top\left(\frac{1}{n}\bX_n\bX_n^\top+t\bSigma^{-1}\right)^{-1}\bbx_n}
=    \frac{1
        }{1 - \bar\bx_n^\top\bSigma^{1/2}  \left(\frac{1}{n}\bSigma^{1/2}\bX_n\bX_n^\top\bSigma^{1/2} +t\bI_p\right)^{-1} \bSigma^{1/2}\bar\bx_n
        }
\end{align*}
is uniformly bounded in $t$. Moreover, we have that
\begin{equation}\label{eqn:mean_stieltjes_conv_1}
\left |\bxi^\top\left(\frac{1}{n}\bSigma^{1/2}\bX_n\bX_n^\top\bSigma^{1/2}+ t\bI_p\right)^{-1}\bSigma^{1/2}\bar\bx_n
\right|\stackrel{a.s.}{\longrightarrow} 0
\end{equation}
following \citet[p. 673]{pan2014comparison} where $\bxi^\top\bxi$ is assumed to be uniformly bounded in $p$. As a result, if $\bTheta=\sum_{i=1}^k \btheta_i\btheta_i^\top$ with finite $k$, then 
\begin{eqnarray*}
&&\bbx_n^\top\left(\frac{1}{n}\bX_n\bX_n^\top+t\bSigma^{-1}\right)^{-1} \bSigma^{-1/2}\bTheta\bSigma^{-1/2} \left(\frac{1}{n}\bX_n\bX_n^\top+t\bSigma^{-1}\right)^{-1}\bbx_n\\
&=&\sum_{i=1}^k \left(\btheta_i^\top\bSigma^{-1/2} \left(\frac{1}{n}\bX_n\bX_n^\top+t\bSigma^{-1}\right)^{-1}\bbx_n\right)^2\\
&=&\sum_{i=1}^k \left(\btheta_i^\top \left(\frac{1}{n}\bSigma^{1/2}\bX_n\bX_n^\top\bSigma^{1/2}+t\bI_p\right)^{-1}\bSigma^{1/2}\bbx_n\right)^2 \stackrel{a.s.}{\longrightarrow} 0.
\end{eqnarray*}

If the rank of $\bTheta$ is not finite and we assume that $\lambda_{max}(\bTheta)=o(1)$, then we get
\begin{eqnarray*}
&& t \bbx_n^\top\left(\frac{1}{n}\bX_n\bX_n^\top+t\bSigma^{-1}\right)^{-1} \bSigma^{-1/2}\bTheta\bSigma^{-1/2} \left(\frac{1}{n}\bX_n\bX_n^\top+t\bSigma^{-1}\right)^{-1}\bbx_n\\
&\le&\lambda_{max}\left(t\left(\frac{1}{n}\bX_n\bX_n^\top+t\bSigma^{-1}\right)^{-1/2} \bSigma^{-1/2}\bTheta\bSigma^{-1/2} \left(\frac{1}{n}\bX_n\bX_n^\top+t\bSigma^{-1}\right)^{-1/2}\right)\\
&\times&
\bar\bx_n^\top  \left(\frac{1}{n}\bX_n\bX_n^\top +t\bSigma^{-1}\right)^{-1} \bar\bx_n\\
&\le&\lambda_{max}\left(\bTheta\right)
\lambda_{max}\left(t\bSigma^{-1}\left(\frac{1}{n}\bX_n\bX_n^\top+t\bSigma^{-1}\right)^{-1}\right)\bar\bx_n^\top  \left(\frac{1}{n}\bX_n\bX_n^\top +t\bSigma^{-1}\right)^{-1} \bar\bx_n\\
&\le& \lambda_{max}\left(\bTheta\right)
\bar\bx_n^\top  \left(\frac{1}{n}\bX_n\bX_n^\top +t\bSigma^{-1}\right)^{-1} \bar\bx_n
\stackrel{a.s.}{\longrightarrow} 0.
\end{eqnarray*}

Thus,
\begin{eqnarray*}
\text{tr}((\bS_n^+)^m\bTheta)&=& \frac{(-1)^{m+1}}{m!}\left.\dfrac{\partial^m}{\partial t^m}\,t\,\text{tr}\left[\left(\frac{1}{n}\bX_n\bX_n^\top+t\bSigma^{-1}\right)^{-1}
\bSigma^{-1/2}\bTheta\bSigma^{-1/2}\right]\right|_{t=0}\,.
\end{eqnarray*}

Lemma 6 in \cite{BPTJMLR2024} yields

\begin{eqnarray}\label{mainlimit}
  &&  \left|\text{tr}\left[\left(\frac{1}{n}\bX_n\bX_n^\top+t\bSigma^{-1}\right)^{-1}
\bSigma^{-1/2}\bTheta\bSigma^{-1/2}\right]\right.\nonumber\\
&& \left.-\text{tr}\left[\left(\tilde{v}(t) \bI_p + t\bSigma^{-1}\right)^{-1}
\bSigma^{-1/2}\bTheta\bSigma^{-1/2}\right]
\right| \stackrel{a.s.}{\rightarrow} 0, 
\end{eqnarray}
for $p/n \rightarrow c \in [1,\infty)$ as $n \to \infty$ where $\tilde{v}(t)$ satisfies the following equation
\[\frac{1}{\tilde{v}(t)}-1=c_n \frac{1}{p}\text{tr}\left[\bSigma\left(\tilde{v}(t)\bSigma+t\bI_p\right)^{-1}\right]\]
with $c_n=p/n$.

Let  $v(t)=\tilde{v}(t)/t$. Then, $v(t)$ is the solution of the following equation
\begin{equation}\label{app-vt}
\frac{1}{v(t)}-t=c_n \frac{1}{p}\text{tr}\left[\bSigma\left(v(t)\bSigma+\bI_p\right)^{-1}\right],
\end{equation}
which can be rewritten as
\begin{equation*}
1-tv(t)=c_n \frac{1}{p}\text{tr}\left[v(t)\bSigma\left(v(t)\bSigma+\bI_p\right)^{-1}\right]=c_n-c_n\frac{1}{p}\text{tr}\left[\left(v(t)\bSigma+\bI_p\right)^{-1}\right],
\end{equation*}

As a result, $v(0)$ solves the equation
\[\frac{1}{p}\text{tr}\left[\left(v(0)\bSigma+\bI_p\right)^{-1}\right]=\frac{c_n-1}{c_n}.\]
The left hand-side of the equation is a monotonically decreasing function in $v(0)$ taking the values from one to zero, while the right hand-side is equal to a number smaller than one. As a result, there is a unique solution of this equation, which is bounded for the fixed value of $c_n$.

Let 
\begin{equation}
f_1(v)=\frac{1}{v}
\quad \text{and} \quad
f_2(v;\bA)=\frac{1}{p}\text{tr}\left[\left(v\bI_p+\bSigma^{-1}\right)^{-1}\bA\right].
\end{equation}
for a matrix $\bA$. Then, the $k$-th order derivatives of these two functions are expressed as
\begin{equation*}
f_1^{(k)}(v)=(-1)^k k!\frac{1}{v^{k+1}}
\end{equation*}
and
\begin{eqnarray*}
f_2^{(k)}(v;\bA)&=&\frac{(-1)^k k!}{p}\text{tr}\left[\left(v\bI_p+\bSigma^{-1}\right)^{-(k+1)}\bA\right]\\
&=&\frac{(-1)^k k!}{p}\text{tr}\left\{\left(v\bSigma+\bI_p\right)^{-1}\left[\bSigma\left(v\bSigma+\bI_p\right)^{-1}\right]^{k}\bSigma^{1/2}\bA\bSigma^{1/2}\right\}
\end{eqnarray*}
for $k=1,2,...$. Then, the application of the Fa\`{a} di Bruno formula (see Section 1.3 in \cite{krantz2002primer}) leads to
\begin{eqnarray*}
&&\text{tr}((\bS_n^+)^m\bTheta)\stackrel{d.a.s.}{\longrightarrow} \frac{(-1)^{m+1}}{m!}\left.\dfrac{\partial^m}{\partial t^m}\,t\,\text{tr}\left[\left(\tilde{v}(t) \bI_p + t\bSigma^{-1}\right)^{-1}
\bSigma^{-1/2}\bTheta\bSigma^{-1/2}\right]\right|_{t=0}\nonumber\\
&=&\frac{(-1)^{m+1}}{m!}\left.\dfrac{\partial^m}{\partial t^m} \text{tr}\left[\left(v(t) \bI_p + \bSigma^{-1}\right)^{-1}
\bSigma^{-1/2}\bTheta\bSigma^{-1/2}\right]\right|_{t=0}\nonumber\\
&=& \frac{(-1)^{m+1}}{m!}
\sum_{k=1}^m f_2^{(k)}(v(0);\bSigma^{-1/2}\bTheta\bSigma^{-1/2}) B_{m,k}\left(v^{(1)}(0),v^{(2)}(0),...,v^{(m-k+1)}(0)\right)\\
&=&  \sum_{k=1}^m \frac{(-1)^{m+k+1} k!}{m!p}\text{tr}\left\{\left(v(0)\bSigma+\bI_p\right)^{-1}\left[\bSigma\left(v(0)\bSigma+\bI_p\right)^{-1}\right]^{k}\bTheta\right\} B_{m,k}\left(v^{(1)}(0),v^{(2)}(0),...,v^{(m-k+1)}(0)\right)
\end{eqnarray*}
for $p/n\rightarrow c \in (1,\infty)$ as $n\rightarrow\infty$ where the symbol $B_{m,k}\left(v_1,v_2,...,v_{m-k+1}\right)$
denotes the partial exponential Bell polynomials defined in Section 2.

To compute $v^{(1)}(0)$, $v^{(2)}(0)$,...,$v^{(m-k+1)}(0)$, we rewrite \eqref{app-vt} as
\begin{equation}\label{app-vt2}
\frac{1}{v(t)}-t=c_n \frac{1}{p}\text{tr}\left[\left(v(t)\bI_p+\bSigma^{-1}\right)^{-1}\right].
\end{equation}
Hence,
\begin{equation*}
-\frac{v^{(1)}(t)}{v(t)^2}-1=-c_n v^{(1)}(t) \frac{1}{p}\text{tr}\left[\left(v(t)\bI_p+\bSigma^{-1}\right)^{-2}\right]
=-c_n v^{(1)}(t) \frac{1}{p}\text{tr}\left\{\left[\bSigma\left(v(t)\bSigma+\bI_p\right)^{-1}\right]^2\right\},
\end{equation*}
from which we get
\begin{eqnarray*}
v^{(1)}(0)&=&
-\frac{1}{\frac{1}{v(0)^2}-c_n  \frac{1}{p}\text{tr}\left\{\left[\bSigma\left(v(0)\bSigma+\bI_p\right)^{-1}\right]^2\right\}}
\end{eqnarray*}

For the computation of the derivative of $v(\cdot)$ of the order larger than one, we apply the Fa\`{a} di Bruno formula to both the sides of \eqref{app-vt2}. This yields
\begin{eqnarray*}
&&\sum_{k=1}^m f_1^{(k)}(v(t)) B_{m,k}\left(v^{(1)}(t),...,v^{(m-k+1)}(t)\right)=c_n\sum_{k=1}^m f_2^{(k)}(v(t);\bI_p) B_{m,k}\left(v^{(1)}(t),...,v^{(m-k+1)}(t)\right)
\end{eqnarray*}
or
\begin{eqnarray*}
&&\sum_{k=1}^m \frac{(-1)^{k} k!}{v(t)^{k+1}} B_{m,k}\left(v^{(1)}(t),...,v^{(m-k+1)}(t)\right)\\
&=& c_n \sum_{k=1}^m \frac{(-1)^{k} k!}{p}\text{tr}\left\{\left[\bSigma\left(v(t)\bSigma+\bI_p\right)^{-1}\right]^{k+1}\right\} B_{m,k}\left(v^{(1)}(t),...,v^{(m-k+1)}(t)\right)
\end{eqnarray*}

Using that $B_{m,1}\left(v^{(1)}(t),v^{(2)}(t),...,v^{(m)}(t)\right)=v^{(m)}(t)$, we finally get
\begin{eqnarray*}
v^{(m)}(0)=\frac{ \sum_{k=2}^m (-1)^{k} k!\left(\frac{1}{v(0)^{k+1}}-c_n\frac{1}{p}\text{tr}\left\{\left[\bSigma\left(v(0)\bSigma+\bI_p\right)^{-1}\right]^{k+1}\right\}\right) B_{m,k}\left(v^{(1)}(0),...,v^{(m-k+1)}(0)\right)}{\frac{1}{v(0)^{2}}-c_n\frac{1}{p}\text{tr}\left\{\left[\bSigma\left(v(0)\bSigma+\bI_p\right)^{-1}\right]^{2}\right\}}
\end{eqnarray*}
\end{proof}

\subsection{Proof of Corollary 2.1}
\begin{proof}

From the properties of the partial exponential Bell polynomials we get that
\begin{eqnarray*}
B_{1,1}\left(v^{(1)}(0)\right)&=&v^{(1)}(0),\\
B_{2,1}\left(v^{(1)}(0),v^{(2)}(0)\right)&=&v^{(2)}(0),\\
B_{3,1}\left(v^{(1)}(0),v^{(2)}(0),v^{(3)}(0)\right)&=&v^{(3)}(0),\\
B_{4,1}\left(v^{(1)}(0),v^{(2)}(0),v^{(3)}(0),v^{(4)}(0)\right)&=&v^{(4)}(0),\\
B_{2,2}\left(v^{(1)}(0)\right)&=& [v^{(1)}(0)]^2,\\
B_{3,2}\left(v^{(1)}(0),v^{(2)}(0)\right)&=& 3v^{(1)}(0) v^{(2)}(0),\\
B_{3,3}\left(v^{(1)}(0)\right)&=& [v^{(1)}(0)]^3,\\
B_{4,2}\left(v^{(1)}(0),v^{(2)}(0),v^{(3)}(0)\right)&=& 4 v^{(1)}(0) v^{(3)}(0) + 3[v^{(2)}(0)]^2 ,\\
B_{4,3}\left(v^{(1)}(0),v^{(2)}(0)\right)&=& 6[v^{(1)}(0)]^2 v^{(2)}(0),\\
B_{4,4}\left(v^{(1)}(0)\right)&=& [v^{(1)}(0)]^4.
\end{eqnarray*}
The substitution of the above formulas for the partial exponential Bell polynomials to (10) and (13) and the application of (12) leads to the statement of the corollary.
\end{proof}

\subsection{Proof of Corollary 2.2}
\begin{proof}
Let 
\begin{equation*}
g_k=\frac{1}{p}\text{tr}\left\{\left[\bSigma\left(v(0)\bSigma+\bI_p\right)^{-1}\right]^{k}\right\}.
\end{equation*}

Then,
\[d_k\left(\frac{1}{p}\bI_p\right)=\frac{1}{p}\text{tr}\left\{\left(v(0)\bSigma+\bI_p\right)^{-1}\left[\bSigma\left(v(0)\bSigma+\bI_p\right)^{-1}\right]^{k}\bI_p\right\}
=g_k-v(0)g_{k+1}\]
and, consequently,
\begin{eqnarray*}
s_m\left(\frac{1}{p}\bI_p\right)&=& \sum_{k=1}^m \frac{(-1)^{m+k+1} k!}{m!}\left(g_k-v(0)g_{k+1}\right) B_{m,k}\left(v^{(1)}(0),...,v^{(m-k+1)}(0)\right)\\
&=& \frac{(-1)^{m}}{m!}\left(g_1-v(0)g_{2}\right) B_{m,1}\left(v^{(1)}(0),...,v^{(m)}(0)\right)\\
&+&\sum_{k=2}^m \frac{(-1)^{m+k+1} k!}{m!}\left(g_k-v(0)g_{k+1}\right) B_{m,k}\left(v^{(1)}(0),...,v^{(m-k+1)}(0)\right).
\end{eqnarray*}

The application of $B_{m,1}\left(v^{(1)}(0),...,v^{(m)}(0)\right)=v^{(m)}(0)$ and (13) yields
\begin{eqnarray*}
&&s_m\left(\frac{1}{p}\bI_p\right)= \frac{(-1)^{m+1}}{m!}\sum_{k=2}^m (-1)^{k} k!\left(g_1-v(0)g_{2}\right)v^{(1)}(0) h_{k+1} B_{m,k}\left(v^{(1)}(0),...,v^{(m-k+1)}(0)\right)\\
&+&
\frac{(-1)^{m+1}}{m!}\sum_{k=2}^m (-1)^{k} k!\left(g_k-v(0)g_{k+1}\right) B_{m,k}\left(v^{(1)}(0),...,v^{(m-k+1)}(0)\right)\\
&=&\frac{(-1)^{m+1}}{m!}\sum_{k=2}^m (-1)^{k} k!\left[\left(g_1-v(0)g_{2}\right)v^{(1)}(0) h_{k+1}+g_k-v(0)g_{k+1}\right] B_{m,k}\left(v^{(1)}(0),...,v^{(m-k+1)}(0)\right).
\end{eqnarray*}

Using 
\begin{eqnarray*}
 && \left(g_1-v(0)g_{2}\right)v^{(1)}(0) h_{k+1}+g_k-v(0)g_{k+1}\\
&=& \left(\frac{1}{p}\text{tr}\left[\bSigma\left(v(0)\bSigma+\bI_p\right)^{-1}\right]-v(0)g_{2}\right)\frac{-v(0)^2}{1-c_nv(0)^2 g_2} h_{k+1}+g_k-v(0)g_{k+1}
 \\
&=& \left(\frac{1}{v(0)}\left(1-\frac{1}{p}\text{tr}\left[\left(v(0)\bSigma+\bI_p\right)^{-1}\right]\right)-v(0)g_{2}\right)\frac{-v(0)^2}{1-c_nv(0)^2 g_2} h_{k+1}+g_k-v(0)g_{k+1}\\
&=& \left(\frac{1}{v(0)}\left(1-\frac{c_n}{c_n-1}\right)-v(0)g_{2}\right)\frac{-v(0)^2}{1-c_nv(0)^2 g_2} h_{k+1}+g_k-v(0)g_{k+1}\\
&=& \frac{-v(0)}{c_n} h_{k+1}+g_k-v(0)g_{k+1}=\frac{-h_k}{c_n}\\   
\end{eqnarray*}
and the equality (see Equation (1.4) in \cite{cvijovic2011new})
\begin{equation*}
 k!B_{m,k}\left(v^{(1)}(0),...,v^{(m-k+1)}(0)\right) = (k-1)! \sum_{r=1}^{m-1} \frac{m!}{r!(m-r)!} v^{(r)}(0) B_{m-r,k-1}\left(v^{(1)}(0),...,v^{(m-r-k+2)}(0)\right),
\end{equation*}
where $B_{i,k-1}(.)=0$ for $i<k-1$, we get
\begin{eqnarray*}
s_m\left(\frac{1}{p}\bI_p\right)&=& \frac{(-1)^{m}}{m!c_n}\sum_{k=2}^m (-1)^{k} h_k (k-1)! \sum_{r=1}^{m-1} \frac{m!}{r!(m-r)!}v^{(r)}(0) B_{m-r,k-1}\left(v^{(1)}(0),...,v^{(m-r-k)}(0)\right)\\
&=& \frac{(-1)^{m}}{m!c_n} \sum_{r=1}^{m-1} \frac{m!}{r!(m-r)!}v^{(r)}(0) \sum_{k=2}^{m-r+1} (-1)^{k} h_k (k-1)!B_{m-r,k-1}\left(v^{(1)}(0),...,v^{(m-r-k+2)}(0)\right)\\
&=& \frac{(-1)^{m}}{m!c_n} \frac{m!}{(m-1)!}v^{(m-1)}(0) h_2 v^{(1)}(0)+
\frac{(-1)^{m}}{m!c_n} \sum_{r=1}^{m-2} \frac{m!}{r!(m-r)!}v^{(r)}(0)\Bigg( h_2 v^{(m-r)}(0)\\
 &-&\sum_{k=3}^{m-r+1} (-1)^{k-1} h_k (k-1)!B_{m-r,k-1}\left(v^{(1)}(0),...,v^{(m-r-k+2)}(0)\right)\Bigg)\\
&=& \frac{(-1)^{m-1}}{(m-1)!c_n}v^{(m-1)}(0) ,
\end{eqnarray*}
where we used that $v^{(1)}(0)=-h_2^{-1}$ and
\[\sum_{k=3}^{m-r+1} (-1)^{k-1} h_k (k-1)!B_{m-r,k-1}\left(v^{(1)}(0),...,v^{(m-r-k+2)}(0)\right)\Bigg)=-\frac{v^{(m-r)}(0)}{v^{(1)}(0)}.\]
\end{proof}

\subsection{Proof of Corollary 2.3}
\begin{proof}
If $\bSigma=\bI_p$, then 
\[d_k\left(\bTheta\right)=\left(v(0)+1\right)^{-(k+1)}\text{tr}(\bTheta), \quad k=1,2,....\]
and $s_{m}(\bTheta)=a_{m}\text{tr}(\bTheta)$ for $m=1,2,...$ where $a_{m}$ does not depend on $\bTheta$. As such, $a_{m}$ can be computed by setting $\bTheta=\frac{1}{p} \bI_p$ and using the findings of Corollary 2.2. Thus,
\[s_{m}\left(\bTheta\right)=s_{m}\left(\frac{1}{p}\bI_p\right)\text{tr}(\bTheta)=\frac{(-1)^{m-1}v^{(m-1)}(0)}{(m-1)! c_n}\text{tr}(\bTheta),\]
where it holds from (6) that
\[ \frac{1}{v(0)+1}=\frac{c_n-1}{c_n}\]
and, hence,
\begin{equation}\label{cor-v0}
v(0)=\frac{1}{c_n-1}
\quad \text{and} \quad
v^{(1)}(0)=
-\frac{c_n}{(c_n-1)^3},
\end{equation}
where the second equality follows from (12). 

Finally, from (13) we get
\begin{eqnarray*}
v^{(m)}(0)&=& -v^{(1)}(0) \sum_{k=2}^m (-1)^{k} k!\left(v(0)^{-(k+1)}-c_n(v(0)+1)^{-(k+1)}\right) B_{m,k}\left(v^{(1)}(0),...,v^{(m-k+1)}(0)\right)\\
&=& \frac{c_n}{(c_n-1)^3}\sum_{k=2}^m (-1)^{k} k!(c_n-1)^{(k+1)}\left(1-c_n^{-k}\right) B_{m,k}\left(v^{(1)}(0),...,v^{(m-k+1)}(0)\right) .
\end{eqnarray*}
\end{proof}

\subsection{Proof of Theorem 2.2}
\begin{proof}
Following the proof of Theorem 2.1, we get
\begin{eqnarray*}
\text{tr}((\bS_n+t\bI_p)^{-(m+1)}\bTheta)&=& \frac{(-1)^m}{m!}\dfrac{\partial^m}{\partial t^m}\,\text{tr}\left[\left(\frac{1}{n}\bX_n\bX_n^\top+t\bSigma^{-1}\right)^{-1}
\bSigma^{-1/2}\bTheta\bSigma^{-1/2}\right] \,,
\end{eqnarray*}
while the application of Lemma 6 in \cite{BPTJMLR2024} leads to
\begin{eqnarray*}
    \left|\text{tr}\left[\left(\frac{1}{n}\bX_n\bX_n^\top+t\bSigma^{-1}\right)^{-1}
\bSigma^{-1/2}\bTheta\bSigma^{-1/2}\right]
-\frac{1}{t}\text{tr}\left[\left(v(t) \bI_p + \bSigma^{-1}\right)^{-1}
\bSigma^{-1/2}\bTheta\bSigma^{-1/2}\right]
\right| \stackrel{a.s.}{\rightarrow} 0,
\end{eqnarray*}
for $p/n \rightarrow c \in (0,\infty)$ as $n \to \infty$ where $v(t)$ is the solution of \eqref{app-vt}. Hence, 
\begin{eqnarray*}
&&\text{tr}((\bS_n+t\bI_p)^{-(m+1)}\bTheta)\stackrel{d.a.s.}{\longrightarrow}\frac{(-1)^m}{m!}\dfrac{\partial^m}{\partial t^m}
\frac{1}{t}\text{tr}\left[\left(v(t) \bI_p + \bSigma^{-1}\right)^{-1} \bSigma^{-1/2}\bTheta\bSigma^{-1/2}\right]\nonumber\\
&=&\frac{(-1)^m}{m!}\dfrac{\partial^{m}{1}/{t} }{\partial t^{m}} \text{tr}\left[\left(v(t) \bI_p + \bSigma^{-1}\right)^{-1} \bSigma^{-1/2}\bTheta\bSigma^{-1/2}\right]\\
&&+\frac{(-1)^m}{m!} \sum_{l=1}^m \frac{m!}{l!(m-l)!}\dfrac{\partial^{m-l}}{\partial t^{m-l}}\frac{1}{t}
\dfrac{\partial^l}{\partial t^l} \text{tr}\left[\left(v(t) \bI_p + \bSigma^{-1}\right)^{-1} \bSigma^{-1/2}\bTheta\bSigma^{-1/2}\right]
\nonumber\\
&=&  t^{-m-1} \text{tr}\left[\left(v(t) \bI_p + \bSigma^{-1}\right)^{-1} \bSigma^{-1/2}\bTheta\bSigma^{-1/2}\right]\\
&&+\frac{(-1)^{m}}{m!}  \sum_{l=1}^m \frac{m!}{l!(m-l)!} (-1)^{m-l} (m-l)! t^{-(m-l)-1}
\sum_{k=1}^l f_2^{(k)}(v(t);\bSigma^{-1/2}\bTheta\bSigma^{-1/2}) \nonumber\\ 
&&\times B_{l,k}\left(v^{(1)}(t),v^{(2)}(t),...,v^{(l-k+1)}(t)\right)\\
&=&  t^{-m-1} \text{tr}\left[\left(v(t)\bSigma+ \bI_p \right)^{-1} \bTheta\right]\\
&&+ \sum_{l=1}^m t^{-(m-l)-1} \sum_{k=1}^l \frac{(-1)^{l+k} k!}{l!}\text{tr}\left\{\left(v(t)\bSigma+\bI_p\right)^{-1}\left[\bSigma\left(v(t)\bSigma+\bI_p\right)^{-1}\right]^{k}\bTheta\right\} \nonumber\\
&&\times B_{l,k}\left(v^{(1)}(t),v^{(2)}(t),...,v^{(l-k+1)}(t)\right)
\end{eqnarray*}
for $p/n\rightarrow c \in (0,\infty)$ as $n\rightarrow\infty$. Furthermore, $v(t)$ is the solution of \eqref{app-vt} which can be rewritten as (5), 
\begin{eqnarray*}
v^{(1)}(t)&=&
-\frac{1}{\frac{1}{v(t)^2}-c_n  \frac{1}{p}\text{tr}\left\{\left[\bSigma\left(v(t)\bSigma+\bI_p\right)^{-1}\right]^2\right\}},
\end{eqnarray*}
and
\begin{eqnarray*}
v^{(m)}(t)=\frac{ \sum_{k=2}^m (-1)^{k} k!\left(\frac{1}{v(t)^{k+1}}-c_n\frac{1}{p}\text{tr}\left\{\left[\bSigma\left(v(t)\bSigma+\bI_p\right)^{-1}\right]^{k+1}\right\}\right) B_{m,k}\left(v^{(1)}(t),...,v^{(m-k+1)}(t)\right)}{\frac{1}{v(t)^{2}}-c_n\frac{1}{p}\text{tr}\left\{\left[\bSigma\left(v(t)\bSigma+\bI_p\right)^{-1}\right]^{2}\right\}}.
\end{eqnarray*}
\end{proof}

\subsection{Proof of 2.4}
\begin{proof}
The application of Theorem 2.2 leads to
\begin{eqnarray*}
&&\tilde{s}_1(t,\bTheta)= t^{-1} d_0(t,\bTheta),\\
&&\tilde{s}_2(t,\bTheta)= t^{-2} d_0(t,\bTheta) + t^{-1} d_1(t,\bTheta) B_{1,1}\left(v^{(1)}(t)\right),\\
&&\tilde{s}_3(t,\bTheta)= t^{-3} d_0(t,\bTheta) + t^{-2}  d_1(t,\bTheta)B_{1,1}\left(v^{(1)}(t)\right)\\
&&+t^{-1}\left\{-\frac{1}{2}v^{(2)}(t) d_1(t,\bTheta)B_{2,1}\left(v^{(1)}(t),v^{(2)}(t)\right)+d_2(t,\bTheta)B_{2,2}\left(v^{(1)}(t)\right)\right\},\\
&&\tilde{s}_4(t,\bTheta)= t^{-4} d_0(t,\bTheta) + t^{-3}  d_1(t,\bTheta)B_{3,1}\left(v^{(1)}(t)\right)\\
&&+t^{-2}\left\{-\frac{1}{2} d_1(t,\bTheta) B_{2,1}\left(v^{(1)}(t),v^{(2)}(t)\right)+ d_2(t,\bTheta) B_{2,2}\left(v^{(1)}(t)\right)\right\}\\
&&+t^{-1}\left\{\frac{1}{6} d_1(t,\bTheta) B_{3,1}\left(v^{(1)}(t),v^{(2)}(t)v^{(3)}(t)\right)-\frac{1}{3} d_2(t,\bTheta)B_{3,2}\left(v^{(1)}(t),v^{(2)}(t)\right)+d_3(t,\bTheta)B_{3,3}\left(v^{(1)}(t)\right)\right\}.
\end{eqnarray*}
The rest of the proof follows by using the formulas  for the partial exponential Bell polynomials as given in the proof of Corollary 2.1.    
\end{proof}

\subsection{Proof of Corollary 2.5}
\begin{proof}
The proof of Corollary 2.5 is done by using the methods of mathematical induction. For $m=0$, we get
\[\tilde{s}_{1}\left(t,\frac{1}{p}\bI_p\right)=t^{-1}d_k\left(t,\frac{1}{p}\bI_p\right)=t^{-1} \frac{c_n-1}{c_n}+\frac{v(t)}{c_n},\]
where the last equality follows from (5) and (22).

Next, we assume that the statement of the corollary is true for $\tilde{s}_{1}\left(t,\frac{1}{p}\bI_p\right), ... \tilde{s}_{m}\left(t,\frac{1}{p}\bI_p\right)$ and prove it for $\tilde{s}_{m+1}\left(t,\frac{1}{p}\bI_p\right)$. Let 
\begin{equation*}
g_k(t)=\frac{1}{p}\text{tr}\left\{\left[\bSigma\left(v(t)\bSigma+\bI_p\right)^{-1}\right]^{k}\right\}.
\end{equation*}
Then,
\[d_k\left(t,\frac{1}{p}\bI_p\right)=\frac{1}{p}\text{tr}\left\{\left(v(t)\bSigma+\bI_p\right)^{-1}\left[\bSigma\left(v(t)\bSigma+\bI_p\right)^{-1}\right]^{k}\bI_p\right\}
=g_k(t)-v(t)g_{k+1}(t)\]
and $h_k(t)=v(t)^{-k}-c_n g_k(t)$. Moreover, from the definition of $v^{(1)}(t)$ we have
\[v^{(1)}(t) h_2(t)=-1\]
and
\begin{equation*}
v^{(1)}(t) (g_1(t)-v(t)g_{2}(t))= \frac{-1}{v(t)^{-2}-c_n g_2(t)} \left(\frac{1-tv(t)}{c_n v(t)} - v(t)g_{2}(t)\right) =-\frac{v(t)}{c_n}-\frac{tv^{(1)}(t)}{c_n}.
\end{equation*}

Using the recursion (26) and the assumption of the induction, we get
\begin{eqnarray*}
\tilde{s}_{m+1}\left(t,\frac{1}{p}\bI_p\right)&=& t^{-(m+1)}\frac{c_n-1}{c_n} +t^{-1} \frac{(-1)^{m-1}v^{(m-1)}(t)}{(m-1)! c_n}\\
&+&t^{-1}\Bigg(\frac{(-1)^{m-1}}{m!}\left(g_1(t)-v(t)g_{2}(t)\right) B_{m,1}\left(v^{(1)}(t),...,v^{(m)}(t)\right)\\
&+&\sum_{k=2}^m \frac{(-1)^{m+k} k!}{m!}\left(g_k(t)-v(t)g_{k+1}(t)\right) B_{m,k}\left(v^{(1)}(t),...,v^{(m-k+1)}(t)\right)\Bigg)\\
&=&  t^{-(m+1)}\frac{c_n-1}{c_n} +t^{-1} \frac{(-1)^{m-1}v^{(m-1)}(t)}{(m-1)! c_n}\\
&+&\frac{(-1)^{m}}{m!t}\sum_{k=2}^m (-1)^{k} k!\left[\left(g_1(t)-v(t)g_2(t)\right)v^{(1)}(t) h_{k+1}(t)+g_k(t)-v(t)g_{k+1}(t)\right]\\
&& \times B_{m,k}\left(v^{(1)}(t),...,v^{(m-k+1)}(t)\right),
\end{eqnarray*}
following the proof of Corollary 2.2 where 
\begin{eqnarray*}
 && \left(g_1(t)-v(t)g_2(t)\right)v^{(1)}(t) h_{k+1}(t)+g_k(t)-v(t)g_{k+1}(t)\\
&=& \left(-\frac{v(t)}{c_n}-\frac{tv^{(1)}(t)}{c_n}\right) h_{k+1}(t)+g_k(t)-v(t)g_{k+1}(t)=-\frac{h_k(t)}{c_n}-\frac{tv^{(1)}(t)}{c_n} h_{k+1}(t).
\end{eqnarray*}

Hence,
\begin{eqnarray*}
\tilde{s}_{m+1}\left(t,\frac{1}{p}\bI_p\right)&=&  t^{-(m+1)}\frac{c_n-1}{c_n} +t^{-1} \frac{(-1)^{m-1}v^{(m-1)}(t)}{(m-1)! c_n}\\
&+&t^{-1}\frac{(-1)^{m-1}}{m! c_n}\sum_{k=2}^m (-1)^{k} k! h_k(t) B_{m,k}\left(v^{(1)}(t),...,v^{(m-k+1)}(t)\right)\\
&+&\frac{(-1)^{m}}{m!c_n} (-1)v^{(1)}(t) \sum_{k=2}^m (-1)^{k} k! h_{k+1}(t) B_{m,k}\left(v^{(1)}(t),...,v^{(m-k+1)}(t)\right)\\
&=&t^{-(m+1)}\frac{c_n-1}{c_n} +t^{-1} \frac{(-1)^{m}v^{(m)}(t)}{m! c_n},
\end{eqnarray*}
following the end of the proof of Corollary 2.2.
\end{proof}

\subsection{Proof of Corollary 2.6}
\begin{proof}
If $\bSigma=\bI_p$, then
\[d_k\left(t,\bTheta\right)=\left(v(t)+1\right)^{-(k+1)}\text{tr}(\bTheta), \quad k=0,1,2,....\]
As a result, $\tilde{s}_{m+1}(t,\bTheta)=a_{m+1}(t)\text{tr}(\bTheta)$ for $m=0,1,...$ where $a_{m+1}(t)$, $m=0,1,...$, does not depend on $\bTheta$. Consequently, $a_{m+1}(t)$, $m=0,1,...$, can be computed by setting $\bTheta=\frac{1}{p} \bI_p$ and using the results of Corollary 2.5. It leads to
\[\tilde{s}_{m+1}\left(t,\bTheta\right)=\tilde{s}_{m+1}\left(t,\frac{1}{p}\bI_p\right)\text{tr}(\bTheta)=\left( t^{-(m+1)}\frac{c_n-1}{c_n}+ \frac{(-1)^{m}v^{(m)}(t)}{(m)! c_n} \right)\text{tr}(\bTheta),\]
where $v(t)$ is the solution of
\[\frac{1}{v(t)+1}=\frac{c_n-1+tv(t)}{c_n}\]
and, hence,
\[v(t)=\frac{-(c_n-1+t)+\sqrt{(c_n-1+t)^2+4t}}{2t}=\frac{2}{(c_n-1+t)+\sqrt{(c_n-1+t)^2+4t}}.\]

Moreover, it holds from (23) and (24) that
\[v^{(1)}(t)=\frac{-1}{v(t)^{-2}-c_n(v(t)+1)^{-2}}\]
and
\[v^{(m)}(t)= -v^{(1)}(t) \sum_{k=2}^m (-1)^{k} k!\left(v(t)^{-(k+1)}-c_n(v(t)+1)^{-(k+1)}\right) B_{m,k}\left(v^{(1)}(t),...,v^{(m-k+1)}(t)\right).\]
\end{proof}

\subsection{Proof of Corollary 2.7}
\begin{proof}
From the definition of $D_m(t,\bTheta)$ and the recursion (26), we get
\[\tilde{s}_{m+1}\left(t,\bTheta\right)=\frac{1}{t}\tilde{s}_{m}\left(t,\bTheta\right)+\frac{1}{t}D_m(t,\bTheta)\]
and, similarly,
\[\tilde{s}_{m+k}\left(t,\bTheta\right)=t^{-k}\tilde{s}_{m}\left(t,\bTheta\right)+t^{-k}D_m(t,\bTheta)+\sum_{j=1}^{k-1}t^{-(k-j)}D_{m+j}(t,\bTheta) ~~\text{for $k \ge 2$}.\]

Hence,
\begin{eqnarray*}
&&\grave{s}_{m}\left(t,\bTheta\right)= \tilde{s}_{m}\left(t,\bTheta\right) +  \sum\limits_{k=1}^m (-1)^{k} t^k\binom{m}{k}\left(t^{-k}\tilde{s}_{m}\left(t,\bTheta\right)+t^{-k}D_m(t,\bTheta)\right)\\
 &+&
 \sum\limits_{k=2}^m (-1)^{k}t^k\binom{m}{k}
 \sum\limits_{j=1}^{k-1} t^{-(k-j)} D_{m+j}(t,\bTheta)\\
&=& \left(\tilde{s}_{m}\left(t,\bTheta\right)+D_m(t,\bTheta)\right)\sum\limits_{k=0}^m (-1)^{k} \binom{m}{k} -D_m(t,\bTheta) +\sum\limits_{k=2}^m (-1)^{k}\binom{m}{k}
 \sum\limits_{j=1}^{k-1} t^{j} D_{m+j}(t,\bTheta)\,.
\end{eqnarray*}
Since $\sum\limits_{k=0}^m (-1)^{k}\binom{m}{k}=(1-1)^m=0$, we get the statement of the corollary.
\end{proof}

\subsection{Proof of Corollary 2.9}
\begin{proof}
The application of Corollary 2.5 for $m=1, 2, ...$ leads to
  \begin{eqnarray*}
      \grave{s}_{m}\left(t,\frac{1}{p}\bI_p\right)&=& \sum\limits_{k=0}^m (-1)^{k}t^k\binom{m}{k}t^{-(m+k)}\frac{c_n-1}{c_n}+ \sum\limits_{k=0}^m (-1)^{k}t^k\binom{m}{k}\frac{(-1)^{m+k-1}v^{(m+k-1)}(t)}{(m+k-1)! c_n}\\
      &=& \frac{c_n-1}{c_n}t^{-m}\underbrace{\sum\limits_{k=0}^m (-1)^{k}\binom{m}{k}}_{(1-1)^m=0}+ \frac{(-1)^{m-1}}{c_n}\sum\limits_{k=0}^m \binom{m}{k}\frac{v^{(m+k-1)}(t)}{(m+k-1)!}t^k\,
  \end{eqnarray*}
and the statement of the corollary follows.
\end{proof}

\subsection{Proof of Theorem 3.1}
\begin{proof} The loss function $L^2_{F;n}$ from (42) can be rewritten in the following way
{\small
\begin{eqnarray}\label{risk2}
L^2_{F;n}&=&||\widehat{\boldsymbol{\Pi}}_{GSE}\bSigma-\bI_p||_F^2=\tr\left[(\widehat{\boldsymbol{\Pi}}_{GSE}\bSigma-\bI_p)(\widehat{\boldsymbol{\Pi}}_{GSE}\bSigma-\bI_p)^\top\right] \nonumber\\
&=& ||\bI_p||^2_F+ ||\bS_n^{\#}(t)\bSigma||^2_F\left(\ta_n+\frac{\tb_n \text{tr}(\bS_n^{\#}(t)\bSigma^2\boldsymbol{\Pi}_0) - \text{tr}(\bS_n^{\#}(t)\bSigma)}{||\bS_n^{\#}(t)\bSigma||^2_F} \right)^2  \nonumber\\
&+&\frac{||\bS_n^{\#}(t)\bSigma||^2_F||\boldsymbol{\Pi}_0\bSigma||^2_F - \left(\text{tr}(\bS_n^{\#}(t)\bSigma^2\boldsymbol{\Pi}_0)\right)^2}{||\bS_n^{\#}(t)\bSigma||^2_F }\nonumber\\
&\times&\left(\tb_n +\frac{\text{tr}(\bS_n^{\#}(t)\bSigma^2\boldsymbol{\Pi}_0)\text{tr}(\bS_n^{\#}(t)\bSigma) - ||\bS_n^{\#}(t)\bSigma||^2_F \text{tr}(\bSigma\boldsymbol{\Pi}_0) }{||\bS_n^{\#}(t)\bSigma||^2_F||\boldsymbol{\Pi}_0\bSigma||^2_F - \left(\text{tr}(\bS_n^{\#}(t)\bSigma^2\boldsymbol{\Pi}_0)\right)^2} \right)^2 \nonumber\\
&-&\left( \frac{[\text{tr}(\bS_n^{\#}(t)\bSigma)]^2}{||\bS_n^{\#}(t)\bSigma||^2_F}+\frac{\left[\text{tr}(\bS_n^{\#}(t)\bSigma^2\boldsymbol{\Pi}_0)\text{tr}(\bS_n^{\#}(t)\bSigma) - ||\bS_n^{\#}(t)\bSigma||^2_F \text{tr}(\bSigma\boldsymbol{\Pi}_0)\right]^2}{||\bS_n^{\#}(t)\bSigma||^2_F\left[||\bS_n^{\#}(t)\bSigma||^2_F||\boldsymbol{\Pi}_0\bSigma||^2_F - \left(\text{tr}(\bS_n^{\#}(t)\bSigma^2\boldsymbol{\Pi}_0)\right)^2\right]} \right) \,.
\end{eqnarray}
}
Since $||\bS_n^{\#}(t)\bSigma||^2_F >0$ and $||\bS_n^{\#}(t)\bSigma||^2_F||\boldsymbol{\Pi}_0\bSigma||^2_F - \left(\text{tr}(\bS_n^{\#}(t)\bSigma^2\boldsymbol{\Pi}_0)\right)^2\ge0$ with equality to zero if and only if $\bS_n^{\#}(t)=\boldsymbol{\Pi}_0$, then, for a fixed value of the tuning parameter $t$, the loss function $L^2_{F;n}$ is minimized at
{\small
\begin{eqnarray}\label{beta-gen_1}
&&\beta_n^*(\bS_n^{\#}(t))=
\frac{||\bS_n^{\#}(t)\bSigma||^2_F \text{tr}(\bSigma\boldsymbol{\Pi}_0)-\text{tr}(\bS_n^{\#}(t)\bSigma^2\boldsymbol{\Pi}_0)\text{tr}(\bS_n^{\#}(t)\bSigma)  }{||\bS_n^{\#}(t)\bSigma||^2_F||\boldsymbol{\Pi}_0\bSigma||^2_F - \left(\text{tr}(\bS_n^{\#}(t)\bSigma^2\boldsymbol{\Pi}_0)\right)^2}
\nonumber\\
&=&\dfrac{\text{tr}\left(\dfrac{\bSigma\boldsymbol{\Pi}_0}{||\bSigma\boldsymbol{\Pi}_0||_F}\right)||\bS_n^{\#}(t)\bSigma||^2_F
 -\text{tr}(\bS_n^{\#}(t)\bSigma)\text{tr}\left(\bS_n^{\#}(t)\dfrac{\bSigma^2\boldsymbol{\Pi}_0}{||\bSigma\boldsymbol{\Pi}_0||_F}\right)}{||\bS_n^{\#}(t)\bSigma||^2_F
 -\Bigl[\text{tr}\left(\bS_n^{\#}(t)\dfrac{\bSigma^2\boldsymbol{\Pi}_0}{||\bSigma\boldsymbol{\Pi}_0||_F}\right)\Bigr]^2}||\bSigma\boldsymbol{\Pi}_0||^{-1}_F\,.
 \end{eqnarray}
and
\begin{eqnarray}\label{alfa-gen_1} 
&&\alpha_n^*(\bS_n^{\#}(t))=
\dfrac{\text{tr}(\bS_n^{\#}(t)\bSigma)-\text{tr}\left(\dfrac{\bSigma\boldsymbol{\Pi}_0}{||\bSigma\boldsymbol{\Pi}_0||_F}\right)  \text{tr}\left(\bS_n^{\#}(t)\dfrac{{\bSigma^2}\boldsymbol{\Pi}_0}{||\bSigma\boldsymbol{\Pi}_0||_F}\right)}{||\bS_n^{\#}(t)\bSigma||^2_F
  -\Bigl[\text{tr}\left(\bS_n^{\#}(t)\dfrac{{\bSigma^2}\boldsymbol{\Pi}_0}{||\bSigma\boldsymbol{\Pi}_0||_F}\right)\Bigr]^2} \,,
 \end{eqnarray}
}
When $\bS_n^{\#}(t) \in \{\bS_n^{-}(t),\bS_n^{\pm}(t)\}$, the last summand in \eqref{risk2} should be minimized with respect to the tuning parameter $t$. The last row in \eqref{risk2} can be written by
{\footnotesize
\begin{eqnarray*}
&&\frac{[\text{tr}(\bS_n^{\#}(t)\bSigma)]^2}{||\bS_n^{\#}(t)\bSigma||^2_F}+\frac{[\text{tr}(\bS_n^{\#}(t)\bSigma^2\boldsymbol{\Pi}_0)\text{tr}(\bS_n^{\#}(t)\bSigma) - ||\bS_n^{\#}(t)\bSigma||^2_F \text{tr}(\bSigma\boldsymbol{\Pi}_0)]^2}{||\bS_n^{\#}(t)\bSigma||^2_F\left[||\bS_n^{\#}(t)\bSigma||^2_F||\boldsymbol{\Pi}_0\bSigma||^2_F - \left(\text{tr}(\bS_n^{\#}(t)\bSigma^2\boldsymbol{\Pi}_0)\right)^2\right]}
\nonumber\\ 
&=&\frac{\left[\text{tr}(\bS_n^{\#}(t)\bSigma)]^2||\boldsymbol{\Pi}_0\bSigma||^2_F
- 2\text{tr}(\bS_n^{\#}(t)\bSigma^2\boldsymbol{\Pi}_0)\text{tr}(\bS_n^{\#}(t)\bSigma) \text{tr}(\bSigma\boldsymbol{\Pi}_0)
+||\bS_n^{\#}(t)\bSigma||^2_F [\text{tr}(\bSigma\boldsymbol{\Pi}_0)\right]^2}{||\bS_n^{\#}(t)\bSigma||^2_F||\boldsymbol{\Pi}_0\bSigma||^2_F - \left(\text{tr}(\bS_n^{\#}(t)\bSigma^2\boldsymbol{\Pi}_0)\right)^2}.\nonumber\\
&=&L^2_{F;n,2}(\bS_n^{\#}(t))+\left[\text{tr}\left(\dfrac{\bSigma\boldsymbol{\Pi}_0}{||\bSigma\boldsymbol{\Pi}_0||_F}\right)\right]^2
\end{eqnarray*}
with
\begin{eqnarray}\label{risk-L-F2_1}
L^2_{F;n,2}(\bS_n^{\#}(t))
&=& \frac{\left[\text{tr}(\bS_n^{\#}(t)\bSigma)
- \text{tr}\left(\bS_n^{\#}(t)\dfrac{{\bSigma^2}\boldsymbol{\Pi}_0}{||\bSigma\boldsymbol{\Pi}_0||_F}\right)\text{tr}\left(\dfrac{\bSigma\boldsymbol{\Pi}_0}{||\bSigma\boldsymbol{\Pi}_0||_F}\right)\right]^2
}{||\bS_n^{\#}(t)\bSigma||^2_F - \Bigl[\text{tr}\left(\bS_n^{\#}(t)\dfrac{{\bSigma^2}\boldsymbol{\Pi}_0}{||\bSigma\boldsymbol{\Pi}_0||_F}\right)\Bigr]^2} .
\end{eqnarray}
}
Then the optimal value of $t$ is found by maximizing $L^2_{F;n,2}(\bS_n^{\#}(t))$. This finishes the proof of the theorem.
\end{proof}

\subsection{Proof of Theorem 3.4}
\begin{proof}
The out-of-sample variance, or the loss function $L(\alpha, t)=\widehat{\mathbf{w}}_{GSE}^\top\bSigma \widehat{\mathbf{w}}_{GSE}$, can be simplified to
\begin{align}\label{Li-lambda}
L(\alpha, t)
    = & \left(\alpha    \mathbf{w}_{\bS_n^{\#}(t)} + (1-\alpha) \bb \right)^\top\bSigma
    \left(\alpha     \mathbf{w}_{\bS_n^{\#}(t)} + (1-\alpha) \bb \right) \nonumber \\
    = & 
    \left(\bb- \mathbf{w}_{\bS_n^{\#}(t)}\right)^\top \bSigma\left(\bb- \mathbf{w}_{\bS_n^{\#}(t)}\right)  
    \left(
        \alpha- \frac{\bb^\top \bSigma \left(\bb- \mathbf{w}_{\bS_n^{\#}(t)}\right)} {\left(\bb- \mathbf{w}_{\bS_n^{\#}(t)}\right)^\top \bSigma\left(\bb- \mathbf{w}_{\bS_n^{\#}(t)}\right)}
    \right)^2 
    \nonumber\\
    &-\frac{\left(\bb^\top \bSigma \left(\bb- \mathbf{w}_{\bS_n^{\#}(t)}\right)\right)^2}
    {\left(\bb- \mathbf{w}_{\bS_n^{\#}(t)}\right)^\top \bSigma\left(\bb- \mathbf{w}_{\bS_n^{\#}(t)}\right)}
    +\bb^\top\bSigma \bb.
 \end{align}
Now, the result of Theorem 3.4 follows immediately.
\end{proof}
\section{Weighted moments of the ordinary inverse of the sample covariance matrix}\label{sec:main-sample_prec}

In this section, we provide the results for the weighted trace moments of the ordinary inverse of the sample covariance matrix when $p<n$. The derivation is based on the asymptotic properties of the ridge-type estimator of the inverse covariance matrix expressed as
\begin{equation*}
\bS_n^{-}(t)=(\bS_n+t\bI_p)^{-1} \quad \text{for} \quad t \ge 0,
\end{equation*}
by considering the special case of $t=0$ when $c_n=p/n<1$. In this case, the sample covariance matrix $\bS_n$ is non-singular and its ordinary inverse is a well-defined matrix. However, the results of Theorem 2.2 cannot be directly used when $p<n$ and $t=0$. An alternative representation of the asymptotic behavior of $\text{tr}( (\bS_n^{-}(t))^{m+1}\bTheta)$ is deduced in Theorem \ref{th3-ridge}.

\begin{theorem}\label{th3-ridge}
Let $\bY_n$ fulfill the stochastic representation (1). Then, under Assumptions \textbf{(A1)}-\textbf{(A3)}, it holds that
\begin{equation}\label{th3_eq1-ridge}
\left|\emph{tr}( (\bS_n^{-}(t))^{m+1}\bTheta)-\tilde{r}_{m+1}(t,\bTheta)\right| \stackrel{a.s.}{\rightarrow} 0\quad\text{for} \quad p/n \rightarrow c \in (0, 1) 
\quad \text{as} \quad n \rightarrow \infty
\end{equation}
where 
\begin{eqnarray}\label{th3-trm-1}
 &&\tilde{r}_{1}\left(t,\bTheta\right)=\frac{1}{w(t)}\emph{tr}\left[\left(\tilde{w}(t) \bI_p + \bSigma\right)^{-1}\bTheta\right],\\
    \tilde{r}_{m+1}\left(t,\bTheta\right)&=& \frac{(-1)^m}{m!}\tilde{d}_0\left(t,\bTheta\right)\sum_{k=1}^m \frac{(-1)^{k} k!}{w(t)^{k+1}} B_{m,k}\left(w^{(1)}(t),...,w^{(m-k+1)}(t)\right)\nonumber\\
&&+\frac{(-1)^m}{m!}\frac{1}{w(t)} \sum_{k=1}^m (-1)^{k} k!\tilde{d}_k\left(t,\bTheta\right)B_{m,k}\left(\tilde{w}^{(1)}(t),...,\tilde{w}^{(m-k+1)}(t)\right)
\nonumber\\
&&+\frac{(-1)^m}{m!}  \sum_{l=1}^{m-1} \frac{m!}{l!(m-l)!}\sum_{k_1=1}^{m-l} \frac{(-1)^{k_1} k_1!}{w(t)^{k_1+1}} B_{m-l,k_1}\left(w^{(1)}(t),...,w^{(m-l-k_1+1)}(t)\right) \nonumber\\
&& \qquad \qquad \times \sum_{k_2=1}^l (-1)^{k_2} k_2!\tilde{d}_{k_2}\left(t,\bTheta\right) B_{l,k_2}\left(\tilde{w}^{(1)}(t),...,\tilde{w}^{(l-k_2+1)}(t)\right),\label{th3-trm}
\end{eqnarray}
with
\begin{eqnarray}\label{th3-dk-ridge}
\tilde{d}_k\left(t,\bTheta\right)&=&\emph{tr}\left[\left(\tilde{w}(t)\bI_p+\bSigma\right)^{-(k+1)}\bTheta\right], \quad k=0,1,2,...,
\\
\tilde{w}(t)&=&\frac{t}{w(t)}, \quad \tilde{w}^{(1)}(t)=\frac{1}{w(t)}-t\frac{w^{(1)}(t)}{[w(t)]^2} ,\label{th3-wtpr}\\
\tilde{w}^{(m)}(t)&=&t \sum_{k=1}^m \frac{(-1)^{k} k!}{w(t)^{k+1}} B_{m,k}\left(w^{(1)}(t),...,w^{(m-k+1)}(t)\right) \label{th3-wtpr-m}\\
&&+ m \sum_{k=1}^{m-1} \frac{(-1)^{k} k!}{w(t)^{k+1}} B_{m-1,k}\left(w^{(1)}(t),...,w^{(m-k)}(t)\right),\nonumber
\end{eqnarray}
for $m=2,3,...$ where $w(t)$ is the unique solution of 
\begin{equation}\label{th3-wt} 
1-w(t)=c_n\frac{1}{p}\emph{tr}\left[\left(\tilde{w}(t)\bI_p+\bSigma\right)^{-1}\bSigma\right]
\end{equation}
with its derivatives satisfy
\begin{equation}\label{th3-wpr}
w^{(1)}(t)=\frac{c_n\tilde{d}_1\left(t,\frac{1}{p}\bSigma\right)w(t)}{w(t)^2+t c_n \tilde{d}_1\left(t,\frac{1}{p}\bSigma\right)}, 
\end{equation}
and $w^{(m)}(0)$ for $m=2,3,...$ is computed recursively by
\begin{eqnarray}\label{th3-wpr-m}
w^{(m)}(t)&=& \frac{w(t)^2 c_n}{w(t)^2+t c_n \tilde{d}_1\left(t,\frac{1}{p}\bSigma\right)} \Bigg(\tilde{d}_1\left(t,\frac{1}{p}\bSigma\right) t \sum_{k=2}^m \frac{(-1)^{k} k!}{w(t)^{k+1}} B_{m,k}\left(w^{(1)}(t),...,w^{(m-k+1)}(t)\right)\nonumber\\
&&+\tilde{d}_1\left(t,\frac{1}{p}\bSigma\right) m \sum_{k=1}^{m-1} \frac{(-1)^{k} k!}{w(t)^{k+1}} B_{m-1,k}\left(w^{(1)}(t),...,w^{(m-k)}(t)\right)\nonumber\\
&&-
\sum_{k=2}^m (-1)^{k} k! \tilde{d}_k\left(t,\frac{1}{p}\bSigma\right) B_{m,k}\left(\tilde{w}^{(1)}(t),...,\tilde{w}^{(m-k+1)}(t)\right)
\Bigg).
\end{eqnarray}
\end{theorem}

\begin{proof}[Proof of Theorem \ref{th3-ridge}:]
Following the proof of Theorem 2.1, we get
\begin{eqnarray*}
\text{tr}((\bS_n+t\bI_p)^{-(m+1)}\bTheta)&=& \frac{(-1)^m}{m!}\dfrac{\partial^m}{\partial t^m}\,\text{tr}\left[\left(\frac{1}{n}\bX_n\bX_n^\top+t\bSigma^{-1}\right)^{-1}
\bSigma^{-1/2}\bTheta\bSigma^{-1/2}\right] \,,
\end{eqnarray*}
where the application of Lemma 6 in \cite{BPTJMLR2024} yields
\begin{eqnarray*}
    \left|\text{tr}\left[\left(\frac{1}{n}\bX_n\bX_n^\top+t\bSigma^{-1}\right)^{-1}\bSigma^{-1/2}\bTheta\bSigma^{-1/2}\right]
- \frac{1}{w(t)}\text{tr}\left[\left(\frac{t}{w(t)} \bI_p + \bSigma\right)^{-1}\bTheta\right]\right| \stackrel{a.s.}{\rightarrow} 0,
\end{eqnarray*}
for $p/n \rightarrow c \in (0,1)$ as $n \to \infty$ with $w(t)$ being the solution of \eqref{th3-wt}. 

Let $\tilde{w}(t)=t/w(t)$. Then, it holds that 
\begin{eqnarray*}
&&\text{tr}((\bS_n+t\bI_p)^{-(m+1)}\bTheta)\stackrel{d.a.s.}{\longrightarrow}\frac{(-1)^m}{m!}\dfrac{\partial^m}{\partial t^m}
\frac{1}{w(t)}\text{tr}\left[\left(\tilde{w}(t) \bI_p + \bSigma\right)^{-1}
\bTheta\right]\nonumber\\
&=& \frac{(-1)^m}{m!}\dfrac{\partial^m}{\partial t^m} \left(\frac{1}{w(t)}\right)\text{tr}\left[\left(\tilde{w}(t) \bI_p + \bSigma\right)^{-1}\bTheta\right]+\frac{(-1)^m}{m!}\frac{1}{w(t)}\dfrac{\partial^m}{\partial t^m}\left(\text{tr}\left[\left(\tilde{w}(t) \bI_p + \bSigma\right)^{-1}\bTheta\right]\right)\nonumber\\
&&+\frac{(-1)^m}{m!}  \sum_{l=1}^{m-1} \frac{m!}{l!(m-l)!}\dfrac{\partial^{m-l}}{\partial t^{m-l}}\frac{1}{w(t)}
\dfrac{\partial^l}{\partial t^l} \text{tr}\left[\left(\tilde{w}(t) \bI_p + \bSigma\right)^{-1}\bTheta\right] \nonumber\\
&=& \frac{(-1)^m}{m!}\text{tr}\left[\left(\tilde{w}(t) \bI_p + \bSigma\right)^{-1}\bTheta\right]\sum_{k=1}^m \frac{(-1)^{k} k!}{w(t)^{k+1}} B_{m,k}\left(w^{(1)}(t),...,w^{(m-k+1)}(t)\right)\nonumber\\
&&+\frac{(-1)^m}{m!}\frac{1}{w(t)} \sum_{k=1}^m (-1)^{k} k!\text{tr}\left[\left(\tilde{w}(t) \bI_p + \bSigma\right)^{-(k+1)}\bTheta\right] B_{m,k}\left(\tilde{w}^{(1)}(t),...,\tilde{w}^{(m-k+1)}(t)\right)
\nonumber\\
&&+\frac{(-1)^m}{m!}  \sum_{l=1}^{m-1} \frac{m!}{l!(m-l)!}\sum_{k_1=1}^{m-l} \frac{(-1)^{k_1} k_1!}{w(t)^{k_1+1}} B_{m-l,k_1}\left(w^{(1)}(t),...,w^{(m-l-k_1+1)}(t)\right) \nonumber\\
&& \qquad \qquad \times \sum_{k_2=1}^l (-1)^{k_2} k_2!\text{tr}\left[\left(\tilde{w}(t) \bI_p + \bSigma\right)^{-(k_2+1)}\bTheta\right] B_{l,k_2}\left(\tilde{w}^{(1)}(t),...,\tilde{w}^{(l-k_2+1)}(t)\right),
\end{eqnarray*}
where
\begin{eqnarray*}
\tilde{w}^{(1)}(t)&=&\frac{1}{w(t)}-t\frac{w^{(1)}(t)}{[w(t)]^2}, \\ 
\tilde{w}^{(m)}(t)&=&t \dfrac{\partial^m}{\partial t^m} \left(\frac{1}{w(t)}\right) + m \dfrac{\partial^{m-1}}{\partial t^{m-1}} \left(\frac{1}{w(t)}\right)\\
&& \hspace{-2cm}= t \sum_{k=1}^m \frac{(-1)^{k} k!}{w(t)^{k+1}} B_{m,k}\left(w^{(1)}(t),...,w^{(m-k+1)}(t)\right) + m \sum_{k=1}^{m-1} \frac{(-1)^{k} k!}{w(t)^{k+1}} B_{m-1,k}\left(w^{(1)}(t),...,w^{(m-k)}(t)\right).
\end{eqnarray*}

Finally, \eqref{th3-wt} yields
\begin{eqnarray*}
-w^{(1)}(t)&=& -c_n\frac{1}{p}\text{tr}\left[\left(\tilde{w}(t)\bI_p+\bSigma\right)^{-2}\bSigma\right]\left(\frac{1}{w(t)}-t\frac{w^{(1)}(t)}{[w(t)]^2}\right),\\
-w^{(m)}(t)&=& c_n \sum_{k=1}^m \frac{(-1)^{k} k!}{p}\text{tr}\left[\left(\tilde{w}(t) \bI_p+\bSigma\right)^{-(k+1)}\bSigma\right] B_{m,k}\left(\tilde{w}^{(1)}(t),...,\tilde{w}^{(m-k+1)}(t)\right)\\
&=&  c_n \sum_{k=2}^m \frac{(-1)^{k} k!}{p}\text{tr}\left[\left( \tilde{w}(t) \bI_p+\bSigma\right)^{-(k+1)}\bSigma\right] B_{m,k}\left(\tilde{w}^{(1)}(t),...,\tilde{w}^{(m-k+1)}(t)\right)\\
&&- c_n  \frac{1}{p}\text{tr}\left[\left( \tilde{w}(t) \bI_p+\bSigma\right)^{-2}\bSigma\right] \tilde{w}^{(m)}(t) \\
&=& c_n \sum_{k=2}^m \frac{(-1)^{k} k!}{p}\text{tr}\left[\left(\tilde{w}(t) \bI_p+\bSigma\right)^{-(k+1)}\bSigma\right] B_{m,k}\left(\tilde{w}^{(1)}(t),...,\tilde{w}^{(m-k+1)}(t)\right)\\
&&- c_n  \frac{1}{p}\text{tr}\left[\left( \tilde{w}(t) \bI_p+\bSigma\right)^{-2}\bSigma\right]  t \sum_{k=2}^m \frac{(-1)^{k} k!}{w(t)^{k+1}} B_{m,k}\left(w^{(1)}(t),...,w^{(m-k+1)}(t)\right)\\
&& - c_n  \frac{1}{p}\text{tr}\left[\left( \tilde{w}(t) \bI_p+\bSigma\right)^{-2}\bSigma\right] m \sum_{k=1}^{m-1} \frac{(-1)^{k} k!}{w(t)^{k+1}} B_{m-1,k}\left(w^{(1)}(t),...,w^{(m-k)}(t)\right)\\
&&+c_n  \frac{1}{p}\text{tr}\left[\left( \tilde{w}(t) \bI_p+\bSigma\right)^{-2}\bSigma\right] \frac{t}{w(t)^2} w^{(m)}(t),
\end{eqnarray*}
from which we get \eqref{th3-wtpr} and \eqref{th3-wtpr-m}. The theorem is proved.
\end{proof}

In Corollary \ref{cor5-inverse}, the results are presented for the special case $t=0$, i.e., for the ordinary inverse of the sample covariance matrix.

\begin{corollary}\label{cor5-inverse}
Let $\bY_n$ fulfill the stochastic representation (1). Then, under Assumptions \textbf{(A1)}-\textbf{(A3)}, it holds that
\begin{equation}\label{cor-th3_eq1-ridge}
\left|\emph{tr}( \bS_n^{-(m+1)}\bTheta)-\tilde{r}_{m+1}(\bTheta)\right| \stackrel{a.s.}{\rightarrow} 0\quad\text{for} \quad p/n \rightarrow c \in (0, 1) 
\quad \text{as} \quad n \rightarrow \infty
\end{equation}
where 
\begin{eqnarray}\label{cor-th3_eq2-ridge}
&&\hspace{0.7cm}\tilde{r}_{1}\left(\bTheta\right)=\frac{1}{1-c_n}\emph{tr}\left[\bSigma^{-1}\bTheta\right],\\
&&\hspace{0.7cm}\tilde{r}_{m+1}\left(\bTheta\right)=
\frac{(-1)^{m+1}}{(m+1)!} 
\sum_{k=1}^{m+1} (-1)^{k} k! \tilde{d}_{k-1}\left(\Theta\right) 
 B_{m+1,k}\left(\tilde{w}^{(1)}(0),...,\tilde{w}^{(m+1-k+1)}(0)\right),\label{cor-th3_eq3-ridge}
\end{eqnarray}
with
\begin{eqnarray}\label{app-tilde-d}
&&\hspace{0.9cm}\tilde{d}_k\left(\bTheta\right)=\emph{tr}\left[\bSigma^{-(k+1)}\bTheta\right], \quad k=0,1,2,...,
\\
&&\hspace{0.9cm}\tilde{w}^{(1)}(0)=\frac{1}{1-c_n},
\tilde{w}^{(m)}(0)= m \sum_{k=1}^{m-1} \frac{(-1)^{k} k!}{(1-c_n)^{k+1}} B_{m-1,k}\left(w^{(1)}(0),...,w^{(m-k)}(0)\right), \label{app-tilde-w}
\end{eqnarray}
for $m=2,3,...$ where $w(0)=1-c_n$,
\begin{equation}\label{cor-th3-wpr}
w^{(1)}(0)=\frac{c_n}{1-c_n}\tilde{d}_1\left(\frac{1}{p}\bSigma\right), 
\end{equation}
and $w^{(m)}(0)$ for $m=2,3,...$ is computed recursively by
\begin{eqnarray}\label{cor-th3-wpr-m}
w^{(m)}(0)&=& -c_n \sum_{k=1}^m (-1)^{k} k! \tilde{d}_k\left(\frac{1}{p}\bSigma\right) B_{m,k}\left(\tilde{w}^{(1)}(0),...,\tilde{w}^{(m-k+1)}(0)\right).
\end{eqnarray}
\end{corollary}

\begin{proof}[Proof of Corollary \ref{cor5-inverse}:]
The statement of the corollary follows from Theorem \ref{th3-ridge} by noting that $w(0)=1-c_n$, $\tilde{w}(0)=0$, $w^{(1)}(0)=\dfrac{c_n}{1-c_n}\tilde{d}_1\left(\frac{1}{p}\bSigma\right)$, $\tilde{w}^{(1)}(0)=\dfrac{1}{1-c_n}$,
\[\tilde{w}^{(m)}(0)= m \sum_{k=1}^{m-1} \frac{(-1)^{k} k!}{(1-c_n)^{k+1}} B_{m-1,k}\left(w^{(1)}(0),...,w^{(m-k)}(0)\right),\]
and
\begin{eqnarray*}
w^{(m)}(0)&=& c_n \Bigg(\tilde{d}_1\left(\frac{1}{p}\bSigma\right) m \sum_{k=1}^{m-1} \frac{(-1)^{k} k!}{(1-c_n)^{k+1}} B_{m-1,k}\left(w^{(1)}(0),...,w^{(m-k)}(0)\right)\nonumber\\
&&-
\sum_{k=2}^m (-1)^{k} k! \tilde{d}_k\left(\frac{1}{p}\bSigma\right) B_{m,k}\left(\tilde{w}^{(1)}(0),...,\tilde{w}^{(m-k+1)}(0)\right)
\Bigg)\\
&=&-c_n\Bigg(-\tilde{d}_k\left(\frac{1}{p}\bSigma\right) \tilde{w}^{(m)}(0) + \sum_{k=2}^m (-1)^{k} k! \tilde{d}_k\left(\frac{1}{p}\bSigma\right) B_{m,k}\left(\tilde{w}^{(1)}(0),...,\tilde{w}^{(m-k+1)}(0)\right)
\Bigg)\\
&=& -c_n \sum_{k=1}^m (-1)^{k} k! \tilde{d}_k\left(\frac{1}{p}\bSigma\right) B_{m,k}\left(\tilde{w}^{(1)}(0),...,\tilde{w}^{(m-k+1)}(0)\right),
\end{eqnarray*}
where the equality $\tilde{w}^{(m)}(0) = B_{m,1}\left(\tilde{w}^{(1)}(0),...,\tilde{w}^{(m)}(0)\right)$ is used.

Similarly, we get
\begin{eqnarray*}
&&\tilde{r}_{m+1}\left(\bTheta\right)= \frac{(-1)^m}{m!}\tilde{d}_0\left(\bTheta\right)\sum_{k=1}^m \frac{(-1)^{k} k!}{(1-c_n)^{k+1}} B_{m,k}\left(w^{(1)}(0),...,w^{(m-k+1)}(0)\right)\nonumber\\
&&+\frac{(-1)^m}{m!}\frac{1}{1-c_n} \sum_{k=1}^m (-1)^{k} k!\tilde{d}_k\left(\bTheta\right)B_{m,k}\left(\tilde{w}^{(1)}(0),...,\tilde{w}^{(m-k+1)}(0)\right)
\nonumber\\
&&+\frac{(-1)^m}{m!}  \sum_{l=1}^{m-1} \frac{m!}{l!(m-l)!}\sum_{k_1=1}^{m-l} \frac{(-1)^{k_1} k_1!}{(1-c_n)^{k_1+1}} B_{m-l,k_1}\left(w^{(1)}(0),...,w^{(m-l-k_1+1)}(0)\right) \nonumber\\
&& \qquad \qquad \times \sum_{k_2=1}^l (-1)^{k_2} k_2!\tilde{d}_{k_2}\left(\bTheta\right) B_{l,k_2}\left(\tilde{w}^{(1)}(0),...,\tilde{w}^{(l-k_2+1)}(0)\right)\\
&&=
\frac{(-1)^m}{m!}\tilde{d}_0\left(\bTheta\right)\frac{\tilde{w}^{(m+1)}(0)}{m+1}+\frac{(-1)^m}{m!}\frac{1}{1-c_n} \sum_{k=1}^m (-1)^{k} k!\tilde{d}_k\left(\bTheta\right)B_{m,k}\left(\tilde{w}^{(1)}(0),...,\tilde{w}^{(m-k+1)}(0)\right)
\nonumber\\
&&+\frac{(-1)^m}{m!}  \sum_{l=1}^{m-1} \frac{m!}{l!(m-l)!}\frac{\tilde{w}^{(m-l+1)}(0)}{m-l+1} \sum_{k_2=1}^l (-1)^{k_2} k_2!\tilde{d}_{k_2}\left(\bTheta\right) B_{l,k_2}\left(\tilde{w}^{(1)}(0),...,\tilde{w}^{(l-k_2+1)}(0)\right)\\
&=&\frac{(-1)^m\tilde{w}^{(m+1)}(0)}{(m+1)!}\tilde{d}_0\left(\bTheta\right)\\
&+&\frac{(-1)^m}{(m+1)!}
\sum_{l=1}^{m} \frac{(m+1)!}{l!(m+1-l)!}\tilde{w}^{(m-l+1)}(0) \sum_{k_2=1}^l (-1)^{k_2} k_2!\tilde{d}_{k_2}\left(\bTheta\right) B_{l,k_2}\left(\tilde{w}^{(1)}(0),...,\tilde{w}^{(l-k_2+1)}(0)\right)\\
&=&\frac{(-1)^m \tilde{w}^{(m+1)}(0)}{(m+1)!}\tilde{d}_0\left(\bTheta\right)\\
&+&\frac{(-1)^m}{(m+1)!} 
\sum_{k_2=1}^m (-1)^{k_2} \tilde{d}_{k_2}\left(\bTheta\right) 
k_2!\sum_{l=k_2}^{m} \frac{(m+1)!}{l!(m+1-l)!}\tilde{w}^{(m-l+1)}(0)B_{l,k_2}\left(\tilde{w}^{(1)}(0),...,\tilde{w}^{(l-k_2+1)}(0)\right).
\end{eqnarray*}

From Equation (1.4) in \cite{cvijovic2011new} and the fact that $B_{i,k_2}(.)=0$ for $i<k_2$, we get
\begin{eqnarray*}
&&k_2!\sum_{l=k_2}^{m} \frac{(m+1)!}{l!(m+1-l)!}\tilde{w}^{(m-l+1)}(0)B_{l,k_2}\left(\tilde{w}^{(1)}(0),...,\tilde{w}^{(l-k_2+1)}(0)\right)\\
&=&k_2!\sum_{r=1}^{m+1-k_2} \frac{(m+1)!}{r!(m+1-r)!}\tilde{w}^{(r)}(0)B_{m+1-r,k_2}\left(\tilde{w}^{(1)}(0),...,\tilde{w}^{(m+1-r-k_2+1)}(0)\right)\\
&=&(k_2+1)! B_{m+1,k_2+1}\left(\tilde{w}^{(1)}(0),...,\tilde{w}^{(m-k_2+1)}(0)\right).
\end{eqnarray*}

Hence,
\begin{eqnarray*}
&&\tilde{r}_{m+1}\left(\bTheta\right)=
\frac{(-1)^m \tilde{w}^{(m+1)}(0)}{(m+1)!}\tilde{d}_0\left(\Theta\right)\\
&-&\frac{(-1)^m}{(m+1)!} 
\sum_{k_2=1}^m (-1)^{k_2+1} \tilde{d}_{k_2}\left(\Theta\right) 
(k_2+1)! B_{m+1,k_2+1}\left(\tilde{w}^{(1)}(0),...,\tilde{w}^{(m-k_2+1)}(0)\right)\\
&=&\frac{(-1)^{m+1}}{(m+1)!} 
\sum_{k_2=0}^m (-1)^{k_2+1} \tilde{d}_{k_2}\left(\Theta\right) 
(k_2+1)! B_{m+1,k_2+1}\left(\tilde{w}^{(1)}(0),...,\tilde{w}^{(m-k_2+1)}(0)\right)\\
&=&\frac{(-1)^{m+1}}{(m+1)!} 
\sum_{k=1}^{m+1} (-1)^{k} k! \tilde{d}_{k-1}\left(\Theta\right) 
 B_{m+1,k}\left(\tilde{w}^{(1)}(0),...,\tilde{w}^{(m+1-k+1)}(0)\right),
\end{eqnarray*}
which completes the proof of the corollary.
\end{proof}

In Corollary \ref{cor5-ridge}, the results for $m=0,1,2,3$ are summarized.
\begin{corollary}\label{cor5-ridge}
Let $\bY_n$ fulfill the stochastic representation (1). Then, under Assumptions \textbf{(A1)}-\textbf{(A3)}, we get \eqref{cor5-inverse} with
\begin{eqnarray*}
\tilde{r}_{1}\left(\bTheta\right)&=& \frac{1}{1-c_n}\emph{tr}\left[ \bSigma^{-1}\bTheta\right],\\
\tilde{r}_{2}\left(\bTheta\right)&=& \frac{c_n}{(1-c_n)^3}\emph{tr}\left[ \bSigma^{-1}\bTheta\right]\frac{1}{p}\emph{tr}\left[ \bSigma^{-1}\right]+\frac{1}{(1-c_n)^2}\emph{tr}\left[ \bSigma^{-2}\bTheta\right],\\
\tilde{r}_{3}\left(\bTheta\right)&=& 
\frac{\tilde{w}^{(3)}(0)}{6}\emph{tr}\left[ \bSigma^{-1}\bTheta\right]-\frac{\tilde{w}^{(2)}(0)}{1-c_n}\emph{tr}\left[ \bSigma^{-2}\bTheta\right]+\frac{1}{(1-c_n)^3}\emph{tr}\left[ \bSigma^{-3}\bTheta\right],\\
\tilde{r}_{4}\left(\bTheta\right)&=& -\frac{\tilde{w}^{(4)}(0)}{24}\emph{tr}\left[ \bSigma^{-1}\bTheta\right]
+\left(\frac{\tilde{w}^{(3)}(0)}{3(1-c_n)}+\frac{[\tilde{w}^{(2)}(0)]^2}{4}\right)\emph{tr}\left[ \bSigma^{-2}\bTheta\right]\\
&-&\frac{3\tilde{w}^{(2)}(0)}{2(1-c_n)^2}\emph{tr}\left[ \bSigma^{-3}\bTheta\right]
+\frac{1}{(1-c_n)^4}\emph{tr}\left[ \bSigma^{-4}\bTheta\right],
\end{eqnarray*}
where 
\begin{eqnarray*}
\tilde{w}^{(2)}(0)&=& -\frac{2}{(1-c_n)^{2}} w^{(1)}(0),\\
\tilde{w}^{(3)}(0)&=& -3 \left(
\frac{1}{(1-c_n)^{2}} w^{(2)}(0)-
\frac{2}{(1-c_n)^{3}} [w^{(1)}(0)]^2
\right),\\
\tilde{w}^{(4)}(0)&=& -4
\left(
\frac{1}{(1-c_n)^{2}} w^{(3)}(0)
-\frac{6}{(1-c_n)^{3}} w^{(1)}(0)w^{(2)}(0)
+\frac{ 6}{(1-c_n)^{4}} [w^{(1)}(0)]^3\right),\\
\end{eqnarray*}
with
\begin{eqnarray*}
w^{(1)}(0)&=&\frac{c_n}{1-c_n}\frac{1}{p}\emph{tr}\left[ \bSigma^{-1}\right],\\
w^{(2)}(0)&=& c_n \left(
 \tilde{w}^{(2)}(0)\frac{1}{p}\emph{tr}\left[ \bSigma^{-1}\right]
-2  \frac{1}{(1-c_n)^2}\frac{1}{p}\emph{tr}\left[ \bSigma^{-2}\right]\right)
,\\
w^{(3)}(0)&=& c_n \left(
\tilde{w}^{(3)}(0) \frac{1}{p}\emph{tr}\left[ \bSigma^{-1}\right]
-6\frac{\tilde{w}^{(2)}(0) }{1-c_n}\frac{1}{p}\emph{tr}\left[ \bSigma^{-2}\right]
+6\frac{1}{(1-c_n)^3}\frac{1}{p}\emph{tr}\left[ \bSigma^{-3}\right]\right).
\end{eqnarray*}
\end{corollary}

\begin{proof}[Proof of Corollary \ref{cor5-ridge}:]
The expression for $m=0$ is already present in \eqref{cor-th3_eq2-ridge} of Corollary \ref{cor5-inverse}. The application of \eqref{cor-th3_eq3-ridge} for $m=1,2,3$ yields 
\begin{eqnarray*}
\tilde{r}_{2}\left(\bTheta\right)&=&-\frac{1}{2}\tilde{d}_0\left(\bTheta\right)\tilde{w}^{(2)}(0)+\tilde{d}_1\left(\bTheta\right)[\tilde{w}^{(1)}(0)]^2\\
\tilde{r}_{3}\left(\bTheta\right)&=& \frac{1}{6}\tilde{d}_0\left(\bTheta\right)\tilde{w}^{(3)}(0)-\tilde{d}_1\left(\bTheta\right)\tilde{w}^{(1)}(0)\tilde{w}^{(2)}(0)+\tilde{d}_2\left(\bTheta\right)[\tilde{w}^{(1)}(0)]^3
\\
\tilde{r}_{4}\left(\bTheta\right)&=& -\frac{1}{24}\tilde{d}_0\left(\bTheta\right)\tilde{w}^{(4)}(0)+\frac{1}{12}\tilde{d}_1\left(\bTheta\right)(4\tilde{w}^{(1)}(0)\tilde{w}^{(3)}(0)+3[\tilde{w}^{(2)}(0)]^2)\\
&&-\frac{3}{2}\tilde{d}_2\left(\bTheta\right)[\tilde{w}^{(1)}(0)]^2\tilde{w}^{(2)}(0)
+\tilde{d}_3\left(\bTheta\right)[\tilde{w}^{(1)}(0)]^4,
\end{eqnarray*}
where $\tilde{w}^{(1)}(0)=\dfrac{1}{1-c_n}$ and $\tilde{w}^{(2)}(0)$, $\tilde{w}^{(3)}(0)$, and $\tilde{w}^{(4)}(0)$ are provided in the statement of the corollary.
\end{proof}

The formulas for $\tilde{r}_1(\bTheta)$ and $\tilde{r}_2(\bTheta)$ in Corollary \ref{cor5-ridge} coincide with the previous findings derived in \cite{BodnarGuptaParolya2016}. In Corollary \ref{cor0S-inverse}, we present the results derived for $\bTheta=\frac{1}{p}\bI_p$.

\begin{corollary}\label{cor0S-inverse}
Let $\bY_n$ fulfill the stochastic representation (1) and let $\bTheta=\frac{1}{p}\bI_p$. Then, under Assumptions \textbf{(A1)}-\textbf{(A2)}, it holds that
\begin{equation*}
\left|\frac{1}{p}\emph{tr}\left[ \bS_n^{-(m+1)}\right]-\frac{(-1)^{m}w^{(m+1)}(0)}{(m+1)! c_n}\right| \stackrel{a.s.}{\rightarrow} 0\quad\text{for} \quad p/n \rightarrow c \in (0,1) 
\quad \text{as} \quad n \rightarrow \infty,
\end{equation*}
where $w^{(m+1)}(0)$, $m=0,1,...$, are defined in \eqref{cor-th3-wpr-m}.
\end{corollary}

\begin{proof}[Proof of Corollary \ref{cor0S-inverse}:]
For $m=0$, we get from Corollary \ref{cor5-inverse} that
\begin{eqnarray*}
\tilde{r}_{1}\left(\dfrac{1}{p}\bI_p\right)=\frac{1}{1-c_n}\tilde{d}_0\left(\frac{1}{p}\bI_p\right)=\frac{1}{c_n}\frac{c_n}{1-c_n}\tilde{d}_1\left(\frac{1}{p}\bSigma\right)=\frac{w^{(1)}(0)}{c_n}.
\end{eqnarray*}

The application of $\tilde{d}_k\left(\frac{1}{p}\bI_p\right) =\tilde{d}_{k+1}\left(\frac{1}{p}\bSigma\right)$ for $m=1,2,...$ and Corollary \ref{cor5-inverse} yield
\begin{eqnarray*}
\tilde{r}_{m+1}\left(\dfrac{1}{p}\bI_p\right)
&=&\frac{(-1)^{m+1}}{(m+1)!} 
\sum_{k=1}^{m+1} (-1)^{k} k! \tilde{d}_{k}\left(\dfrac{1}{p}\bSigma\right) 
 B_{m+1,k}\left(\tilde{w}^{(1)}(0),...,\tilde{w}^{(m+1-k+1)}(0)\right)\\
 &=&\frac{(-1)^m}{(m+1)!c_n}w^{(m+1)}(0),
\end{eqnarray*}
where the second equality follows from \eqref{cor-th3-wpr-m}.
\end{proof}

\begin{corollary}\label{cor5a-inverse}
Let $\bY_n$ fulfill the stochastic representation (1) with $\bSigma=\bI_p$. Then, under Assumptions \textbf{(A2)}-\textbf{(A3)}, it holds that
\begin{equation}\label{cor5a-eq1-inverse}
\left|\emph{tr}( \bS_n^{-(m+1)}\bTheta)- \frac{(-1)^{m}w^{(m+1)}(0)}{(m+1)! c_n}
\emph{tr}\left(\bTheta\right)\right| \stackrel{a.s.}{\rightarrow} 0
\end{equation}
for $p/n \rightarrow c \in (0,1)$ as $n \rightarrow \infty$, where $w^{(m+1)}(0)$, $m=0,1,...$, are given in \eqref{cor-th3-wpr-m} with $\bSigma=\bI_p$.
\end{corollary}

\begin{proof}[Proof of Corollary \ref{cor5a-inverse}:]
The result follows from Corollary \ref{cor5-inverse} by noting that $d_{k}(\bTheta)=\text{tr}\left[\bTheta\right]$ for $k=0,1,...$ when $\bSigma=\bI_p$.

\end{proof}

As a special case of Corollary \ref{cor5a-inverse}, we get
\begin{eqnarray*}
&\left|\text{tr}( \bS_n^{-1}\bTheta)- \frac{1}{1-c_n}\text{tr}\left(\bTheta\right)\right| \stackrel{a.s.}{\rightarrow} 0, 
&\left|\text{tr}( \bS_n^{-2}\bTheta)-\frac{1}{(1-c_n)^3}\text{tr}\left(\bTheta\right)\right| \stackrel{a.s.}{\rightarrow}0,\\
&\left|\text{tr}(\bS_n^{-3}\bTheta)-\frac{1+c_n}{(1-c_n)^5}\text{tr}\left(\bTheta\right)\right| \stackrel{a.s.}{\rightarrow} 0,\quad
&\left|\text{tr}(\bS_n^{-4}\bTheta)-\frac{1+3c_n+c_n^2}{(1-c_n)^7}\text{tr}\left(\bTheta\right)\right| \stackrel{a.s.}{\rightarrow}0,
\end{eqnarray*}
for $p/n \rightarrow c \in (0,1)$ as $n \rightarrow \infty$. 

\section{Weighted moments of the sample covariance matrix}\label{sec:main-sample_cov}

In this section, we present the auxiliary results for the weighted trace moments of the sample covariance matrix needed for the estimation of some quantities from Section 3. The results of this section are of the independent interest and are established in both cases $p/n>1$ and $p/n<1$.

\begin{theorem}\label{th3}
Let $\bY_n$ fulfill the stochastic representation (1). Then, under Assumptions \textbf{(A1)}-\textbf{(A3)}, it holds that
\begin{equation}\label{th3_eq1}
\left|\emph{tr}( \bS_n^{m}\bTheta)-\breve{s}_m(\bTheta)\right| \stackrel{a.s.}{\rightarrow} 0\quad\text{for} \quad p/n \rightarrow c \in (0,\infty) 
\quad \text{as} \quad n \rightarrow \infty
\end{equation}
where 
\begin{eqnarray}\label{th3-bsm}
    \breve{s}_m(\bTheta)&=&  \sum_{k=1}^m \frac{(-1)^{m+k} k!}{m!}\emph{tr}\left\{\bSigma^{k}\bTheta\right\}  B_{m,k}\left(u^{(1)}(0),u^{(2)}(0),...,u^{(m-k+1)}(0)\right),
\end{eqnarray}
$u^{(1)}(0)=1$ and $u^{(2)}(0)$,...,$u^{(m)}(0)$ are computed recursively by
\begin{equation}\label{th3-u0pr-m}
u^{(m)}(0)= m c_n  \sum_{k=1}^{m-1} \frac{(-1)^{k} k!}{p}\emph{tr}\left\{\bSigma^k\right\} B_{m-1,k}\left(u^{(1)}(0),...,u^{(m-k)}(0)\right).
\end{equation}
\end{theorem}

\begin{proof}[Proof of Theorem \ref{th3}:]
Following the proof of Theorem 3.1 in \cite{bodnar2014strong}, we get the following identity
\begin{eqnarray*}
\text{tr}(\bS_n^{m}\bTheta)&=& \frac{(-1)^m}{m!}\left.\dfrac{\partial^m}{\partial t^m}\,\text{tr}\left[\left(t\bS_n+\bI_p\right)^{-1}
\bTheta\right] \right|_{t=0}
\end{eqnarray*}
The proof of Theorem 2.1 and the application of Lemma 6 in \cite{BPTJMLR2024} yields
\begin{eqnarray*}
\text{tr}(\bS_n^m\bTheta)&=&
\frac{(-1)^m}{m!}\left.\dfrac{\partial^m}{\partial t^m}\,\text{tr}\left[\left(\frac{t}{n}\bX_n\bX_n^\top+\bSigma^{-1}\right)^{-1}
\bSigma^{-1/2}\bTheta\bSigma^{-1/2}\right] \right|_{t=0}\\
&\stackrel{a.s.}{\longrightarrow}&
\frac{(-1)^m}{m!}\left.\dfrac{\partial^m}{\partial t^m}\,\text{tr}\left[\left(u(t)\bI_p+\bSigma^{-1}\right)^{-1}
\bSigma^{-1/2}\bTheta\bSigma^{-1/2}\right] \right|_{t=0}
\,,
\end{eqnarray*}
for $p/n \rightarrow c \in [0,\infty)$ as $n \to \infty$ where $u(t)$ solves the following equation 
\begin{equation}\label{ut-main}
t-u(t)=c_n t-c_n t \frac{1}{p}\text{tr}\left[\left(u(t)\bI_p+\bSigma^{-1}\right)^{-1}\bSigma^{-1}\right].
\end{equation}

From the proof of Theorem 2.1, we get
\begin{eqnarray*}
\text{tr}(\bS_n^m\bTheta)
&\stackrel{d.a.s.}{\longrightarrow}&
\sum_{k=1}^m \frac{(-1)^{m+k} k!}{m!}\text{tr}\left\{\left(u(0)\bSigma+\bI_p\right)^{-1}\left[\bSigma\left(u(0)\bSigma+\bI_p\right)^{-1}\right]^{k}\bTheta\right\} \\
&& \times B_{m,k}\left(u^{(1)}(0),u^{(2)}(0),...,u^{(m-k+1)}(0)\right)
\,,
\end{eqnarray*}
for $p/n \rightarrow c \in [0,\infty)$ as $n \to \infty$. Furthermore, the application of \eqref{ut-main} yields $u(0)=0$ and $u^{(1)}(0)=1$. Furthermore, it holds that
\begin{eqnarray*}
u^{(m)}(t)&=&c_n t \dfrac{\partial^m}{\partial t^m} \frac{1}{p}\text{tr}\left[\left(u(t)\bI_p+\bSigma^{-1}\right)^{-1}\bSigma^{-1}\right]+ m c_n \dfrac{\partial^{m-1}}{\partial t^{m-1}} \frac{1}{p}\text{tr}\left[\left(u(t)\bI_p+\bSigma^{-1}\right)^{-1}\bSigma^{-1}\right]
\end{eqnarray*}
and, hence,
\begin{eqnarray*}
u^{(m)}(0)&=& m c_n  \sum_{k=1}^{m-1} \frac{(-1)^{k} k!}{p}\text{tr}\left\{\left[\bSigma\left(u(0)\bSigma+\bI_p\right)^{-1}\right]^{k+1}\bSigma^{-1}\right\} B_{m-1,k}\left(u^{(1)}(0),...,u^{(m-k)}(0)\right).
\end{eqnarray*}
\end{proof}

\begin{corollary}\label{cor5}
Let $\bY_n$ fulfill the stochastic representation (1). Then,  under Assumptions \textbf{(A1)}-\textbf{(A3)}, we get \eqref{th3_eq1} with
\begin{eqnarray*}
\breve{s}_1(\bTheta)&=& \emph{tr}(\bSigma\bTheta),\\
\breve{s}_2(\bTheta)&=& c_n \frac{1}{p}\emph{tr}(\bSigma)\emph{tr}(\bSigma\bTheta)+\emph{tr}(\bSigma^2\bTheta),\\
\breve{s}_3(\bTheta)&=& c_n\left\{c_n \left[\frac{1}{p}\emph{tr}(\bSigma)\right]^2 + \frac{1}{p}\emph{tr}(\bSigma^2)\right\}\emph{tr}(\bSigma\bTheta)+2c_n\frac{1}{p}\emph{tr}(\bSigma)\emph{tr}(\bSigma^2\bTheta)+\emph{tr}(\bSigma^3\bTheta),\\
\breve{s}_4(\bTheta)&=& c_n\left\{c_n^2\left[\frac{1}{p}\emph{tr}(\bSigma)\right]^3+3c_n\frac{1}{p}\emph{tr}(\bSigma)\frac{1}{p}\emph{tr}(\bSigma^2)+\frac{1}{p}\emph{tr}(\bSigma^3)\right\}\emph{tr}(\bSigma\bTheta)\\
&&+
c_n\left\{3c_n \left[\frac{1}{p}\emph{tr}(\bSigma)\right]^2 + 2\frac{1}{p}\emph{tr}(\bSigma^2)\right\}\emph{tr}(\bSigma^2\bTheta)+3c_n\frac{1}{p}\emph{tr}(\bSigma)\emph{tr}(\bSigma^3\bTheta)+\emph{tr}(\bSigma^4\bTheta),
\end{eqnarray*}
for $p/n \rightarrow c \in [0,\infty)$ as $n \rightarrow \infty$.
\end{corollary}

\begin{proof}[Proof of Corollary \ref{cor5}:]
The results of Theorem \ref{th3} and the proof of Corollary 2.1 yield
 \begin{eqnarray*}
\breve{s}_1(\bTheta)&=& \text{tr}(\bSigma\bTheta)B_{1,1}\left(u^{(1)}(0)\right)=\text{tr}(\bSigma\bTheta)u^{(1)}(0),\\
\breve{s}_2(\bTheta)&=& -\frac{1}{2}\text{tr}(\bSigma\bTheta) B_{2,1}\left(u^{(1)}(0),u^{(2)}(0)\right)+\text{tr}(\bSigma^2\bTheta) B_{2,2}\left(u^{(1)}(0)\right)\\
&=& -\frac{1}{2}u^{(2)}(0)\text{tr}(\bSigma\bTheta) +\text{tr}(\bSigma^2\bTheta) [u^{(1)}(0)]^2,\\
\breve{s}_3(\bTheta)&=& \frac{1}{6}\text{tr}(\bSigma\bTheta)B_{3,1}\left(u^{(1)}(0),u^{(2)}(0),u^{(3)}(0)\right)-\frac{1}{3}\text{tr}(\bSigma^2\bTheta)B_{3,2}\left(u^{(1)}(0),u^{(2)}(0)\right)+\text{tr}(\bSigma^3\bTheta)B_{3,3}\left(u^{(1)}(0)\right)\\&&=\frac{1}{6}\text{tr}(\bSigma\bTheta)u^{(3)}(0)-\text{tr}(\bSigma^2\bTheta)u^{(1)}(0)u^{(2)}(0)+\text{tr}(\bSigma^3\bTheta)[u^{(1)}(0)]^3,\\
\breve{s}_4(\bTheta)&=& -\frac{1}{24}\text{tr}(\bSigma\bTheta) B_{4,1}\left(u^{(1)}(0),u^{(2)}(0),u^{(3)}(0),u^{(4)}(0)\right)+\frac{1}{12}
\text{tr}(\bSigma^2\bTheta)B_{4,2}\left(u^{(1)}(0),u^{(2)}(0),u^{(3)}(0)\right)\\
&&-\frac{1}{4}\text{tr}(\bSigma^3\bTheta)B_{4,3}\left(u^{(1)}(0),u^{(2)}(0)\right)+\text{tr}(\bSigma^4\bTheta)B_{4,4}\left(u^{(1)}(0)\right)\\&&=-\frac{1}{24}\text{tr}(\bSigma\bTheta) u^{(4)}(0)+\frac{1}{12}
\text{tr}(\bSigma^2\bTheta)[4u^{(1)}(0)u^{(3)}(0)+3[u^{(2)}(0)]^2]\\
&&-\frac{3}{2}\text{tr}(\bSigma^3\bTheta)[u^{(1)}(0)]^2u^{(2)}(0)+\text{tr}(\bSigma^4\bTheta)[u^{(1)}(0)]^4.
\end{eqnarray*}   

Furthermore, we get from Theorem \ref{th3} that \begin{eqnarray*}  
u^{(1)}(0)&=&1, \\
u^{(2)}(0)&=&-2c_n\frac{1}{p}\text{tr}(\bSigma) B_{1,1}\left(u^{(1)}(0)\right)= -2c_n\frac{1}{p}\text{tr}(\bSigma),\\
u^{(3)}(0)&=&3c_n\left\{-\frac{1}{p}\text{tr}(\bSigma)B_{2,1}\left(u^{(1)}(0),u^{(2)}(0)\right)+2\frac{1}{p}\text{tr}(\bSigma^2)B_{2,2}\left(u^{(1)}(0)\right)\right\}\\
&=&3c_n\left\{-\frac{1}{p}\text{tr}(\bSigma)u^{(2)}(0)+2\frac{1}{p}\text{tr}(\bSigma^2)[u^{(1)}(0)]^2\right\}
=6c_n\left\{c_n\left[\frac{1}{p}\text{tr}(\bSigma)\right]^2+
\frac{1}{p}\text{tr}(\bSigma^2)
\right\},\\
u^{(4)}(0)&=&4c_n\Bigg\{-\frac{1}{p}\text{tr}(\bSigma)B_{3,1}\left(u^{(1)}(0),u^{(2)}(0),u^{(3)}(0)\right)+2\frac{1}{p}\text{tr}(\bSigma^2)B_{3,2}\left(u^{(1)}(0),u^{(2)}(0)\right)\\
&&+6\frac{1}{p}\text{tr}(\bSigma^3)B_{3,3}\left(u^{(1)}(0)\right)\Bigg\}\\
&=&4c_n\Bigg\{-\frac{1}{p}\text{tr}(\bSigma)u^{(3)}(0)+6\frac{1}{p}\text{tr}(\bSigma^2)u^{(1)}(0)u^{(2)}(0)+6\frac{1}{p}\text{tr}(\bSigma^3)[u^{(1)}(0)]^3)\Bigg\}\\
&=&-24c_n\left\{c_n^2\left[\frac{1}{p}\text{tr}(\bSigma)\right]^3+3c_n\frac{1}{p}\text{tr}(\bSigma)\frac{1}{p}\text{tr}(\bSigma^2)+\frac{1}{p}\text{tr}(\bSigma^3)\right\}.
\end{eqnarray*}  
Substituting the expressions of $u^{(1)}(0)$, $u^{(2)}(0)$, $u^{(3)}(0)$, and $u^{(4)}(0)$ into the formulas obtained for $\breve{s}_1$,...,$\breve{s}_4$ leads to the statement of the corollary.
\end{proof}

In Corollary \ref{cor0S} we present the results derived for $\bTheta=\frac{1}{p}\bI_p$.

\begin{corollary}\label{cor0S}
Let $\bY_n$ fulfill the stochastic representation (1). Then, under Assumptions \textbf{(A1)}-\textbf{(A2)}, it holds for $m=1,2,...$ that
\begin{equation*}
\left|\frac{1}{p}\emph{tr}\left[ \bS_n^{m}\right]-\frac{(-1)^{m}u^{(m+1)}(0)}{(m+1)! c_n}\right| \stackrel{a.s.}{\rightarrow} 0\quad\text{for} \quad p/n \rightarrow c \in (0,\infty) 
\quad \text{as} \quad n \rightarrow \infty,
\end{equation*}
where
$u^{m}(0)$, $m=1,...$, are defined in \eqref{th3-u0pr-m}.
\end{corollary}

\begin{proof}[Proof of Corollary \ref{cor0S}:]
From \eqref{th3-u0pr-m} we have that
\begin{eqnarray*}
\frac{u^{(m+1)}(0)}{(m+1)c_n}&=& \sum_{k=1}^{m} \frac{(-1)^{k} k!}{p}\text{tr}(\bSigma^{k}) B_{m,k}\left(u^{(1)}(0),...,u^{(m-k+1)}(0)\right),
\end{eqnarray*}
which together with \eqref{th3-bsm} computed for $\bTheta=\frac{1}{p}\bI_p$ completes the proof of the theorem.
\end{proof}

\begin{corollary}\label{cor5a}
Let $\bY_n$ fulfill the stochastic representation (1) with $\bSigma=\bI_p$. Then, under Assumptions \textbf{(A2)}-\textbf{(A3)}, it holds that
\begin{equation}\label{cor5a-eq1}
\left|\emph{tr}( \bS_n^{m}\bTheta)- \frac{(-1)^{m}u^{(m+1)}(0)}{(m+1)! c_n}\emph{tr}\left\{\bTheta\right\}\right| \stackrel{a.s.}{\rightarrow} 0
\end{equation}
for $p/n \rightarrow c \in (0,\infty)$ as $n \rightarrow \infty$, where $u^{(m+1)}(0)$, $m=1,...$, are given in \eqref{th3-u0pr-m} with $\bSigma=\bI_p$.
\end{corollary}

\begin{proof}[Proof of Corollary \ref{cor5a}:]
If $\bSigma=\bI_p$, then
\begin{eqnarray*}
\breve{s}_m(\bTheta) =  \frac{(-1)^{m} \text{tr}\left\{\bTheta\right\}}{m!} \sum_{k=1}^m (-1)^{k}k!B_{m,k}\left(u^{(1)}(0),u^{(2)}(0),...,u^{(m-k+1)}(0)\right),
\end{eqnarray*}
which together with \eqref{th3-u0pr-m} completes the proof of the corollary.
\end{proof}

As a special case of Corollary \ref{cor5a}, we get
\begin{eqnarray*}
&\left|\text{tr}( \bS_n\bTheta)-\text{tr}\left(\bTheta\right)\right| \stackrel{a.s.}{\rightarrow} 0,
&\left|\text{tr}( \bS_n^2\bTheta)-(c_n+1)\text{tr}\left(\bTheta\right)\right| \stackrel{a.s.}{\rightarrow} 0,\\
&\left|\text{tr}(\bS_n^3\bTheta)-(c_n^2+3c_n+1)\text{tr}\left(\bTheta\right)\right| \stackrel{a.s.}{\rightarrow}0,\quad
&\left|\text{tr}(\bS_n^4\bTheta)-(c_n^3+6c_n^2+6c_n+1)\text{tr}\left(\bTheta\right)\right| \stackrel{a.s.}{\rightarrow} 0,
\end{eqnarray*}
for $p/n \rightarrow c \in (0,\infty)$ as $n \rightarrow \infty$. 

\section{Consistent estimators for $h_m$ and $d_m(.,.)$}\label{estimation_full}
In this section we provide consistent estimators of all unknown quantities needed in Section 3. Indeed, the findings of Corollary 2.2 provide the closed-form formulae of consistent estimators for $v^{(m)}(0)$, $m=0,1,...$, as summarized in (8). Using these results together with Corollary 2.1, we get consistent estimators for $h_m$, $d_m\left(\frac{1}{p}\bI_p\right)$, and, more generally, $d_m\left(\bTheta\right)$ for $m=1,2, \ldots$. In particular, consistent estimators for $h_2$, $h_3$, $d_1\left(\frac{1}{p}\bI_p\right)$, $d_2\left(\frac{1}{p}\bI_p\right)$, $d_1\left(\bTheta\right)$, $d_2\left(\bTheta\right)$, and $d_3\left(\bTheta\right)$ are given by
{\footnotesize
\begin{eqnarray}
   &&\hat{h}_2 = -\frac{1}{\hat{v}^{(1)}(0)}= \frac{1}{c_n \frac{1}{p}\text{tr}\left[ (\bS_n^+)^{2}\right]},\label{hh_2}\\
  &&\hat{h}_3 = \frac{-\hat{v}^{(2)}(0)}{2[\hat{v}^{(1)}(0)]^3}=  \frac{\frac{1}{p}\text{tr}\left[ (\bS_n^+)^{3}\right]}
   {c_n^2 \left\{\frac{1}{p}\text{tr}\left[ (\bS_n^+)^{2}\right]\right\}^3}, \label{hh_3}\\
   &&\hat{d}_1\left(\frac{1}{p}\bI_p\right) =- \frac{\hat{s}_1\left(\frac{1}{p}\bI_p\right)}{\hat{v}^{(1)}(0)} = \frac{\frac{1}{p}\text{tr}\left[\bS_n^+\right]}{c_n \frac{1}{p}\text{tr}\left[ (\bS_n^+)^{2}\right]},  \label{hd_1}\\
   &&\hat{d}_1\left(\bTheta\right) =- \frac{\hat{s}_1\left(\bTheta\right)}{\hat{v}^{(1)}(0)} = \frac{\text{tr}\left[\bS_n^+\bTheta\right]}{c_n \frac{1}{p}\text{tr}\left[ (\bS_n^+)^{2}\right]},  \label{hd_1Pi0}\\
 && \hat{d}_2\left(\frac{1}{p}\bI_p\right) =\!\!\frac{\frac{1}{2}\hat{v}^{(2)}(0) \hat{d}_1\left(\frac{1}{p}\bI_p\right)-\hat{s}_2\left(\frac{1}{p}\bI_p\right) }{[\hat{v}^{(1)}(0)]^2} =
   \frac{\frac{1}{p}\text{tr}\left[\bS_n^+\right] \frac{1}{p}\text{tr}\left[ (\bS_n^+)^{3}\right] -\left\{\frac{1}{p}\text{tr}\left[ (\bS_n^+)^{2}\right]\right\}^2}
  {c_n^2 \left\{\frac{1}{p}\text{tr}\left[ (\bS_n^+)^{2}\right]\right\}^3}  \label{hd_2} \\
&&  \hat{d}_2\left(\bTheta\right)= \frac{\frac{1}{2} \hat{v}''(0)\hat{d}_1\left(\bTheta\right)-\hat{s}_2(\bTheta)}{[\hat{v}'(0)]^2}= \frac{\text{tr}\left[\bS_n^+\bTheta\right] \frac{1}{p}\text{tr}\left[ (\bS_n^+)^{3}\right]-\frac{1}{p}\text{tr}\left[ (\bS_n^+)^{2}\right]\tr\left[(\bS_n^+)^{2}\bTheta\right]}{c_n^2 \left\{\frac{1}{p}\text{tr}\left[ (\bS_n^+)^{2}\right]\right\}^3} \label{hd_22}\\
 && \label{hd_3}  \hat{d}_3\left(\bTheta\right)= \frac{\hat{v}'(0)\hat{v}''(0)\hat{d}_2(\bTheta)-\frac{1}{6}\hat{v}'''(0)\hat{d}_1(\bTheta)-\hat{s}_3(\bTheta)}{[\hat{v}'(0)]^3}\\
   &=& \frac{\tr\left((\bS_n^+)^{3}\bTheta\right)}{c^3_n\left\{\frac{1}{p}\text{tr}\left[ (\bS_n^+)^{2}\right]\right\}^3}+\frac{2\left(\frac{1}{p}\text{tr}\left[ (\bS_n^+)^{3}\right]\right)^2\text{tr}\left[\bS_n^+\bTheta\right] }{c^3_n \left\{\frac{1}{p}\text{tr}\left[ (\bS_n^+)^{2}\right]\right\}^5 }-\frac{\tr\left[(\bS_n^+)^{2}\bTheta\right]+\frac{1}{p}\text{tr}\left[ (\bS_n^+)^{4}\right]\text{tr}\left[ \bS_n^+\bTheta\right]}{c^3_n \left\{\frac{1}{p}\text{tr}\left[ (\bS_n^+)^{2}\right]\right\}^4 }
 \nonumber \,.
\end{eqnarray} 
}


 Similarly, the findings of Theorem 2.2 and Theorem \ref{th3} lead to the consistent estimators of $d_0(t,\bTheta)$, $d_1(t,\bTheta)$, $q_1(\bTheta)=\text{tr}\left[\bSigma\bTheta\right]$ and $q_2(\bTheta)=\text{tr}\left[\bSigma^2\bTheta\right]$ expressed as 
\begin{eqnarray}
\hat{d}_0(t,\bTheta)&=& \left\{
\begin{array}{cc}
  t~ \text{tr}\left[(\bS_n+t\bI_p)^{-1}\bTheta\right] \quad  &  \text{for} \quad t>0 \\
   \text{tr}\left[(\bI_p-\bS_n\bS_n^+)\bTheta\right]& \text{for} \quad t=0
\end{array}
\right.
  ,\label{hd0t}\\
\hat{d}_1(t,\bTheta)&=&  -\frac{t \text{tr}\left[(\bS_n+t\bI_p)^{-2}\bTheta\right]-t^{-1}\hat{d}_0(t,\bTheta)}{c_n \left(\frac{1}{p}\text{tr}\left[ (\bS_n+t\bI_p)^{-2}\right]
-t^{-2} \frac{c_n-1}{c_n}\right)} \quad \text{for} \quad t>0   ,\label{hd1t}\\
\hat{q}_1(\bTheta)&=&  \text{tr}\left[\bS_n\bTheta\right],  \label{hq_1}\\
\hat{q}_2(\bTheta)&=&  \text{tr}\left[\bS_n^2\bTheta\right] -c_n\frac{1}{p}\text{tr}\left[\bS_n\right]\text{tr}\left[\bS_n\bTheta\right].
\label{hq_2}
\end{eqnarray}
Note that the expression of $\hat{d}_0(0,\bTheta)$ is obtained using continuity at $t=0$ of $d_0$ and $\hat{d}_0$ together with the Woodbary matrix identity in the following way
\begin{eqnarray*}
    \hat{d}_0(0,\bTheta)&=&\lim\limits_{t\to0} t~ \text{tr}\left[(\bS_n+t\bI_p)^{-1}\bTheta\right]= \lim\limits_{t\to0} t~ \text{tr}\left[ 
    \frac{1}{t}\bTheta-t\frac{\bY_n}{\sqrt{n}}\left(\frac{1}{n}\bY_n^\top\bY_n+t\bI_n\right)^{-1}\frac{\bY^\top_n}{\sqrt{n}} \bTheta\right]\\
    &=& \tr(\bTheta) - \lim\limits_{t\to0} \text{tr}\left[\frac{\bY_n}{\sqrt{n}}\left(\frac{1}{n}\bY^\top_n\bY_n+t\bI_n\right)^{-1}\frac{\bY_n^\top}{\sqrt{n}} \bTheta\right]\\
    &=& \text{tr}(\bTheta) -  \text{tr}\left[\bY_n(\bY^\top_n\bY_n)^{-1}\bY_n^\top \bTheta\right]=\text{tr}(\bTheta) -  \text{tr}\left[\bS_n\bS_n^+\bTheta\right]\\
    &=& \text{tr}\left[(\bI_p-\bS_n\bS_n^+)\bTheta\right]\,,
\end{eqnarray*}
since the $n\times n$ companion matrix $\bY^\top_n\bY_n$ is non-singular for $p>n$.

\section{Shrinkage estimation of the precision matrix with the Moore-Penrose inverse}\label{sec:prec_MP-supp}
\subsection{Oracle estimator}
The proof of the following theorem follows directly from Theorem 2.1 and Corollary 2.1. 

\begin{theorem}\label{th:sh-prec-MP-oracle} 
Let $\bY_n$ fulfill the stochastic representation (1). Then, under Assumptions \textbf{(A1)}-\textbf{(A2)}, it holds that
   \begin{eqnarray*}
          \left|\alpha_{MP;n}^* - \alpha_{MP}^*\right|\stackrel{a.s.}{\rightarrow} 0,~~\text{and}~~\left|\beta_{MP;n}^* - \beta_{MP}^*\right|\stackrel{a.s.}{\rightarrow} 0~~
   \text{for} \quad p/n \rightarrow c \in (1,\infty) 
~ \text{as} ~~ n \rightarrow \infty
   \end{eqnarray*}
with
{\footnotesize
\begin{eqnarray}\label{alp_prec-opt}
  \alpha_{MP}^*&=& \dfrac{d_1(\bSigma) -d_1\left(\dfrac{\bSigma^2\boldsymbol{\Pi}_0}{||\bSigma\boldsymbol{\Pi}_0||_F}\right) \tr\left(\dfrac{\bSigma\boldsymbol{\Pi}_0}{||\bSigma\boldsymbol{\Pi}_0||_F}\right)}{-\dfrac{1}{h_2}\left(d_2(\bSigma^2)-d_1(\bSigma^2)\dfrac{h_3}{h_2}\right)
  -\dfrac{1}{h_2} d^2_1\left(\dfrac{\bSigma^2\boldsymbol{\Pi}_0}{||\bSigma\boldsymbol{\Pi}_0||_F}\right) },\\[0.3cm]  \beta_{MP}^*&=&\dfrac{\left(d_2(\bSigma^2)-d_1(\bSigma^2)\dfrac{h_3}{h_2}\right)\tr\left(\dfrac{\bSigma\boldsymbol{\Pi}_0}{||\bSigma\boldsymbol{\Pi}_0||_F}\right)+d_1(\bSigma)d_1\left(  \dfrac{\bSigma^2\boldsymbol{\Pi}_0}{||\bSigma\boldsymbol{\Pi}_0||_F} \right)}{\left(d_2(\bSigma^2)-d_1(\bSigma^2)\dfrac{h_3}{h_2}\right)+d^2_1\left(\dfrac{\bSigma^2\boldsymbol{\Pi}_0}{||\bSigma\boldsymbol{\Pi}_0||_F}\right) }||\bSigma\boldsymbol{\Pi}_0||^{-1}_F\,, \label{bet_prec-opt}
  \end{eqnarray}
}
where $d_i$ and $h_j$ are given in (11) and (14) for $i=1,2$ and $j=2,3$, respectively.
\end{theorem}

\subsection{Bona-fide estimator: Proof of Theorem 3.2}
\begin{proof}
    With some tedious but straightforward computations, we get the following equalities
\begin{eqnarray*}
d_1\left(\frac{1}{p}\bSigma\right)&=&\frac{1}{v(0)}\left(\frac{1}{c_n v(0)}-d_1\left(\frac{1}{p}\bI_p\right)\right),\\
d_1\left(\frac{1}{p}\bSigma^2\right)&=& \frac{1}{[v(0)]^2}\left(\frac{1}{p}\tr[\bSigma]+d_1\left(\frac{1}{p}\bI_p\right)-\frac{2}{c_n v(0)}\right),\\
d_1\left(\frac{1}{p}\bSigma^2\boldsymbol{\Pi}_0\right)&=&
\frac{1}{[v(0)]^2}\left(\frac{1}{p}\tr[\boldsymbol{\bSigma\Pi}_0]+d_1\left(\frac{1}{p}\boldsymbol{\Pi}_0\right)\right)-\frac{2}{[v(0)]^3}\left(\frac{1}{p}\tr[\boldsymbol{\Pi}_0]-
d_0\left(0,\frac{1}{p}\boldsymbol{\Pi}_0\right)\right),\\
d_2\left(\frac{1}{p}\bSigma^2\right)&=& \frac{1}{v(0)} d_1\left(\frac{1}{p}\bSigma^2\right) - \frac{1}{[v(0)]^2} \left(d_1\left(\frac{1}{p}\bSigma\right)-d_2\left(\frac{1}{p}\bI_p\right)\right) ,
\end{eqnarray*}
by using that
$\bSigma_n(v(0)\bSigma_n+\bI)^{-1}=\frac{1}{v(0)}\left(\bI_p -  (v(0)\bSigma_n+\bI_p)^{-1}\right).$
Those simplifications show how these quantities can be consistently estimated using the results presented in Section \ref{estimation_full}. 
\end{proof}

\section{Shrinkage estimation of the precision matrix with the ridge inverse}\label{sec:prec_r-supp}
\subsection{Oracle estimators}
The proofs of the following theorems follow directly from Theorem 2.2 and Corollary 2.4. 
\begin{theorem}\label{th_prec_R} Let $\bY_n$ fulfill the stochastic representation (1). Then, under Assumptions \textbf{(A1)}-\textbf{(A2)} for any $t_0>0$, it holds that
   \begin{eqnarray*}
          \left|\alpha_{R;n}^*(t_0) - \alpha_R^*(t_0)\right|\stackrel{a.s.}{\rightarrow} 0,~~\text{and}~~\left|\beta_{R;n}^*(t_0) - \beta_R^*(t_0)\right|\stackrel{a.s.}{\rightarrow} 0~~
   \text{for} \quad p/n \rightarrow c \in (0,\infty) 
~ \text{as} ~~ n \rightarrow \infty
   \end{eqnarray*}
with
{\footnotesize
\begin{eqnarray*}\label{alp_prec-opt-R}
  \alpha_R^*(t_0)&=& \dfrac{t_0^{-1}d_0(t_0,\bSigma) -t_0^{-1}d_0\left(t_0,\dfrac{\bSigma^2\boldsymbol{\Pi}_0}{||\bSigma\boldsymbol{\Pi}_0||_F}\right) \tr\left(\dfrac{\bSigma\boldsymbol{\Pi}_0}{||\bSigma\boldsymbol{\Pi}_0||_F}\right)}{t_0^{-2}d_0\left(t_0,\bSigma^2\right)+t_0^{-1}v^{(1)}(t_0)d_1\left(t_0,\bSigma^2\right)
  -t_0^{-2}d^2_0\left(t_0,\dfrac{\bSigma^2\boldsymbol{\Pi}_0}{||\bSigma\boldsymbol{\Pi}_0||_F}\right) },\\[0.3cm]
\beta_R^*(t_0)&=&\dfrac{\left(t_0^{-2}d_0\left(t_0,\bSigma^2\right)+t_0^{-1}v^{(1)}(t_0)d_1\left(t_0,\bSigma^2\right)\right)\tr\left(\dfrac{\bSigma\boldsymbol{\Pi}_0}{||\bSigma\boldsymbol{\Pi}_0||_F}\right)-t_0^{-2}d_0(t_0,\bSigma)d_0\left(t_0,\dfrac{\bSigma^2\boldsymbol{\Pi}_0}{||\bSigma\boldsymbol{\Pi}_0||_F} \right)}{t_0^{-2}d_0\left(t_0,\bSigma^2\right)+t_0^{-1}v^{(1)}(t_0)d_1\left(t_0,\bSigma^2\right)
  -t_0^{-2}d^2_0\left(t_0,\dfrac{\bSigma^2\boldsymbol{\Pi}_0}{||\bSigma\boldsymbol{\Pi}_0||_F}\right) }||\bSigma\boldsymbol{\Pi}_0||^{-1}_F\,, \label{bet_prec-opt-R}
  \end{eqnarray*}
}
where $d_i(t_0,.)$ and $v^{(1)}(t_0)$ are given in (22) and (23) for $i=1,2$.
\end{theorem}

\begin{theorem}\label{th_prec_R-L} Let $\bY_n$ fulfill the stochastic representation (1). Then, under Assumptions \textbf{(A1)}-\textbf{(A2)} for any $t>0$, it holds that
   \begin{eqnarray*}\frac{1}{p}\left|L^2_{R;n,2}(t) - L^2_{R;2}(t)\right|\stackrel{a.s.}{\rightarrow} 0,~~
   \text{for} \quad p/n \rightarrow c \in (0,\infty) 
~ \text{as} ~~ n \rightarrow \infty\qquad\text{with}
   \end{eqnarray*}
\begin{eqnarray*}\label{L_prec-opt-R}
L^2_{R;2}(t) &=&\frac{\left[d_0\left(t,\bSigma\right)
- d_0\left(t,\dfrac{\bSigma^2\boldsymbol{\Pi}_0}{||\bSigma\boldsymbol{\Pi}_0||_F}\right)\emph{tr}\left(\dfrac{\bSigma\boldsymbol{\Pi}_0}{||\bSigma\boldsymbol{\Pi}_0||_F}\right)\right]^2
}{d_0\left(t,\bSigma^2\right)+tv^{(1)}(t)d_1\left(t,\bSigma^2\right)
  -d^2_0\left(t,\dfrac{\bSigma^2\boldsymbol{\Pi}_0}{||\bSigma\boldsymbol{\Pi}_0||_F}\right) },
  \end{eqnarray*}
where $d_i(t,.)$ and $v^{(1)}(t)$ are given in (22) and (23) for $i=1,2$.
\end{theorem}



\subsection{Bona-fide estimators: Proof of Theorem 3.3}
\begin{proof}
It holds that
\begin{eqnarray*}
d_0\left(t_0,\frac{1}{p}\bSigma\right) &=&  \frac{1}{p}\text{tr}\left[\left(v(t_0)\bSigma+\bI_p\right)^{-1}\bSigma\right]=\frac{1}{c_nv(t_0)}-\frac{t_0}{c_n},\\
d_0\left(t_0,\frac{1}{p}\bSigma^2\right) &=& \frac{1}{p}\text{tr}\left[\left(v(t_0)\bSigma+\bI_p\right)^{-1}\bSigma^2\right]=\frac{1}{v(t_0)}\left( \frac{1}{p}\text{tr}\left[\bSigma\right]-\frac{1}{c_nv(t_0)}+\frac{t_0}{c_n}\right),\\
d_0\left(t_0,\frac{1}{p}\bSigma^2\boldsymbol{\Pi}_0\right) &=& \frac{1}{p}\text{tr}\left[\left(v(t_0)\bSigma+\bI_p\right)^{-1}\bSigma^2\boldsymbol{\Pi}_0\right]\\
&=&\frac{1}{v(t_0)}\left( \frac{1}{p}\text{tr}\left[\bSigma\boldsymbol{\Pi}_0\right]-\frac{1}{p}\text{tr}\left[\left(v(t_0)\bSigma+\bI_p\right)^{-1}\bSigma\boldsymbol{\Pi}_0\right]\right)\\
&=&\frac{1}{v(t_0)} \frac{1}{p}\text{tr}\left[\bSigma\boldsymbol{\Pi}_0\right]-\frac{1}{[v(t_0)]^2}\left(\frac{1}{p}\text{tr}\left[\boldsymbol{\Pi}_0\right]-d_0\left(t_0,\frac{1}{p}\boldsymbol{\Pi}_0\right)\right),\\
d_1\left(t_0,\frac{1}{p}\bSigma^2\right) &=& \frac{1}{p}\text{tr}\left\{\left(v(t_0)\bSigma+\bI_p\right)^{-1}\bSigma\left(v(t_0)\bSigma+\bI_p\right)^{-1}\bSigma^2\right\}\\
&&\hspace{-3.5cm}=\frac{1}{v(t_0)}\left( \frac{1}{p}\text{tr}\left[\left(v(t_0)\bSigma+\bI_p\right)^{-1}\bSigma^2\right]
-\frac{1}{p}\text{tr}\left[\left(v(t_0)\bSigma+\bI_p\right)^{-1}\bSigma\left(v(t_0)\bSigma+\bI_p\right)^{-1}\bSigma\right]\right)\\
&&\hspace{-3.5cm}=\frac{1}{[v(t_0)]^2}\left(\frac{1}{p}\text{tr}\left[\bSigma\right]+ d_1\left(t_0,\frac{1}{p}\bI_p\right)-\frac{2}{c_nv(t_0)}+\frac{2t_0}{c_n}\right),
\end{eqnarray*}
where $v(t_0)$, $v^{(1)}(t_0)$, $d_0\left(t_0,\frac{1}{p}\boldsymbol{\Pi}_0\right)$, $d_1\left(t_0,\frac{1}{p}\bI_p\right)$,
$q_1\left(\frac{1}{p}\boldsymbol{\Pi}_0\right)$, $q_2\left(\frac{1}{p}\boldsymbol{\Pi}_0\right)$ and $q_2\left(\frac{1}{p}\boldsymbol{\Pi}_0^2\right)$ are consistently estimated as in (7), \eqref{hd0t}, \eqref{hd1t}, \eqref{hq_1} and \eqref{hq_2}, respectively.
\end{proof}

\section{Shrinkage estimation of the precision matrix with the Moore-Penrose-ridge inverse}\label{sec:prec_MPR}

\subsection{Oracle estimators}
The application of Corollary 2.8 yields

\begin{theorem}\label{th_prec_MPR} Let $\bY_n$ fulfill the stochastic representation (1). Then, under Assumptions \textbf{(A1)}-\textbf{(A2)} for any $t_0>0$, it holds that
   \begin{eqnarray*}
          \left|\alpha_{MPR;n}^*(t_0) - \alpha_{MPR}^*(t_0)\right|\stackrel{a.s.}{\rightarrow} 0,~~\text{and}~~\left|\beta_{MPR;n}^*(t_0) - \beta_{MPR}^*(t_0)\right|\stackrel{a.s.}{\rightarrow} 0,
   \end{eqnarray*}
for $p/n \rightarrow c \in (0,\infty)$ as $n \rightarrow \infty$ with
{\footnotesize
\begin{eqnarray*}
  \alpha_{MPR}^*(t_0)&=& \dfrac{-v^{(1)}(t_0)d_1(t_0,\bSigma) +v^{(1)}(t_0)d_1\left(t_0,\dfrac{\bSigma^2\boldsymbol{\Pi}_0}{||\bSigma\boldsymbol{\Pi}_0||_F}\right) \tr\left(\dfrac{\bSigma\boldsymbol{\Pi}_0}{||\bSigma\boldsymbol{\Pi}_0||_F}\right)}{\grave{s}_2(t_0,\bSigma^2)
  -[v^{(1)}(t_0)]^{2}d^2_1\left(t_0,\dfrac{\bSigma^2\boldsymbol{\Pi}_0}{||\bSigma\boldsymbol{\Pi}_0||_F}\right) },\\[0.2cm]
\beta_{MPR}^*(t_0)&=&\dfrac{\grave{s}_2(t_0,\bSigma^2)\tr\left(\dfrac{\bSigma\boldsymbol{\Pi}_0}{||\bSigma\boldsymbol{\Pi}_0||_F}\right)-[v^{(1)}(t_0)]^{2}d_1(t_0,\bSigma)d_1\left(t_0,\dfrac{\bSigma^2\boldsymbol{\Pi}_0}{||\bSigma\boldsymbol{\Pi}_0||_F} \right)}{\grave{s}_2(t_0,\bSigma^2)
  -[v^{(1)}(t_0)]^{2}d^2_1\left(t_0,\dfrac{\bSigma^2\boldsymbol{\Pi}_0}{||\bSigma\boldsymbol{\Pi}_0||_F}\right) }||\bSigma\boldsymbol{\Pi}_0||^{-1}_F\,, 
  \end{eqnarray*}
}
where $d_1(t_0,.)$, $v^{(1)}(t_0)$ and $\grave{s}_2(t_0,.)$ are given in (22), (23), and (38), respectively. 
\end{theorem}

The optimal value of the tuning parameter $t$ should be chosen by maximizing $L^2_{MPR;n,2}(t)=L^2_{F;n,2}(\bS_n^{\pm}(t))$. Since the latter depends on the unknown population covariance matrix $\bSigma$, we first derive its asymptotic deterministic equivalent in
Theorem \ref{th_prec_MPR-L} by using Corollary 2.8.

\begin{theorem}\label{th_prec_MPR-L} Let $\bY_n$ fulfill the stochastic representation (1). Then, under Assumptions \textbf{(A1)}-\textbf{(A2)} for any $t>0$, it holds that
   \begin{eqnarray*}\frac{1}{p}\left|L^2_{MPR;n,2}(t) - L^2_{MPR;2}(t)\right|\stackrel{a.s.}{\rightarrow} 0,~~
   \text{for} \quad p/n \rightarrow c \in (0,\infty) 
~ \text{as} ~~ n \rightarrow \infty\quad\text{with}
   \end{eqnarray*}
   
{\small
\begin{eqnarray*}\label{L_prec-opt-MPR}
L^2_{MPR;2}(t) &=&\frac{\left[-v^{(1)}(t)d_1(t,\bSigma) +v^{(1)}(t)d_1\left(t,\dfrac{\bSigma^2\boldsymbol{\Pi}_0}{||\bSigma\boldsymbol{\Pi}_0||_F}\right)\emph{tr}\left(\dfrac{\bSigma\boldsymbol{\Pi}_0}{||\bSigma\boldsymbol{\Pi}_0||_F}\right)\right]^2
}{\grave{s}_2(t,\bSigma^2)
  -[v^{(1)}(t)]^2d^2_1\left(t,\dfrac{\bSigma^2\boldsymbol{\Pi}_0}{||\bSigma\boldsymbol{\Pi}_0||_F}\right) },
  \end{eqnarray*}}
where $d_1(t,.)$, $v^{(1)}(t)$ and $\grave{s}_2(t,.)$ are given in (22), (23), and (38), respectively.
\end{theorem}

\subsection{Bona-fide estimators}
Consistent estimators of $\alpha_{MPR}^*(t_0)$ and $\beta_{MPR}^*(t_0)$ are given in Theorem \ref{th_prec_MPR-est}, while Theorem \ref{th_prec_MPR-L-est} provides a consistent estimator of $L^2_{MPR;2}(t)$.

\begin{theorem}\label{th_prec_MPR-est} Let $\bY_n$ fulfill the stochastic representation (1). Then, under Assumptions \textbf{(A1)}-\textbf{(A2)} for any $t>0$, consistent estimators for $\alpha_{MPR}^*(t_0)$ and $\beta_{MPR}^*(t_0)$ are given by
{\footnotesize
\begin{eqnarray}
\label{alp_prec-opt-MPR-hat}
\hat{\alpha}_{MPR}^*(t_0)&=& \dfrac{-\hat{v}^{(1)}(t_0)\hat{d}_1\left(t_0,\frac{1}{p}\bSigma\right)\hat{q}_2\left(\frac{1}{p}\boldsymbol{\Pi}_0^2\right) +\hat{v}^{(1)}(t_0) \hat{d}_1\left(t_0,\frac{1}{p}\bSigma^2\boldsymbol{\Pi}_0\right) \hat{q}_1\left(\frac{1}{p}\boldsymbol{\Pi}_0\right)}{\hat{\grave{s}}_2\left(t_0,\frac{1}{p}\bSigma^2\right) \hat{q}_2\left(\frac{1}{p}\boldsymbol{\Pi}_0^2\right)
  -[\hat{v}^{(1)}(t_0)]^2\hat{d}^2_1\left(t_0,\frac{1}{p}\bSigma^2\boldsymbol{\Pi}_0\right) },\\[0.3cm]
\hat{\beta}_{MPR}^*(t_0)&=&\dfrac{\hat{\grave{s}}_2\left(t_0,\frac{1}{p}\bSigma^2\right)
\hat{q}_1\left(\frac{1}{p}\boldsymbol{\Pi}_0\right)
-[\hat{v}^{(1)}(t_0)]^2\hat{d}_1\left(t_0,\frac{1}{p}\bSigma\right)\hat{d}_1\left(t_0,\frac{1}{p}\bSigma^2\boldsymbol{\Pi}_0\right)}
{\hat{\grave{s}}_2\left(t_0,\frac{1}{p}\bSigma^2\right) \hat{q}_2\left(\frac{1}{p}\boldsymbol{\Pi}_0^2\right)
  -[\hat{v}^{(1)}(t_0)]^2\hat{d}^2_1\left(t_0,\frac{1}{p}\bSigma^2\boldsymbol{\Pi}_0\right) }\,,\label{bet_prec-opt-MPR-hat}
  \end{eqnarray}
}
with
{\scriptsize
\begin{eqnarray}\label{hat-grave_s2}
\hat{\grave{s}}_2\left(t_0,\frac{1}{p}\bSigma^2\right)&=& -\left\{[\hat{v}^{(1)}(t_0)]^2 \hat{d}_2\left(t_0,\frac{1}{p}\bSigma^2\right)-\frac{1}{2}\hat{v}^{(2)}(t_0) \hat{d}_1\left(t_0,\frac{1}{p}\bSigma^2\right)\right\} \\
&&\hspace{-1.4cm}+t_0\left\{\frac{1}{6}\hat{v}^{(3)}(t_0) \hat{d}_1\left(t_0,\frac{1}{p}\bSigma^2\right)-\hat{v}^{(1)}(t_0)\hat{v}^{(2)}(t_0) \hat{d}_2\left(t_0,\frac{1}{p}\bSigma^2\right)+[\hat{v}^{(1)}(t_0)]^3 \hat{d}_3\left(t_0,\frac{1}{p}\bSigma^2\right)\right\},\nonumber\\
\hat{d}_1\left(t_0,\frac{1}{p}\bSigma\right)
&=&\frac{[\hat{v}(t_0)]^{-2}+[\hat{v}^{(1)}(t_0)]^{-1}}{c_n}, \label{hd1-Sigma-MPR}\\
\hat{d}_1\left(t_0,\frac{1}{p}\bSigma^2\right)&=&
\frac{1}{\hat{v}(t_0)}\left\{d_0\left(t_0,\frac{1}{p}\bSigma^2\right) - 
\hat{d}_1\left(t_0,\frac{1}{p}\bSigma\right)\right\},\label{hd1-Sigma2-MPR}\\
\hat{d}_1\left(t_0,\frac{1}{p}\bSigma^2\boldsymbol{\Pi}_0\right)&=&\frac{1}{\hat{v}(t_0)}\hat{d}_0\left(t_0,\frac{1}{p}\bSigma^2\boldsymbol{\Pi}_0\right) +\frac{1}{[\hat{v}(t_0)]^2} \hat{d}_1\left(t_0,\frac{1}{p}\boldsymbol{\Pi}_0\right)\nonumber\\
&-& \frac{1}{[\hat{v}(t_0)]^3}\left\{ \frac{1}{p}\text{tr}\left\{\boldsymbol{\Pi}_0\right\}-
\hat{d}_0\left(t_0,\frac{1}{p}\boldsymbol{\Pi}_0\right)\right\},\label{hd1-Sigma2Pi0-MPR}\\
\hat{d}_2\left(t_0,\frac{1}{p}\bSigma^2\right)&=&
\frac{1}{\hat{v}(t_0)}\left\{\hat{d}_1\left(t_0,\frac{1}{p}\bSigma^2\right) - \frac{1}{c_n}\left(
\frac{1}{[\hat{v}(t_0)]^{3}}+\frac{\hat{v}^{(2)}(t_0)}{ 2[\hat{v}^{(1)}(t_0)]^3}\right)\right\},\label{hd2-Sigma2-MPR}\\
\hat{d}_3\left(t_0,\frac{1}{p}\bSigma^2\right)&=&
\frac{1}{\hat{v}(t_0)}\left\{\hat{d}_2\left(t_0,\frac{1}{p}\bSigma^2\right) - \frac{1}{c_n}\left(
\frac{1}{[\hat{v}(t_0)]^{4}}+\frac{1}{2}\frac{[\hat{v}^{(2)}(t_0)]^2}{[\hat{v}^{(1)}(t_0)]^5}-\frac{1}{6}\frac{\hat{v}^{(3)}(t_0)}{[\hat{v}^{(1)}(t_0)]^4}\right)
\right\},\label{hd3-Sigma2-MPR}
\end{eqnarray}
}
where $\hat{v}(t_0)$, $\hat{v}^{(1)}(t_0)$, $\hat{v}^{(2)}(t_0)$, $\hat{v}^{(3)}(t_0)$, $\hat{d}_0\left(t_0,\frac{1}{p}\boldsymbol{\Pi}_0\right)$, $\hat{d}_1\left(t_0,\frac{1}{p}\boldsymbol{\Pi}_0\right)$,
$\hat{q}_1\left(\frac{1}{p}\boldsymbol{\Pi}_0\right)$, $\hat{q}_2\left(\frac{1}{p}\boldsymbol{\Pi}_0^2\right)$, $\hat{d}_0\left(t_0,\frac{1}{p}\bSigma^2\right)$ and $\hat{d}_0\left(t_0,\frac{1}{p}\bSigma^2\boldsymbol{\Pi}_0\right)$ are given in (7), \eqref{hd0t}, \eqref{hd1t}, \eqref{hq_1}, \eqref{hq_2}, (56) and (57).
\end{theorem}

\begin{proof}[Proof of Theorem \ref{th_prec_MPR-est}:]
It holds that
\begin{eqnarray*}
d_1\left(t_0,\frac{1}{p}\bSigma\right)&=&\frac{1}{p}\text{tr}\left\{\left(v(t_0)\bSigma+\bI_p\right)^{-1}\bSigma\left(v(t_0)\bSigma+\bI_p\right)^{-1}\bSigma\right\}\\
&=&\frac{[v(t_0)]^{-2}-h_2(t)}{c_n}=\frac{[v(t_0)]^{-2}+[v^{(1)}(t_0)]^{-1}}{c_n},
\end{eqnarray*} 
\begin{eqnarray*}
d_1\left(t_0,\frac{1}{p}\bSigma^2\right)&=&\frac{1}{p}\text{tr}\left\{\left(v(t_0)\bSigma+\bI_p\right)^{-1}\bSigma\left(v(t_0)\bSigma+\bI_p\right)^{-1}\bSigma^2\right\}\\
&=&
\frac{1}{v(t_0)}\left\{d_0\left(t_0,\frac{1}{p}\bSigma^2\right) - 
d_1\left(t_0,\frac{1}{p}\bSigma\right)\right\},
\end{eqnarray*} 
\begin{eqnarray*}
d_1\left(t_0,\frac{1}{p}\bSigma^2\boldsymbol{\Pi}_0\right)&=&\frac{1}{p}\text{tr}\left\{\left(v(t_0)\bSigma+\bI_p\right)^{-1}\bSigma\left(v(t_0)\bSigma+\bI_p\right)^{-1}\bSigma^2\boldsymbol{\Pi}_0\right\}\\
&=&\frac{1}{v(t_0)}d_0\left(t_0,\frac{1}{p}\bSigma^2\boldsymbol{\Pi}_0\right) +\frac{1}{[v(t_0)]^2} d_1\left(t_0,\frac{1}{p}\boldsymbol{\Pi}_0\right)\\
&-& \frac{1}{[v(t_0)]^3}\left\{ \frac{1}{p}\text{tr}\left\{\boldsymbol{\Pi}_0\right\}-
d_0\left(t_0,\frac{1}{p}\boldsymbol{\Pi}_0\right)\right\},
\end{eqnarray*} \begin{eqnarray*}
d_2\left(t_0,\frac{1}{p}\bSigma^2\right)&=&\frac{1}{p}\text{tr}\left\{\left(v(t_0)\bSigma+\bI_p\right)^{-1}\bSigma\left(v(t_0)\bSigma+\bI_p\right)^{-1}\bSigma\left(v(t_0)\bSigma+\bI_p\right)^{-1}\bSigma^2\right\}\\
&=&
\frac{1}{v(t_0)}\left\{d_1\left(t_0,\frac{1}{p}\bSigma^2\right) - 
\frac{[v(t_0)]^{-3}-h_3(t)}{c_n}\right\}\\
&=&
\frac{1}{v(t_0)}\left\{d_1\left(t_0,\frac{1}{p}\bSigma^2\right) - \frac{1}{c_n}\left(
\frac{1}{[v(t_0)]^{3}}+\frac{v^{(2)}(t_0)}{ 2[v^{(1)}(t_0)]^3}\right)\right\},
\end{eqnarray*}
\begin{eqnarray*}
d_3\left(t_0,\frac{1}{p}\bSigma^2\right)&=&\frac{1}{p}\text{tr}\left\{\left(v(t_0)\bSigma+\bI_p\right)^{-1}\bSigma\left(v(t_0)\bSigma+\bI_p\right)^{-1}\bSigma\left(v(t_0)\bSigma+\bI_p\right)^{-1}\bSigma\left(v(t_0)\bSigma+\bI_p\right)^{-1}\bSigma^2\right\}\\
&=&
\frac{1}{v(t_0)}\left\{d_2\left(t_0,\frac{1}{p}\bSigma^2\right) - 
\frac{[v(t_0)]^{-4}-h_4(t)}{c_n}\right\}\\
&=&
\frac{1}{v(t_0)}\left\{d_2\left(t_0,\frac{1}{p}\bSigma^2\right) - \frac{1}{c_n}\left(
\frac{1}{[v(t_0)]^{4}}+\frac{1}{2}\frac{[v^{(2)}(t_0)]^2}{[v^{(1)}(t_0)]^5}-\frac{1}{6}\frac{v^{(3)}(t_0)}{[v^{(1)}(t_0)]^4}\right)
\right\}.
\end{eqnarray*}
\end{proof}

\begin{theorem}\label{th_prec_MPR-L-est} Let $\bY_n$ fulfill the stochastic representation (1). Then, under Assumptions \textbf{(A1)}-\textbf{(A2)} for any $t>0$, it holds that
\begin{eqnarray*}\frac{1}{p}\left|\hat{L}^2_{MPR;2}(t) - L^2_{MPR;2}(t)\right|\stackrel{a.s.}{\rightarrow} 0,~~
   \text{for} \quad p/n \rightarrow c \in (0,\infty) 
~ \text{as} ~~ n \rightarrow \infty
   \end{eqnarray*}
with
{\footnotesize
\begin{equation}
\label{L_prec-opt-MPR-hat}
\hat{L}^2_{MPR;2}(t) =\frac{1}{\hat{q}_2\left(\dfrac{1}{p}\boldsymbol{\Pi}_0^2\right)}\frac{[\hat{v}^{(1)}(t)]^2\left[\hat{d}_1\left(t,\dfrac{1}{p}\bSigma\right)\hat{q}_2\left(\dfrac{1}{p}\boldsymbol{\Pi}_0^2\right)
-\hat{d}_1\left(t,\dfrac{1}{p}\bSigma^2\boldsymbol{\Pi}_0\right) \hat{q}_1\left(\frac{1}{p}\boldsymbol{\Pi}_0\right)\right]^2
}{\hat{\grave{s}}_2\left(t,\frac{1}{p}\bSigma^2\right)\hat{q}_2\left(\frac{1}{p}\boldsymbol{\Pi}_0^2\right)
  -[\hat{v}^{(1)}(t)]^2\hat{d}^2_1\left(t,\frac{1}{p}\bSigma^2\boldsymbol{\Pi}_0\right) },
\end{equation}
}
where $\hat{v}^{(1)}(t)$, 
$\hat{q}_1\left(\frac{1}{p}\boldsymbol{\Pi}_0\right)$, $\hat{q}_2\left(\frac{1}{p}\boldsymbol{\Pi}_0^2\right)$, $\hat{\grave{s}}_2(t,\bSigma^2)$, $\hat{d}_1\left(t,\dfrac{1}{p}\bSigma\right)$ and $\hat{d}_1\left(t,\dfrac{1}{p}\bSigma^2\boldsymbol{\Pi}_0\right)$ are given in (7), \eqref{hq_1}, \eqref{hq_2}, \eqref{hat-grave_s2}, \eqref{hd1-Sigma-MPR} and \eqref{hd1-Sigma2Pi0-MPR}, respectively.
\end{theorem}

In the case of $t=0$, a consistent estimator for $L^2_{MPR;2}(0)$ is constructed by using the properties of the Moore-Penrose inverse. Namely, it holds that
{\footnotesize
\begin{equation}
\label{L0_prec-opt-MPR-hat}
\hat{L}^2_{MPR;2}(0) =\frac{1}{\hat{q}_2\left(\dfrac{1}{p}\boldsymbol{\Pi}_0^2\right)}\frac{[\hat{v}^{(1)}(0)]^2\left[\hat{d}_1\left(\dfrac{1}{p}\bSigma\right)\hat{q}_2\left(\dfrac{1}{p}\boldsymbol{\Pi}_0^2\right)
-{\hat{d}_1\left(0, \dfrac{1}{p}\bSigma^2\boldsymbol{\Pi}_0\right)} \hat{q}_1\left(\frac{1}{p}\boldsymbol{\Pi}_0\right)\right]^2
}{\hat{s}_2\left(\frac{1}{p}\bSigma^2\right)\hat{q}_2\left(\frac{1}{p}\boldsymbol{\Pi}_0^2\right)
  -[\hat{v}^{(1)}(0)]^2{\hat{d}^2_1\left(0, \frac{1}{p}\bSigma^2\boldsymbol{\Pi}_0\right) }}
\end{equation}
}
with
\begin{eqnarray}\label{hat-s2}
\hat{s}_2\left(\frac{1}{p}\bSigma^2\right)&=& -\left\{[\hat{v}^{(1)}(0)]^2 \hat{d}_2\left(\frac{1}{p}\bSigma^2\right)-\frac{1}{2}\hat{v}^{(2)}(0) \hat{d}_1\left(\frac{1}{p}\bSigma^2\right)\right\},
\end{eqnarray}
where $\hat{v}^{(1)}(0)$, $\hat{v}^{(2)}(0)$, $\hat{q}_1\left(\frac{1}{p}\boldsymbol{\Pi}_0\right)$, $\hat{q}_2\left(\frac{1}{p}\boldsymbol{\Pi}_0^2\right)$, $\hat{d}_1\left(\dfrac{1}{p}\bSigma\right)$, $\hat{d}_1\left(0, \dfrac{1}{p}\bSigma^2\boldsymbol{\Pi}_0\right)$, $\hat{d}_1\left(\frac{1}{p}\bSigma^2\right)$ and $\hat{d}_2\left(\frac{1}{p}\bSigma^2\right)$ are given in (8), \eqref{hq_1}, \eqref{hq_2}, (48), (49), (50) and (51), respectively. Furthermore, if $\boldsymbol{\Pi}_0=\bI_p$, then one can simplify the computation of $\hat{d}_1\left(\dfrac{1}{p}\bSigma^2\boldsymbol{\Pi}_0\right)$ by using $d_0\left(0,\frac{1}{p}\bI_p\right)=\frac{c_n-1}{c_n}$.

To find the optimal value of the tuning parameter $t$, we first maximize $\hat{L}^2_{MPR;2}(t)$ over the open interval $(0,\infty)$ to get $t^*$. If $\hat{L}^2_{MPR;2}(t^*)>\hat{L}^2_{MPR;2}(0)$, then $t^*$ should be used as the optimal value of the tuning parameter. Otherwise, one opts for the Moore-Penrose inverse and the optimal shrinkage 
intensities are determined as in Section 3.1.1.

\section{Oracle estimator of the GMV portfolio weights with the Moore-Penrose inverse}

Next, we derive the expression of the oracle estimator of the GMV portfolio weights.
\begin{theorem}\label{MP_GMV_as}
 Let $\bY_n$ possess the stochastic representation as in (1). 
Then, under Assumptions \textbf{(A1)}-\textbf{(A2)}, it holds that
    \begin{equation}\label{alpha-as}
  \left|\alpha_n^*-\alpha^*\right| \stackrel{a.s.}{\rightarrow} 0
\end{equation}
for $p/n \to c \in (1, +\infty)$ as $n \to \infty$ with
\begin{equation}\label{psi-star}
\alpha^*=\left\{\begin{array}{cc}
  \frac{p\bb^\top\bSigma \bb-\frac{d_1\left(\bOne\bb^\top\bSigma  \right)}{d_1\left(\frac{\bOne\bOne^\top}{p}\right)}}{p\bb^\top\bSigma \bb-2\frac{d_1\left(\bOne\bb^\top\bSigma \right)}{d_1\left( \frac{\bOne\bOne^\top}{p}\right)}+\frac{d_3\left(\frac{\bOne\bOne^\top}{p}\right)}{d_1^2\left(\frac{\bOne\bOne^\top}{p}\right)}}   & ~~~\text{if~~ $p\bb^\top\bSigma\bb = O(1),$}\\[0.8cm]
  1, & \text{otherwise.}
\end{array}\right.
\end{equation}
\end{theorem}

\begin{proof}[Proof of Theorem \ref{MP_GMV_as}:]
First, it is noted that in case of Moore-Penrose inverse, i.e., $\bS_n^{\#}(t)=\bS_n^+$ the function $L_{n;2}(t)$ is independent of $t$ in the case of the Moore-Penrose inverse, i.e., $\bS_n^{\#}(t)=\bS_n^+$, so we just need to find the asymptotic behavior of $\alpha_n^*(t)=\alpha_n^*$ with
\begin{eqnarray*}
\alpha_n^* &=& 
    \frac{
        \bb^\top \bSigma \left(\bb -  \mathbf{w}_{\bS_n^{+}}\right)
    }{
        \left(\bb -  \mathbf{w}_{\bS_n^{+}}\right)^\top \bSigma\left(\bb- \mathbf{w}_{\bS_n^{+}}\right)
    }=      \frac{
        p\bb^\top \bSigma \bb -  \frac{\bb^\top \bSigma \bS_n^+\bOne}{\frac{1}{p}\bOne^\top\bS_n^+\bOne}
    }{
        p\bb^\top\bSigma\bb -  2\frac{\bb^\top\bSigma\bS_n^+\bOne}{\frac{1}{p}\bOne^\top\bS_n^+\bOne}+ \frac{\frac{1}{p}\bOne^\top\bS_n^+\bSigma\bS_n^+\bOne}{(\frac{1}{p}\bOne^\top\bS_n^+\bOne)^2}
    }
\end{eqnarray*}

Next, we study the asymptotic behavior of $\bb^\top \bSigma\bS_n^+\bOne$ and $\frac{1}{p}\bOne^\top\bS_n^+\bOne$. Since $\tr\left(\frac{\bOne\bOne^\top}{p}\right)=1$, $$\tr\left(\bOne\bb^\top \bSigma\right)=\tr\left(1/2\left(\bOne\bb^\top+\bb\bOne^\top\right) \bSigma\right)\leq \lambda_{max}(\bSigma)\bb^\top\bOne=\lambda_{max}(\bSigma)<\infty,$$ 
and using Theorem 2.1, we get for $p,n\to\infty$ with $p/n\to c>1$ that
\begin{equation}\label{ratioterm}
   \frac{\bb^\top \bSigma\bS_n^+\bOne}{\frac{1}{p}\bOne^\top\bS_n^+\bOne}\overset{a.s.}{\longrightarrow} \frac{d_1\left(\bOne\bb^\top \bSigma\right)}{d_1\left(\frac{\bOne\bOne^\top }{p}\right)}\,,
\end{equation}
where due to Theorem 2.1, it holds that
\begin{eqnarray}\label{1orderterm}
\tr(\bS_n^+\bTheta)\overset{d.a.s.}{\longrightarrow}(-v'(0))d_1\left(\bTheta\right) \quad   
\text{with}\quad \bTheta\in\{\bOne\bb^\top \bSigma,~ \frac{1}{p}\bOne\bOne^\top\}.
\end{eqnarray}

To examine the asymptotic behavior of$\frac{1}{p}\bOne^\top\bS_n^+\bSigma\bS_n^+\bOne$, we consider the eigenvalue decomposition of $\bSigma=\sum\limits_{i=1}^p \mathbf{v}_i\mathbf{v}^\top_i\tau_i$ with $\tau_1,\ldots, \tau_p$ eigenvalues of $\bSigma$ and $\bv_1, \ldots, \bv_p$ the corresponding eigenvectors. It holds that
\begin{eqnarray*}
   \frac{1}{p}\bOne^\top\bS_n^+\bSigma\bS_n^+\bOne=\frac{1}{p}\sum\limits_{i=1}^p \bOne^\top\bS_n^+\mathbf{v}_i\mathbf{v}^\top_i\bS_n^+\bOne \tau_i=\sum\limits_{i=1}^p \left(\frac{1}{\sqrt{p}}\bOne^\top\bS_n^+\mathbf{v}^\top_i\right)^2\tau_i\,.
\end{eqnarray*}

Since $\tr\left(\frac{1}{\sqrt{p}}\mathbf{v}_i\bOne^\top\right)$ is bounded by one due to the Cauchy-Schwarz inequality, the application of Theorem 2.1 yields 
\begin{eqnarray*}
    \left|\frac{1}{\sqrt{p}}\bOne^\top\bS_n^+\mathbf{v}_i- (-v'(0))d_1\left( \frac{\mathbf{v}_i\bOne^\top}{\sqrt{p}}\right)\right|{a.s.}{\rightarrow} 0
\quad \text{for} \quad p/n \rightarrow c \in (1,\infty) \quad \text{as} \quad n \rightarrow \infty.
\end{eqnarray*}
Furthermore, due to $\frac{\mathbf{v}^\top_i\bOne}{\sqrt{p}}=O(1)$, one immediately gets for $(p, n)\to\infty$ using the trace inequality that 
\begin{eqnarray}\label{smallo}
  && \frac{1}{\sqrt{p}}\bOne^\top\bS_n^+\mathbf{v}_i- (-v'(0))d_1\left( \frac{\mathbf{v}_i\bOne^\top}{\sqrt{p}}\right)=o_{\mathbbm{P}}\left(\frac{\mathbf{v}^\top_i\bOne}{\sqrt{p}}\right),\\
  && (-v'(0))d_1\left( \frac{\mathbf{v}_i\bOne^\top}{\sqrt{p}}\right)=O\left(\frac{\mathbf{v}^\top_i\bOne}{\sqrt{p}}\right) \label{bigO}\,. 
\end{eqnarray}
The application of \eqref{smallo} and \eqref{bigO} results into
    \begin{eqnarray*}
\sum\limits_{i=1}^p \left(\frac{1}{\sqrt{p}}\bOne^\top\bS_n^+\mathbf{v}^\top_i\right)^2\tau_i&=&\sum\limits_{i=1}^p \left(\frac{1}{\sqrt{p}}\bOne^\top\bS_n^+\mathbf{v}^\top_i-(-v'(0))d_1\left( \frac{\mathbf{v}_i\bOne^\top}{\sqrt{p}}\right)+(-v'(0))d_1\left( \frac{\mathbf{v}_i\bOne^\top}{\sqrt{p}}\right)\right)^2\tau_i\\
&=& \sum\limits_{i=1}^p \underbrace{\left(\frac{1}{\sqrt{p}}\bOne^\top\bS_n^+\mathbf{v}^\top_i-(-v'(0))d_1\left( \frac{\mathbf{v}_i\bOne^\top}{\sqrt{p}}\right)\right)^2}_{o_{\mathbbm{P}}\left(\frac{\mathbf{v}^\top_i\bOne}{\sqrt{p}}\right)^2}\tau_i\\
&+& 2 \sum\limits_{i=1}^p \underbrace{\left(\frac{1}{\sqrt{p}}\bOne^\top\bS_n^+\mathbf{v}^\top_i-(-v'(0))d_1\left( \frac{\mathbf{v}_i\bOne^\top}{\sqrt{p}}\right)\right)}_{o_{\mathbbm{P}}\left(\frac{\mathbf{v}^\top_i\bOne}{\sqrt{p}}\right)} \underbrace{(-v'(0))d_1\left( \frac{\mathbf{v}_i\bOne^\top}{\sqrt{p}}\right)}_{O\left(\frac{\mathbf{v}^\top_i\bOne}{\sqrt{p}}\right)}\tau_i\\
&+& \sum\limits_{i=1}^p  (-v'(0))^2d^2_1\left( \frac{\mathbf{v}_i\bOne^\top}{\sqrt{p}}\right)\tau_i\\
&=& C\sum\limits_{i=1}^p o_{\mathbbm{P}}\left(\frac{\bOne^\top\mathbf{v}_i\mathbf{v}^\top_i\bOne}{p}\right)\tau_i+\sum\limits_{i=1}^p  (-v'(0))^2d^2_1\left( \frac{\mathbf{v}_i\bOne^\top}{\sqrt{p}}\right)\tau_i\,,
\end{eqnarray*}
where the constant $C$ is independent of $p$ and $n$, and
\begin{eqnarray*}
    \sum\limits_{i=1}^p o_{\mathbbm{P}}\left(\frac{\bOne^\top\mathbf{v}_i\mathbf{v}^\top_i\bOne}{p}\right)\tau_i &\leq&   o_{\mathbbm{P}}\left(\sum\limits_{i=1}^p\frac{\bOne^\top\mathbf{v}_i\mathbf{v}^\top_i\bOne}{p}\right) \tau_{max}=  o_{\mathbbm{P}}\left(\frac{\bOne^\top\sum\limits_{i=1}^p\mathbf{v}_i\mathbf{v}^\top_i\bOne}{p}\right)\tau_{max} \\
    &=&   o_{\mathbbm{P}}\left(\frac{1}{p}\bOne^\top\underbrace{[\bv_1, \ldots, \bv_p][\bv_1, \ldots, \bv_p]^\top}_{=\bI_p}\bOne\right)\tau_{max}=  o_{\mathbbm{P}}\left(\tau_{max}\right)\,.
\end{eqnarray*}

Since the maximum eigenvalue of $\bSigma$, i.e., $\tau_{max}=\lambda_{max}(\bSigma)$, is uniformly bounded in $p$, we get $o_{\mathbbm{P}}\left(\tau_{max}\right)=o_{\mathbbm{P}}\left(1\right)$ and the definitions of $d_1(\cdot)$ and $d_3(\cdot)$ from (11) yield
\begin{eqnarray*}
&&\sum\limits_{i=1}^p \left(\frac{1}{\sqrt{p}}\bOne^\top\bS_n^+\mathbf{v}^\top_i\right)^2\tau_i    \overset{d.a.s.}{\longrightarrow}    \sum\limits_{i=1}^p  (-v'(0))^2d^2_1\left( \frac{\mathbf{v}_i\bOne^\top}{\sqrt{p}}\right)\tau_i
\nonumber\\
&=&(v'(0))^2\sum\limits_{i=1}^p\text{tr}^2\left\{\left(v(0)\bSigma+\bI_p\right)^{-1}\left[\bSigma\left(v(0)\bSigma+\bI_p\right)^{-1}\right]\frac{\mathbf{v}_i\bOne^\top}{\sqrt{p}}\right\}\tau_i
\nonumber\\
&=& \frac{(v'(0))^2}{p}\sum\limits_{i=1}^p \bOne^\top \left(v(0)\bSigma+\bI_p\right)^{-1}\left[\bSigma\left(v(0)\bSigma+\bI_p\right)^{-1}\right] \mathbf{v}_i\mathbf{v}_i^\top \left(v(0)\bSigma+\bI_p\right)^{-1}\left[\bSigma\left(v(0)\bSigma+\bI_p\right)^{-1}\right]\bOne \tau_i
\nonumber\\
&=& \frac{(v'(0))^2}{p}\bOne^\top \left(v(0)\bSigma+\bI_p\right)^{-1}\left[\bSigma\left(v(0)\bSigma+\bI_p\right)^{-1}\right] \underbrace{\left(\sum\limits_{i=1}^p \mathbf{v}_i\mathbf{v}_i^\top \tau_i\right)}_{\bSigma}\left(v(0)\bSigma+\bI_p\right)^{-1}\left[\bSigma\left(v(0)\bSigma+\bI_p\right)^{-1}\right]\bOne 
\nonumber\\
&=& (v'(0))^2\tr\left\{ \left(v(0)\bSigma+\bI_p\right)^{-1}\left[\bSigma\left(v(0)\bSigma+\bI_p\right)^{-1}\right]^3 \frac{\bOne\bOne^\top}{p} \right\}
= (v'(0))^2d_3\left(\frac{\bOne\bOne^\top}{p}\right)
\,.
\end{eqnarray*}

Finally, the application of the last equality, \eqref{ratioterm}, \eqref{1orderterm} and the expression of $\alpha^*_n$ finish the proof of the theorem.
\end{proof}

\begin{remark}
 It should be noted that the condition $p\bb^\top\bSigma\bb = O(1)$ actually implies that the variance of the target portfolio is of order $O(1/p)$. This situation is quite natural since the variance of the GMV portfolio is also typically at most $1/p$, similar to that of the equally weighted portfolio, represented as $1/p\bOne$. On the other hand, if the rate of $\bb^\top\bSigma\bb$ exceeds $1/p$, then $\alpha_n^* \to 1$, indicating that the target portfolio may be suboptimal.
\end{remark}

\section{Additional results of numerical studies}\label{sec:S-num-study}

\subsection{Benchmarks: Precision matrix}\label{sec:benchmarks-prec}

The following benchmark approaches are considered for the estimation of the precision matrix:
\begin{itemize}
\item {\bf Empirical Bayes} ridge-type estimator of \cite{kubokawa2008estimation} given by
\begin{equation}\label{EBR-prec}
\widehat{\boldsymbol{\Pi}}_{EBR}=p \left((n-1)\bS_n+\tr[\bS_n]\bI_p\right)^{-1}.
\end{equation}
\item {\bf Optimal ridge} estimator of \cite{wang2015shrinkage} expressed as
\begin{equation}\label{OR-prec}
\widehat{\boldsymbol{\Pi}}_{OR}=
\hat{\alpha}_{OR;n}\left(\bS_n+\hat{\beta}_{OR;n}\bI_p\right)^{-1},
\end{equation}
where $\hat{\alpha}_{OR;n}=\hat{R}_{1,n}(\hat{\beta}_{OR;n})/\hat{R}_{2,n}(\hat{\beta}_{OR;n})$,
\[\hat{R}_{1,n}(\lambda)=\frac{\hat{a}_{1;n}(\lambda)}{1-c_n\hat{a}_{1;n}(\lambda)},\quad
\hat{R}_{2,n}(\lambda)=
\frac{\hat{a}_{1;n}(\lambda)}{(1-c_n\hat{a}_{1;n}(\lambda))^3}
-\frac{\hat{a}_{2;n}(\lambda)}{(1-c_n\hat{a}_{1;n}(\lambda))^4}
\]
with 
$\hat{a}_{1;n}(\lambda)=1-\frac{1}{p}\tr[(\bS_n/\lambda+\bI_p)^{-1}],\,
\hat{a}_{2;n}(\lambda)=\frac{1}{p}\tr[(\bS_n/\lambda+\bI_p)^{-1}-\frac{1}{p}\tr[(\bS_n/\lambda+\bI_p)^{-2},$ and $\hat{\beta}_{OR;n}$ minimizes
$L_{OR;n}=1-\hat{R}_{1,n}(\lambda)^2/\hat{R}_{2,n}(\lambda)$.

\item {\bf Inverse nonlinear shrinkage} estimator of the covariance matrix introduced in \cite{lw20}. For $i\in\{1,\ldots,p\}$, is given by
\begin{equation}\label{or_sh_cov_mat}
\bS_{NLSh}= \bU \text{diag}(d_1^{or},...,d_p^{or})\bU^\top, \quad d_i^{or}=\left\{
  \begin{array}{ll}
    \frac{d_i}{|1-c-c d_i \breve{m}_{F}(d_i)|^2},    & \text{if $d_i>0$},\\
    \frac{1}{(c-1)\breve{m}_{\underline{F}}(0)}, & \text{if $d_i=0$},
  \end{array}
 \right.
\end{equation}
where $\bU=(\mathbf{u}_1,...\mathbf{u}_p)$ is the matrix with the sample eigenvectors of $\bS_n$, $d_i$, $i=1,...,$ are the sample eigenvalues of $\bS_n$ and $\breve{m}_{F}(x)=\lim\limits_{z\to x}m_{F}(z)$ with $m_{F}(z)$ the limiting Stieltjes transform of the sample covariance matrix. A numerical approach to estimate $\breve{m}_{F}(x)$ is provided in \cite{lw20} and is available in the R-package \textit{HDShOP} (see \cite{HDShOP}). 
\item {\bf Oracle nonlinear shrinkage} estimator is derived for the loss function considered in \cite{ledoit2021shrinkage} and it is given by
\begin{equation}\label{or1_sh_cov_mat}
\bS_{oNLSh}= \bU \text{diag}(\tilde{d}_1^{or},...,\tilde{d}_p^{or})\bU^\top, \quad \tilde{d}_i^{or}=\frac{\mathbf{u}_i^\top \bSigma^2  \mathbf{u}_i}{\mathbf{u}_i^\top \bSigma\mathbf{u}_i}\,.
\end{equation}
\end{itemize}


\subsection{Additional figures: Precision matrix}\label{sec:S-fig-prec}

In Figure \ref{fig:shrinkage-prec-add}, we present the results of the simulation study obtained for the three suggested shrinkage estimators of the precision matrix and three benchmark approaches. The results of Figure \ref{fig:shrinkage-prec-add} complement the findings of Figure 3 by adding the values of the PRIAL computed for the empirical Bayes estimator and the inverse nonlinear shrinkage estimator. In Figure \ref{fig:shrinkage-prec-time}, we report the average computational time of the most relevant estimators presented in Figure \ref{fig:shrinkage-prec-add} based on 100 repetitions. The MP estimator is used as a benchmark since it is essentially a plug-in estimator, requiring only the inversion of the nonzero eigenvalues. As expected, the MP shrinkage estimator performs best in terms of computational efficiency, being closest to the benchmark. The NL shrinkage estimator ranks second, while all ridge-type estimators are the most computationally demanding. 

\begin{figure}[h!t]
\centering
\begin{tabular}{cc}
\hspace{-0.5cm}\includegraphics[width=7cm]{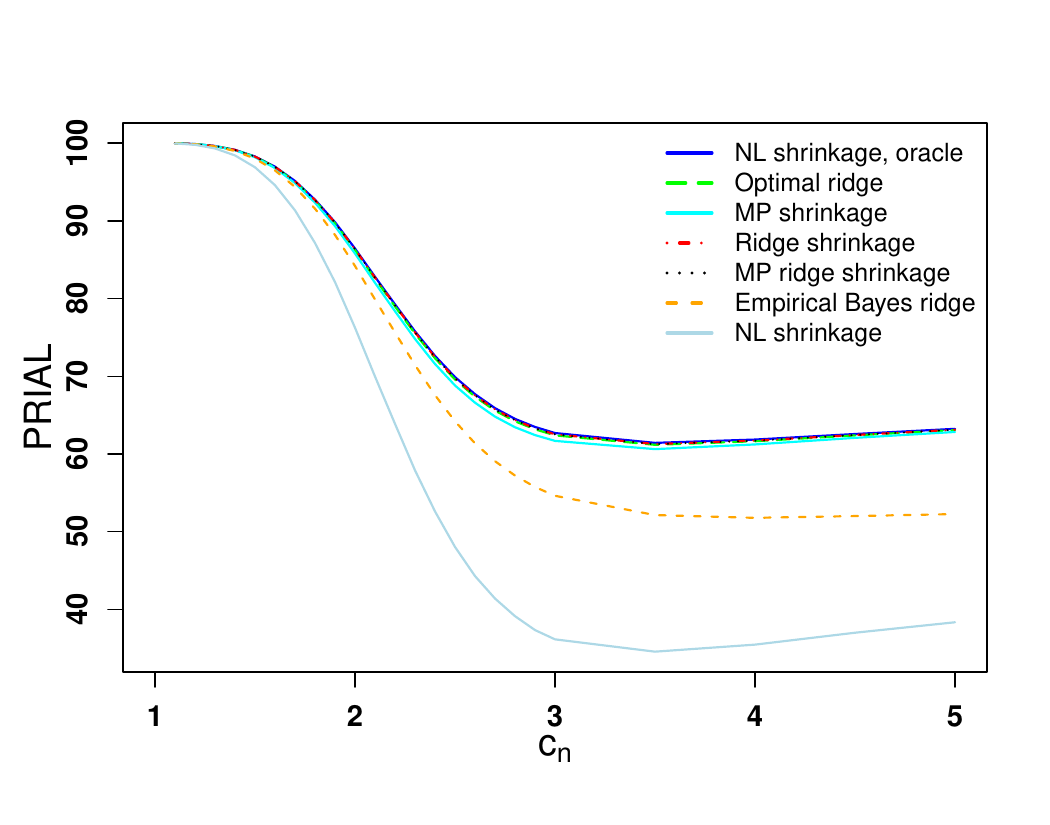}&
\hspace{-0.5cm}\includegraphics[width=7cm]{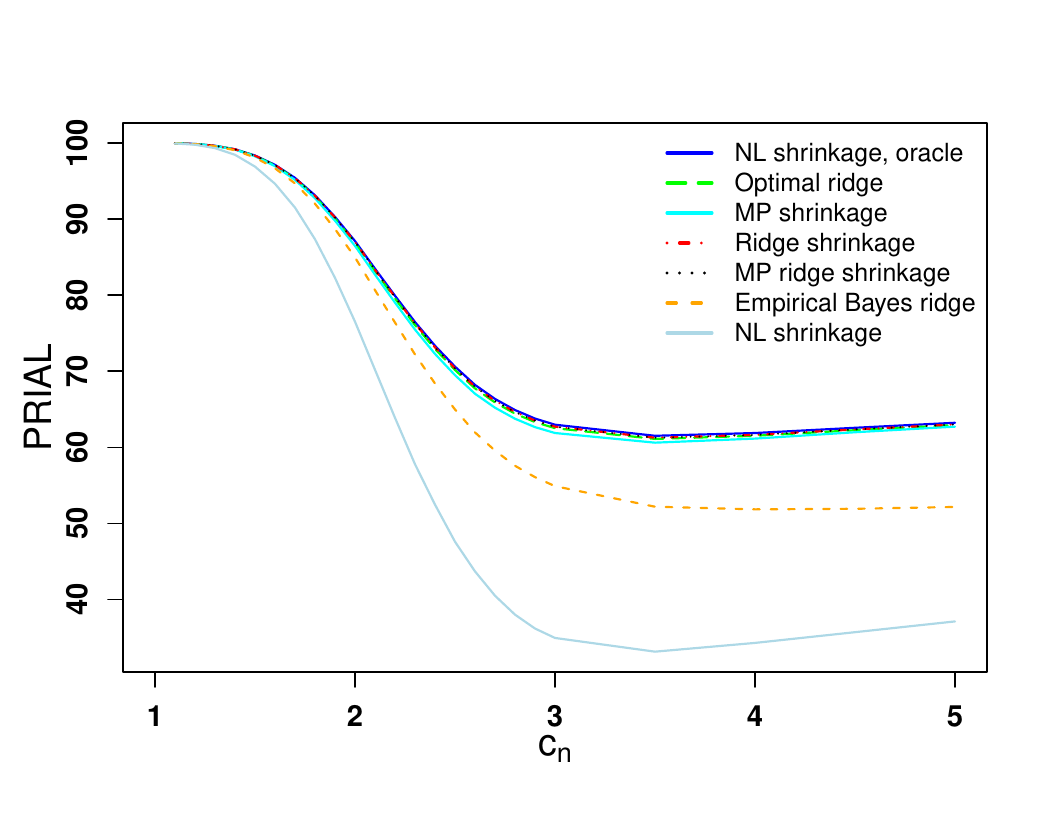}\\[-1cm]
\hspace{-0.5cm}\includegraphics[width=7cm]{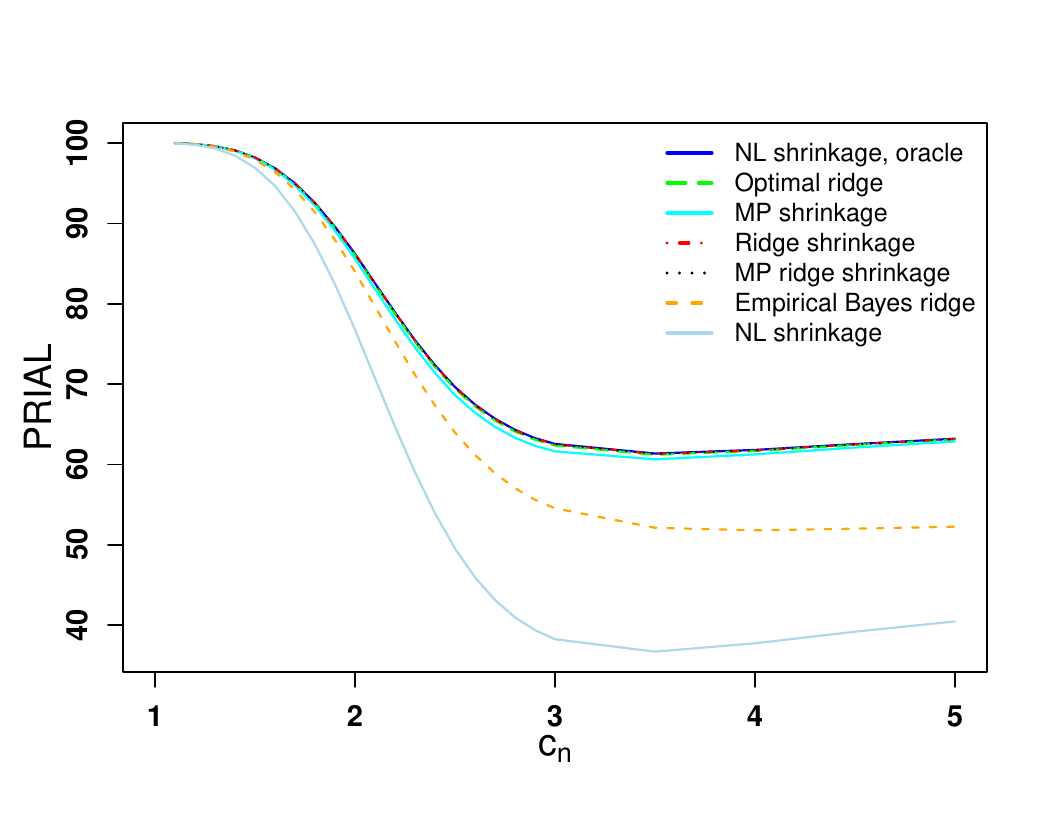}&
\hspace{-0.5cm}\includegraphics[width=7cm]{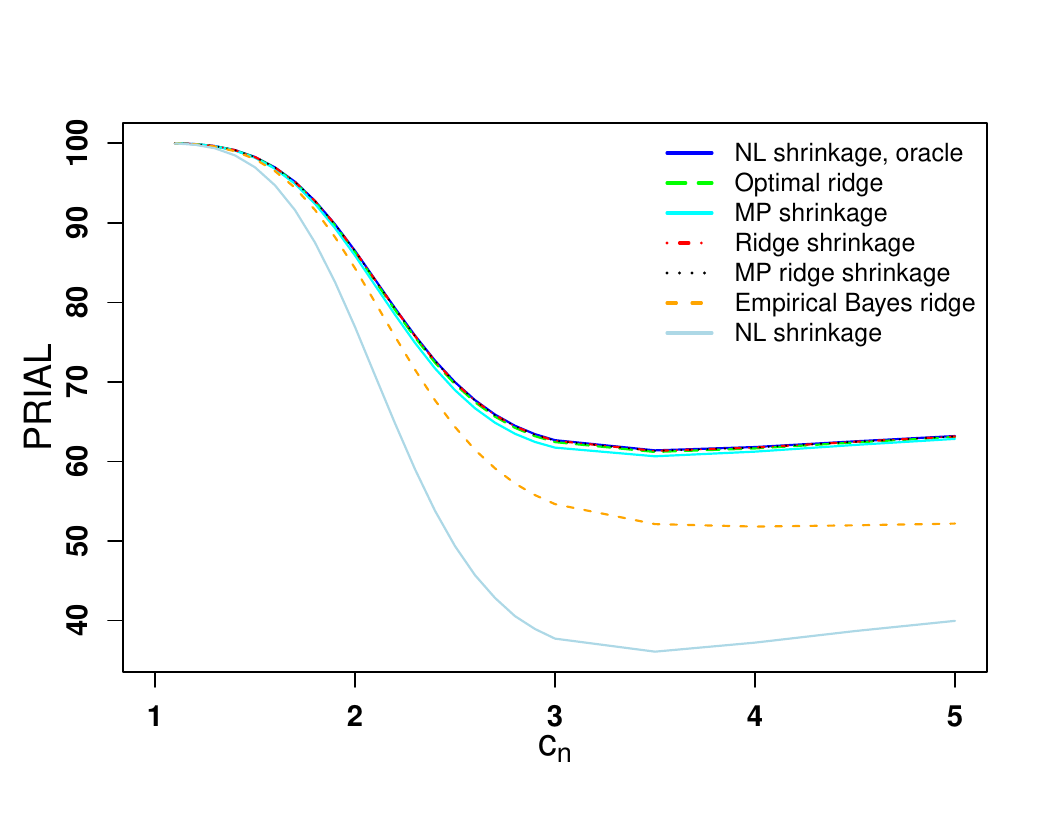}\\[-0.7cm]
\end{tabular}
 \caption{PRIAL for $c_n \in (1,5]$, $n=100$ (first row) and $n=250$ (second row) when the elements of $\bX_n$ are drawn from the normal distribution (first column) and scale $t$-distribution (second column).}
\label{fig:shrinkage-prec-add}
 \end{figure} 
\begin{figure}[h!t]
\centering
\begin{tabular}{cc}
\hspace{-0.5cm}\includegraphics[scale=0.35]{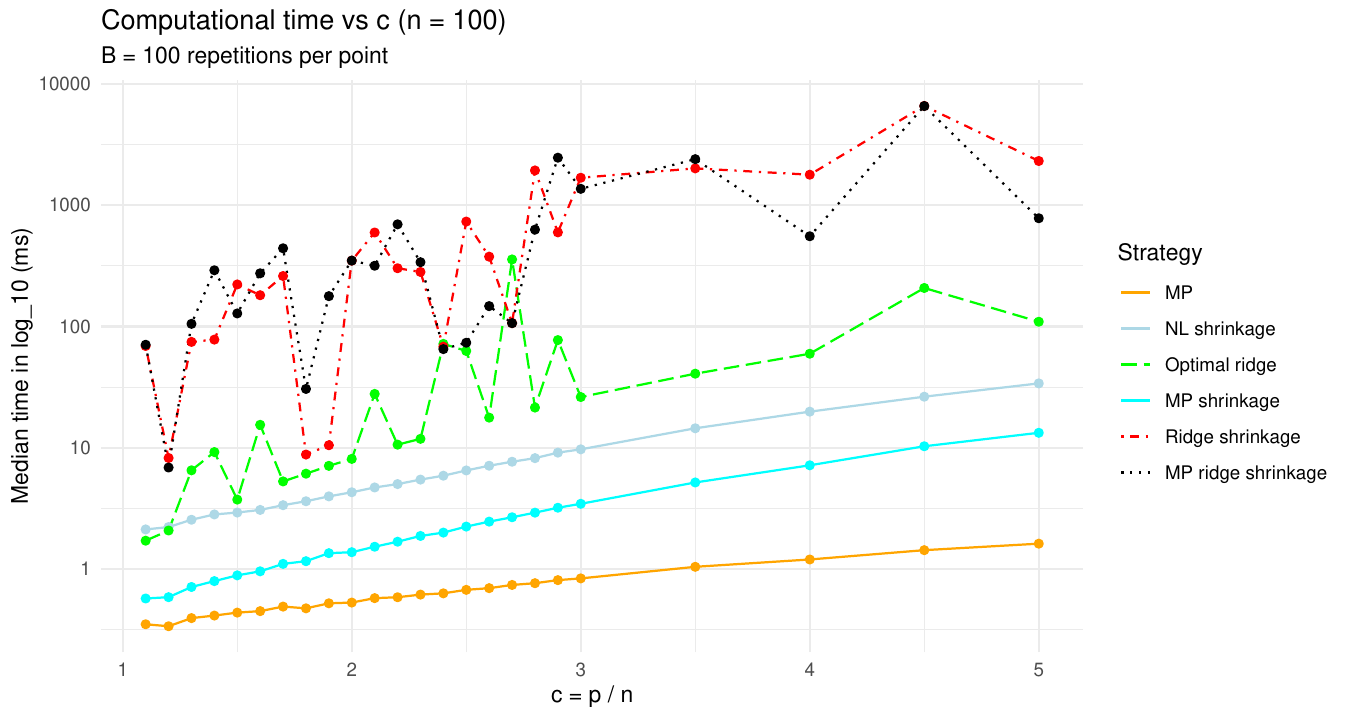}&
\hspace{-0.5cm}\includegraphics[scale=0.35]{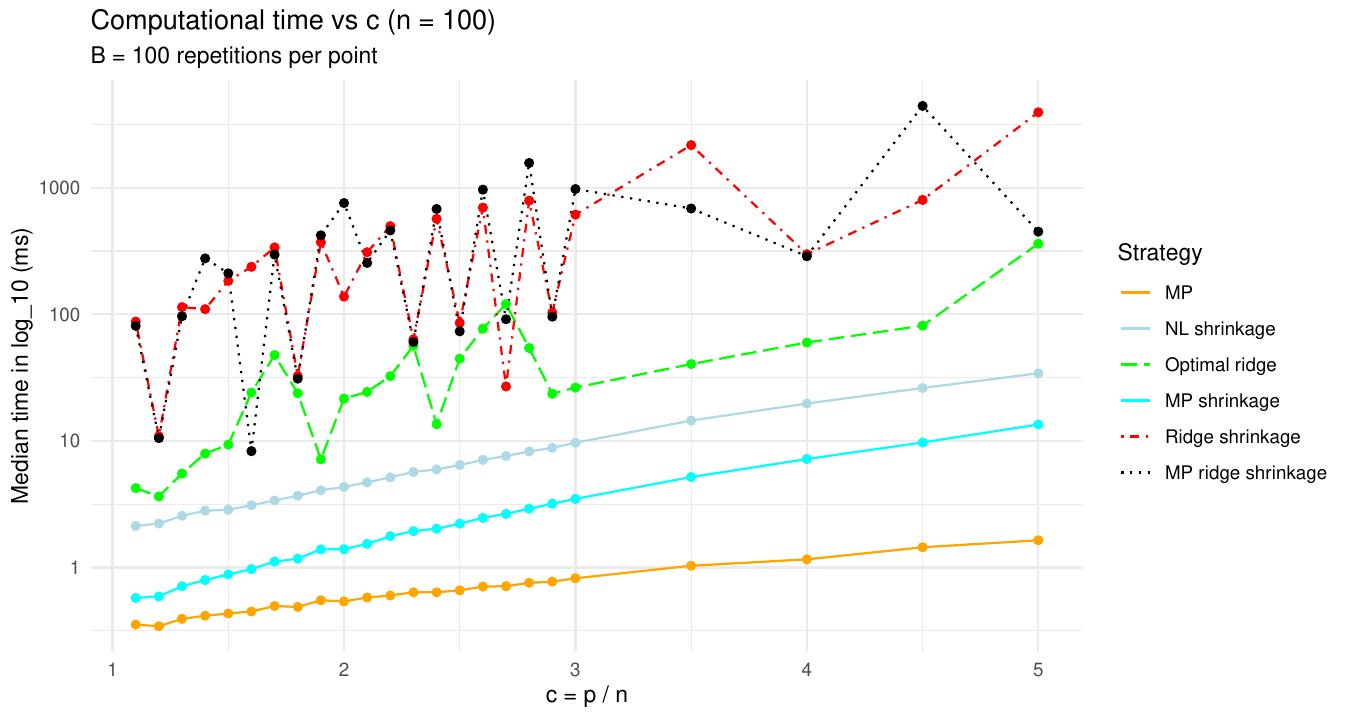}\\[-0.5cm]
\hspace{-0.5cm}\includegraphics[scale=0.35]{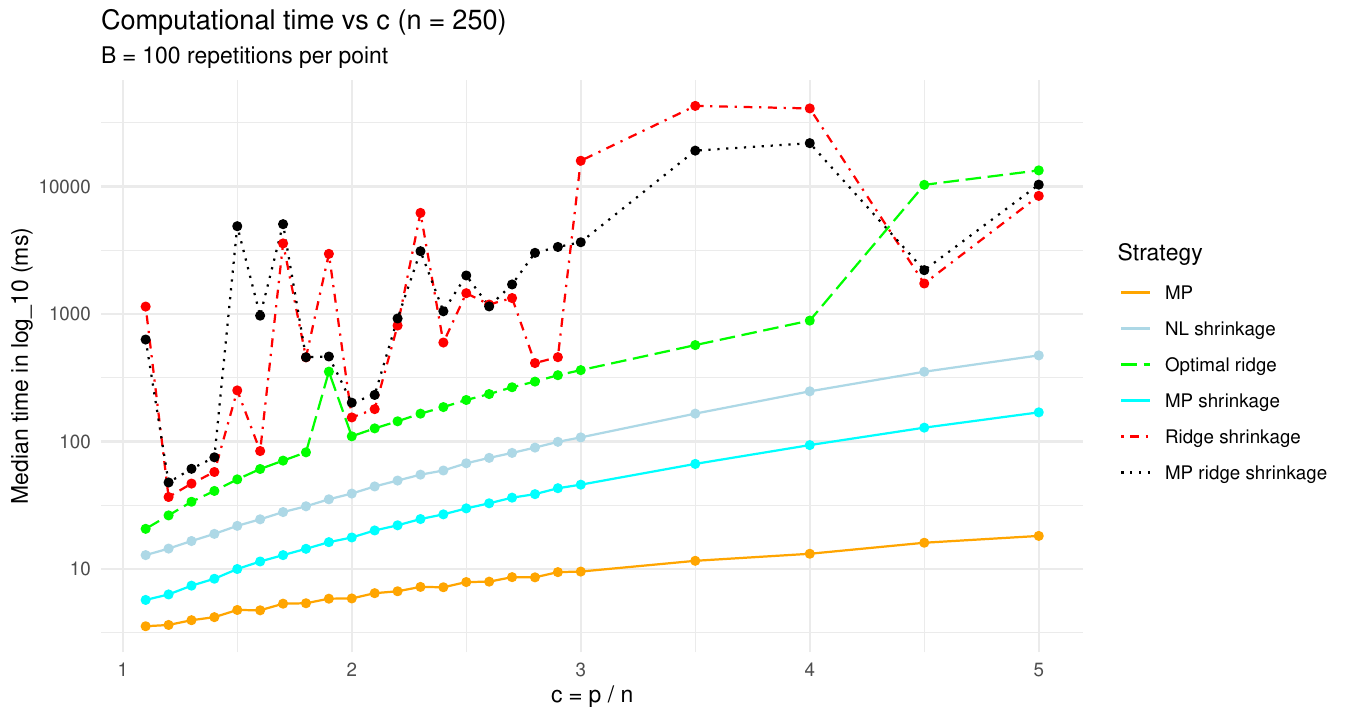}&
\hspace{-0.5cm}\includegraphics[scale=0.35]{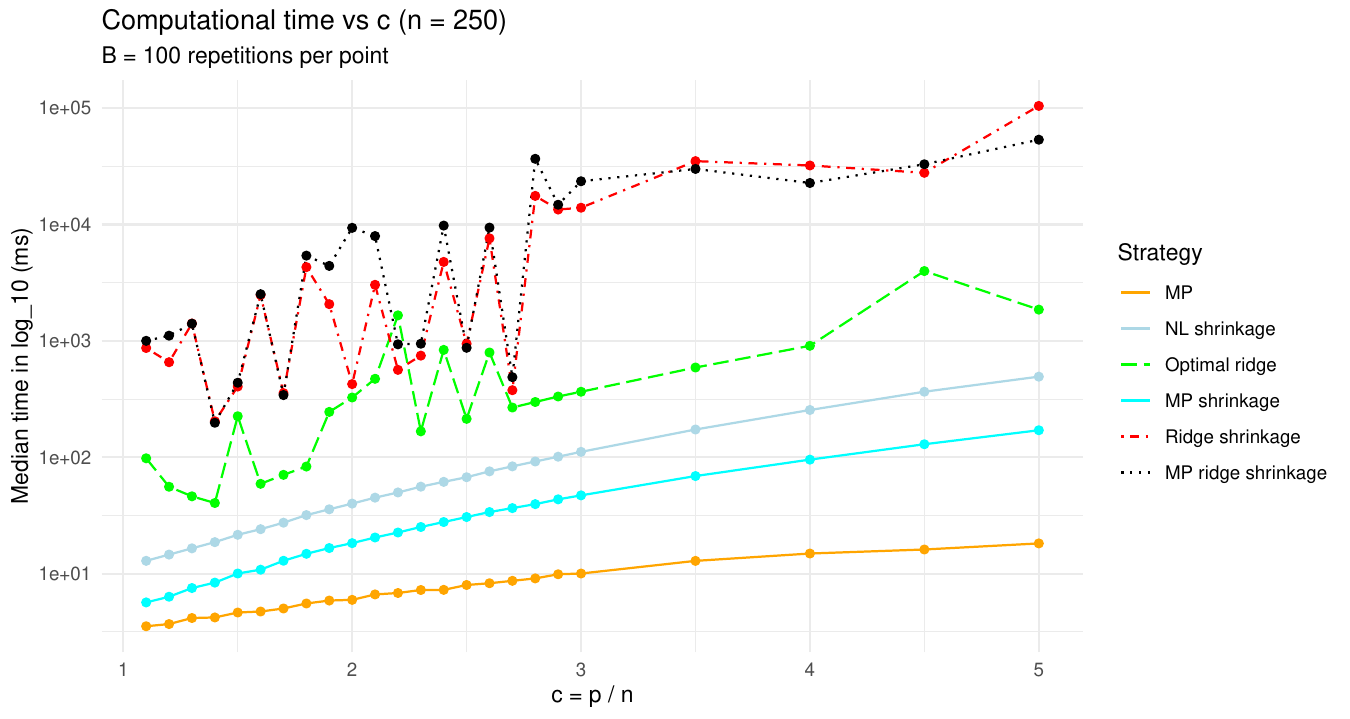}\\[-0.5cm]
\end{tabular}
 \caption{Computational time (in logarithmic scale) for precision matrix estimators for $c_n \in (1,5]$, $n=100$ (first row) and $n=250$ (second row) when the elements of $\bX_n$ are drawn from the normal distribution (first column) and scale $t$-distribution (second column).}
\label{fig:shrinkage-prec-time}
 \end{figure}
Figure \ref{fig:opt-t-prec} depicts the kernel density estimators of the relative differences between the values of $t$ which maximize $\hat{L}^2_{R;2}(t)$ and $L^2_{R;n,2}(t)$ in the case of the ridge estimator, and $L^2_{MPR;n,2}(t)$ and  $\hat{L}^2_{MPR;2}(t)$ for the Moore-Penrose ridge estimator. The results are obtained when the elements of $\bX_n$ are drawn from the normal distribution and $t$-distribution with $n \in \{100,250\}$ and $c_n \in \{1.1,2,4\}$. We observe that the computed kernel densities are concentrated around zero in almost all of the considered cases, except for the extreme case when $n=100$ and $c_n=4$. Furthermore, the kernel densities are skewed to the right with the exception when $n=250$ and $c_n=1.1$. As such, it can be concluded that the optimization of the estimated loss function instead of the true one may lead to larger values of the tuning parameter $t$. However, as the sample size increases, the difference between the two values, as expected, becomes negligible. Finally, faster convergence is achieved when the matrix $\bX_n$ is drawn from the normal distribution. 

\begin{figure}[h!t]
\centering
\begin{tabular}{cc}
$\mathbf{c_n=1.1,~n=100}$&$\mathbf{c_n=1.1,~n=250}$\\[-0.8cm]
\includegraphics[width=7cm]{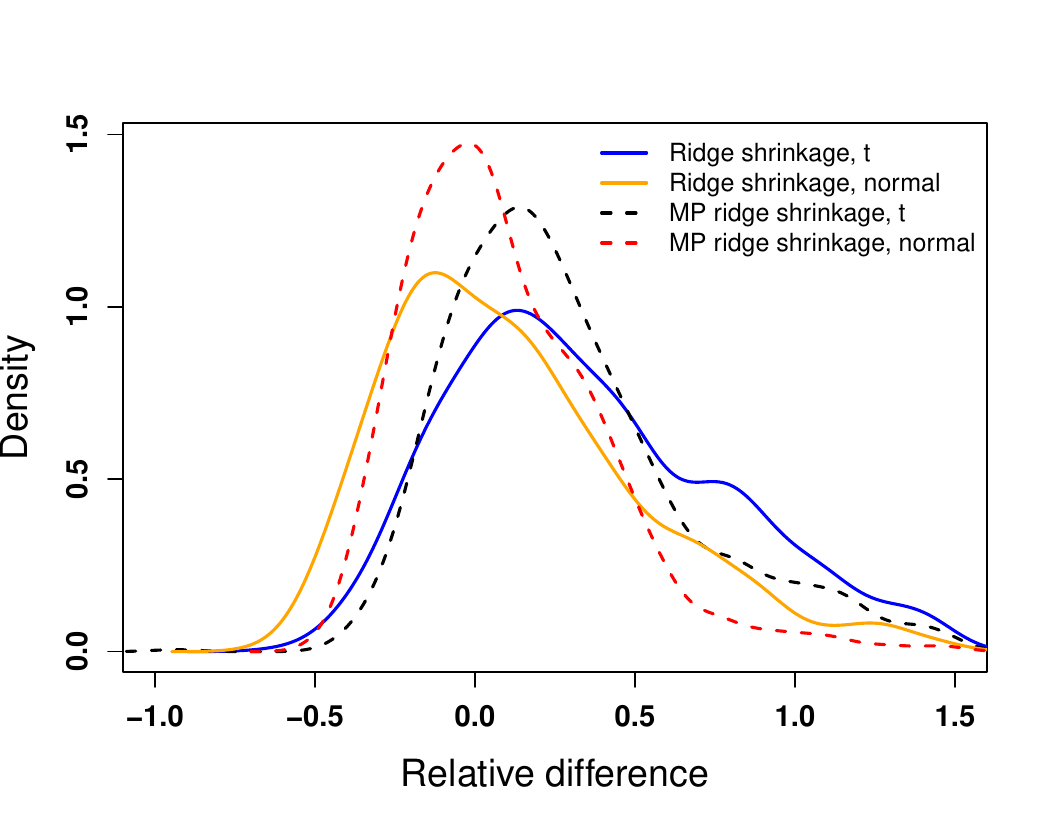}&
\includegraphics[width=7cm]{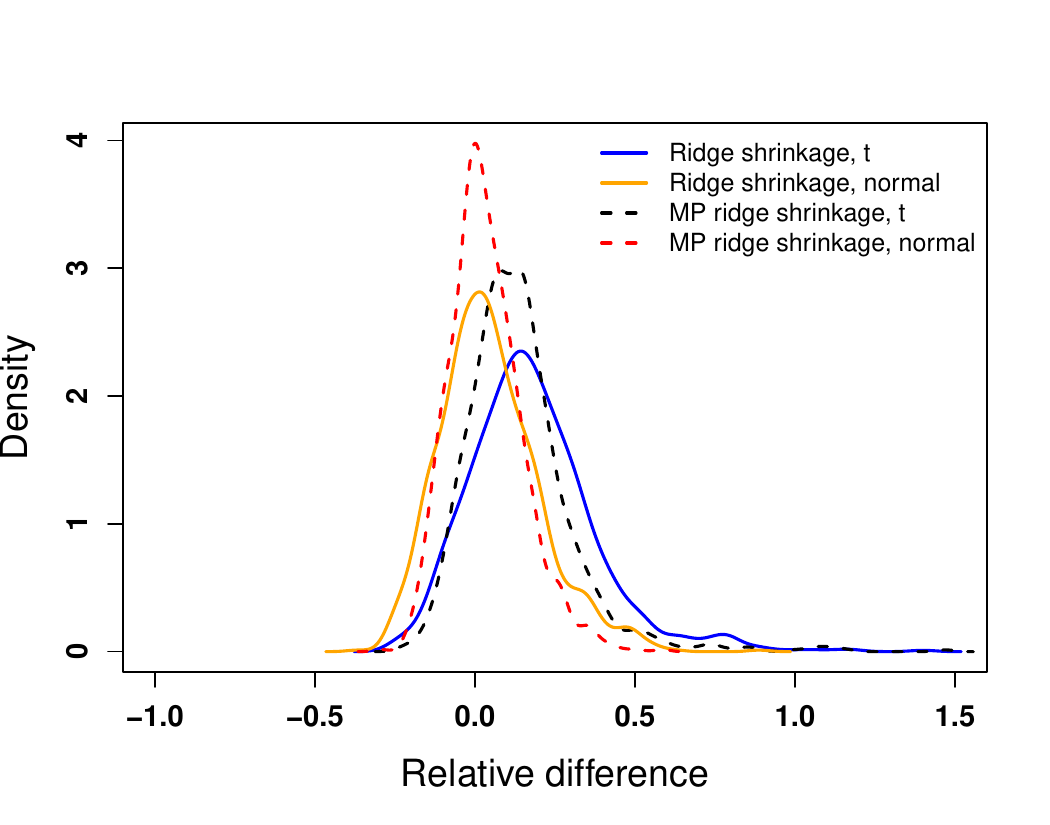}\\
$\mathbf{c_n=2,~n=100}$&$\mathbf{c_n=2,~n=250}$\\[-0.8cm]
\includegraphics[width=7cm]{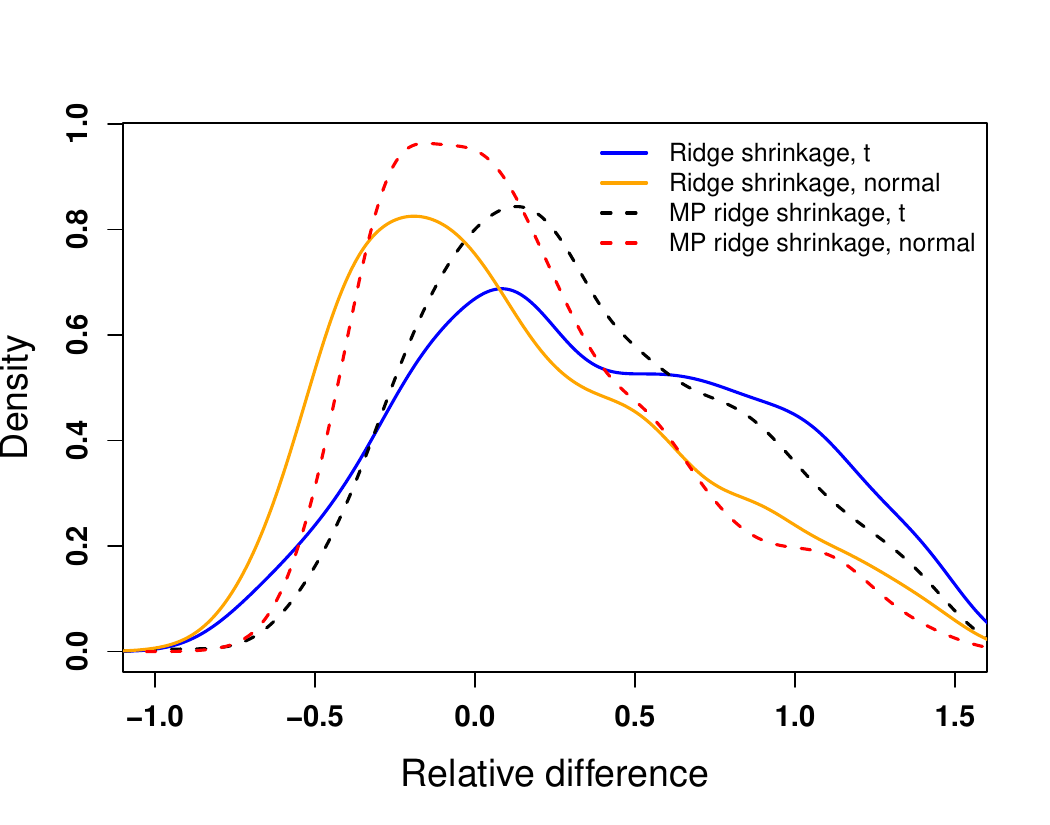}&
\includegraphics[width=7cm]{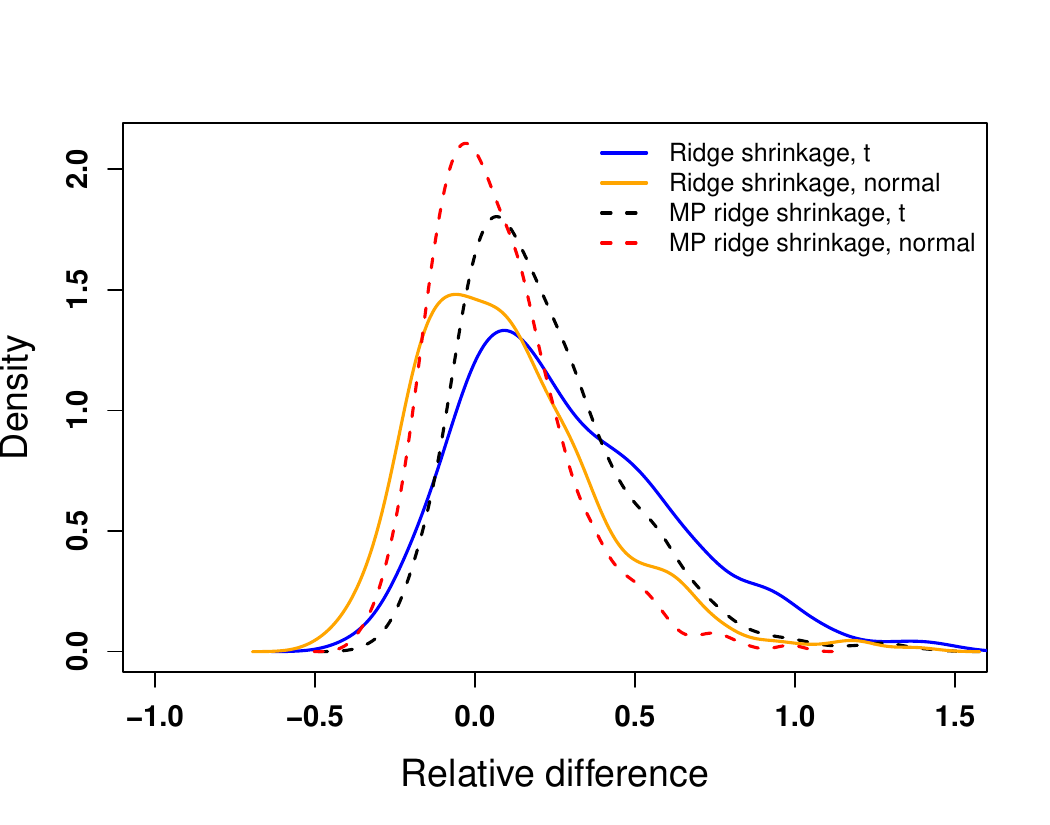}\\
$\mathbf{c_n=4,~n=100}$&$\mathbf{c_n=4,~n=250}$\\[-0.8cm]
\includegraphics[width=7cm]{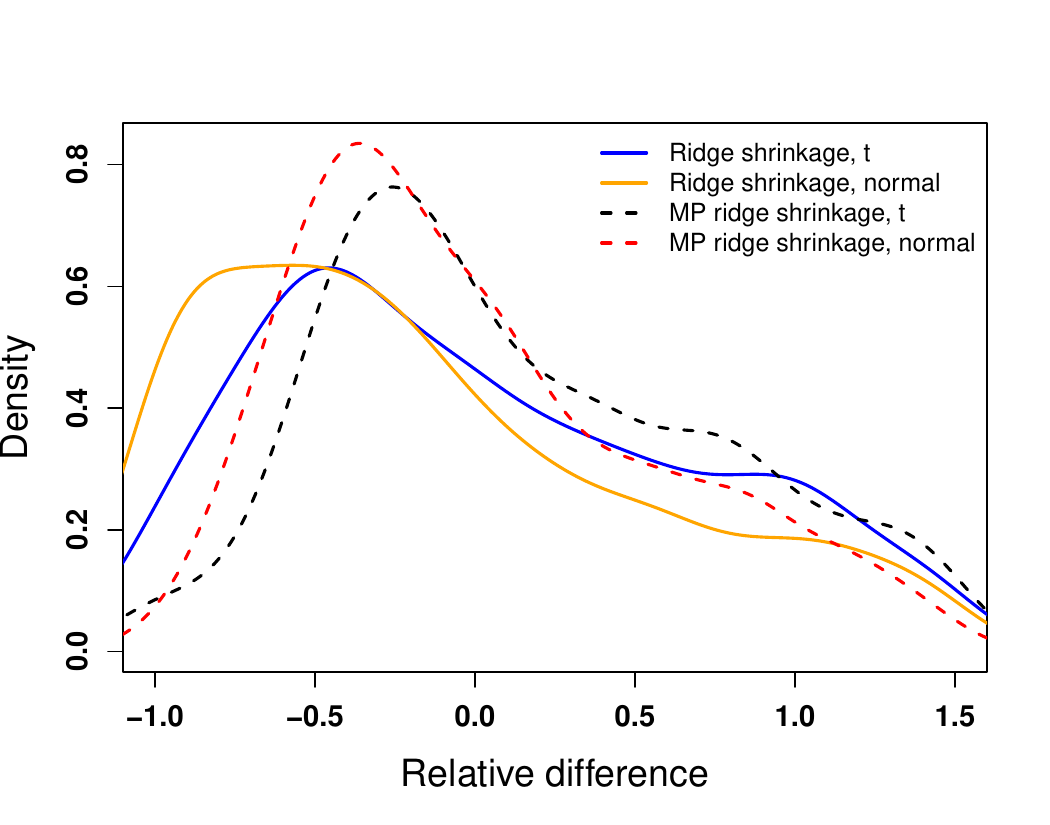}&
\includegraphics[width=7cm]{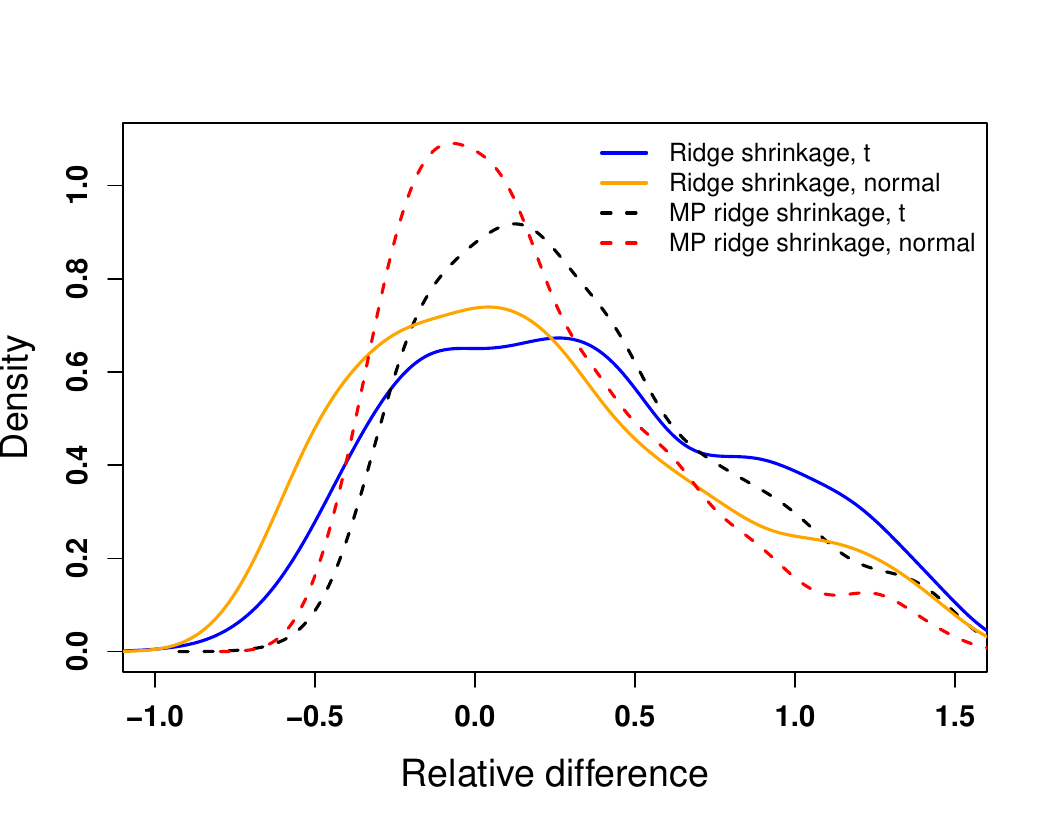}\\
\end{tabular}
 \caption{Relative differences of $t$'s that maximize $L^2_{R;n,2}(t)$ ($L^2_{MPR;n,2}(t)$) and $\hat{L}^2_{R;2}(t)$ ($\hat{L}^2_{MPR;2}(t)$)  for $c_n \in \{1.1,2,4\}$, $n=100$ and $n=250$ when the elements of $\bX_n$ are drawn from the normal distribution and scale $t$-distribution.}
\label{fig:opt-t-prec}
 \end{figure}

\subsection{Benchmarks: Global minimum variance portfolio}\label{sec:benchmarks-gmv}

Three benchmark methods to estimate the weights of the global minimum variance portfolio are considered in Section 4.2:
\begin{itemize}
\item Traditional {\bf sample estimator} is given by
\begin{eqnarray}\label{hw_GMV-S}
  \mathbf{w}_{S}=\frac{\bS_n^{+}\bOne}{\bOne^\top\bS_n^{+}\bOne},
\end{eqnarray}
where $\bS_n^{+}$ is the Moore-Penrose inverse of $\bS_n$.
\item {\bf Reflexive inverse estimator} of \cite{bodnar2018estimation} is defined by
\begin{eqnarray}\label{hw_GMV-ref}  \mathbf{w}_{Ref}=\hat{\alpha}^*_{Ref}\mathbf{w}_{S}+(1-\hat{\alpha}^*_{Ref})\mathbf{b},
\end{eqnarray}
with
\[\hat{\alpha}^*_{Ref}=\frac{(p/n-1)\hat{R}_{Ref}}{(p/n-1)^2+p/n+(p/n-1)\hat{R}_{Ref}}\]
where
\[\hat{R}_{Ref}=\frac{p}{n}\left(\frac{p}{n}-1\right)\mathbf{b}^\top\bS_n\mathbf{b}\cdot\bOne^\top\bS_n^{+}\bOne-1\]
In the simulation study, we use the equally weighted portfolio as the target portfolio $\mathbf{b}$, i.e., $\mathbf{b}=\frac{1}{p}\mathbf{1}$.

\item {\bf Double shrinkage estimator} or {\bf ridge shrinkage estimator} of \cite{BPTJMLR2024} after some simplifications is expressed as
\begin{eqnarray}\label{hw_GMV-ridge}  \mathbf{w}_{Rid}=\hat{\alpha}^*_{Rid}(\hat{\eta})\frac{\bS_n^{-}(\hat{\eta})\bOne}{\bOne^\top\bS_n^{-}(\hat{\eta})\bOne}+(1-\hat{\alpha}^*_{Rid}(\hat{\eta}))\mathbf{b},
\end{eqnarray}
with
\[\hat{\alpha}^*_{Rid}(\hat{\eta})=\frac{1-
        \frac{(1+\hat{\eta})^{-1}}{\bb^\top \bS_n \bb
        }\frac{d_1(\hat{\eta})}{ \bOne^\top \bS_n^{-}(\hat{\eta})\bOne}
        }
        {1-\frac{2(1+\hat{\eta})^{-1}}{\bb^\top \bS_n \bb} \frac{d_1(\hat{\eta})}{\bOne^\top \bS_{n}^{-}(\hat{\eta})\bOne}
        +\frac{(1+\hat{\eta})^{-2}}{\bb^\top \bS_n \bb} \frac{(1-\hat{v}_2^{(1)}(\hat{\eta}, 0))d_2(\hat{\eta})}{\left(\bOne^\top \bS_{n}^{-}(\hat{\eta})\bOne\right)^2}}
\]
where 
\begin{eqnarray*}
d_1 (\eta)&=&  \frac{(1+\eta)}{ \hat{v}(\eta, 0)}\left(1-\eta \bb^\top \bS_n^{-}(\eta) \bOne\right),\\  
d_2(\eta)&=& \frac{(1+\eta)^2}{\hat{v}(\eta, 0)}\left( \bOne^\top\bS_{n}^{-}(\eta)\bOne  - \eta\bOne^\top(\bS_{n}^{-}(\eta))^2\bOne 
\right),
\\
\hat{v}(\eta, 0) &=& 1-\frac{p}{n}\left(1-\eta\frac{1}{p}\tr\left(\bS_{n}^{-}(\eta)\right) \right),\\
\hat{v}_1^{(1)}(\eta, 0)
&=&\hat{v}(\eta,0) \frac{p}{n}\left(\frac{1}{p}\tr\left(\bS_{n}^{-}(\eta)\right)-\eta\frac{1}{p}\tr\left(\left(\bS_{n}^{-}(\eta) \right)^{2}\right)\right),\\
\hat{v}_2^{(1)}(\eta, 0) &=&
1-\frac{1}{\hat{v}(\eta,0)}+\eta\frac{\hat{v}_1^{(1)}(\eta,0)}
{\hat{v}(\eta,0)^2},
\end{eqnarray*}
and $\hat{\eta}$ maximizes of the following function
\[
 \hat{L}_{n;2}(\eta)=
    \frac{\left(1-
        \frac{(1+\eta)^{-1}}{\bb^\top \bS_n \bb
        }\frac{d_1(\eta)}{ \bOne^\top \bS_n^{-}(\eta)\bOne}\right)^2
        }
        {1-\frac{2(1+\eta)^{-1}}{\bb^\top \bS_n \bb} \frac{d_1(\eta)}{\bOne^\top \bS_{n}^{-}(\eta)\bOne}
        +\frac{(1+\eta)^{-2}}{\bb^\top \bS_n \bb} \frac{(1-\hat{v}_2^{(1)}(\eta, 0))d_2(\eta)}{\left(\bOne^\top \bS_{n}^{-}(\eta)\bOne\right)^2}}.
\]
In the simulation study, the target portfolio $\bb$ is set to be equal to the equally weighted portfolio.
\end{itemize}

\subsection{Additional figures: Global minimum variance portfolio}\label{sec:S-fig-gmv}

In Figure \ref{fig:shrinkage-gmv-add}, we present the results of the simulation study for the proposed MP shrinkage estimator of the GMV portfolio weights alongside three benchmark approaches. The results in Figure \ref{fig:shrinkage-gmv-add} complement those in Figure 4 by including the rOSV values computed for the traditional GMV portfolio estimator obtained by directly plugging in the Moore-Penrose inverse (see \eqref{hw_GMV-S}). Figure  \ref{fig:shrinkage-gmv-time} complements the findings from Figure \ref{fig:shrinkage-gmv-add} by presenting the corresponding computational time.
\begin{figure}[h!t]
\centering
\begin{tabular}{cc}
\hspace{-0.5cm}\includegraphics[width=7cm]{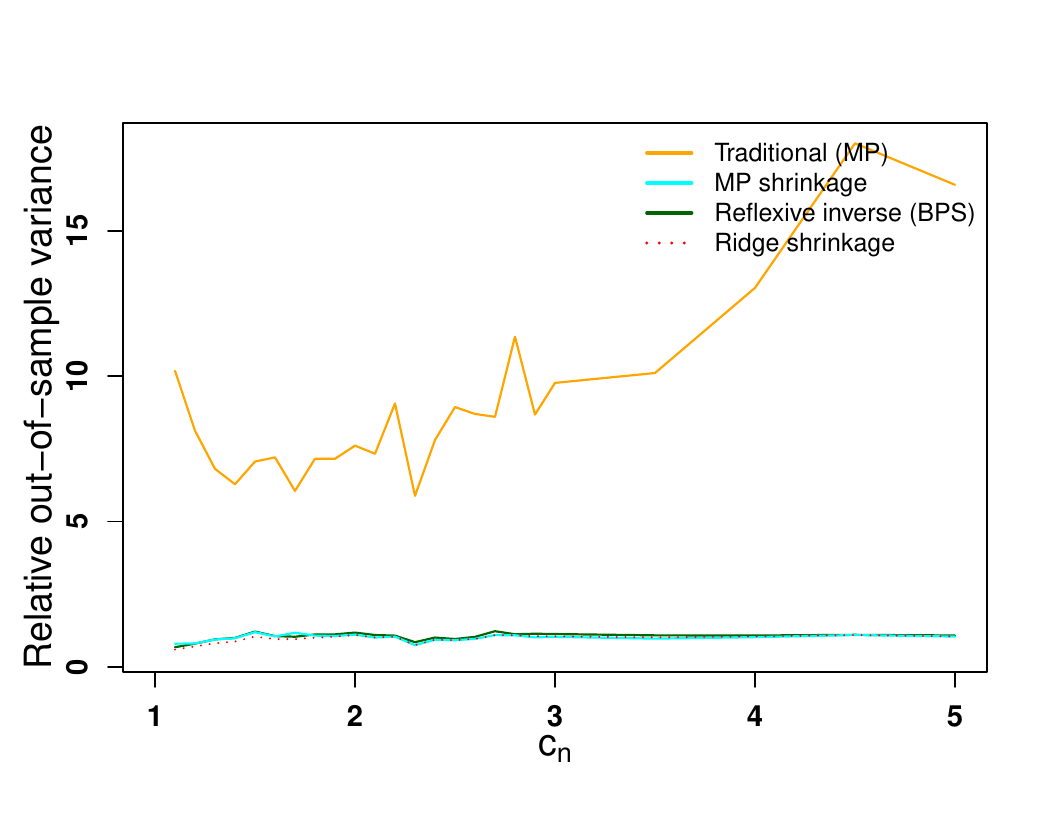}&
\hspace{-0.5cm}\includegraphics[width=7cm]{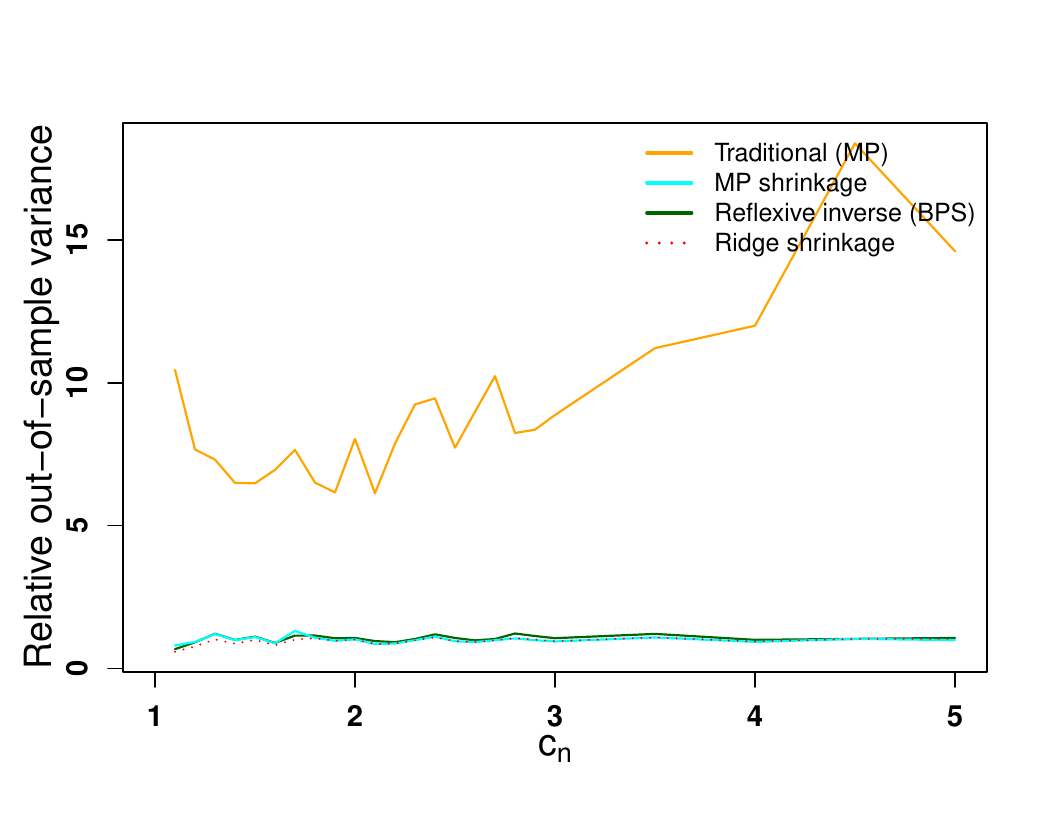}\\[-1cm]
\hspace{-0.5cm}\includegraphics[width=7cm]{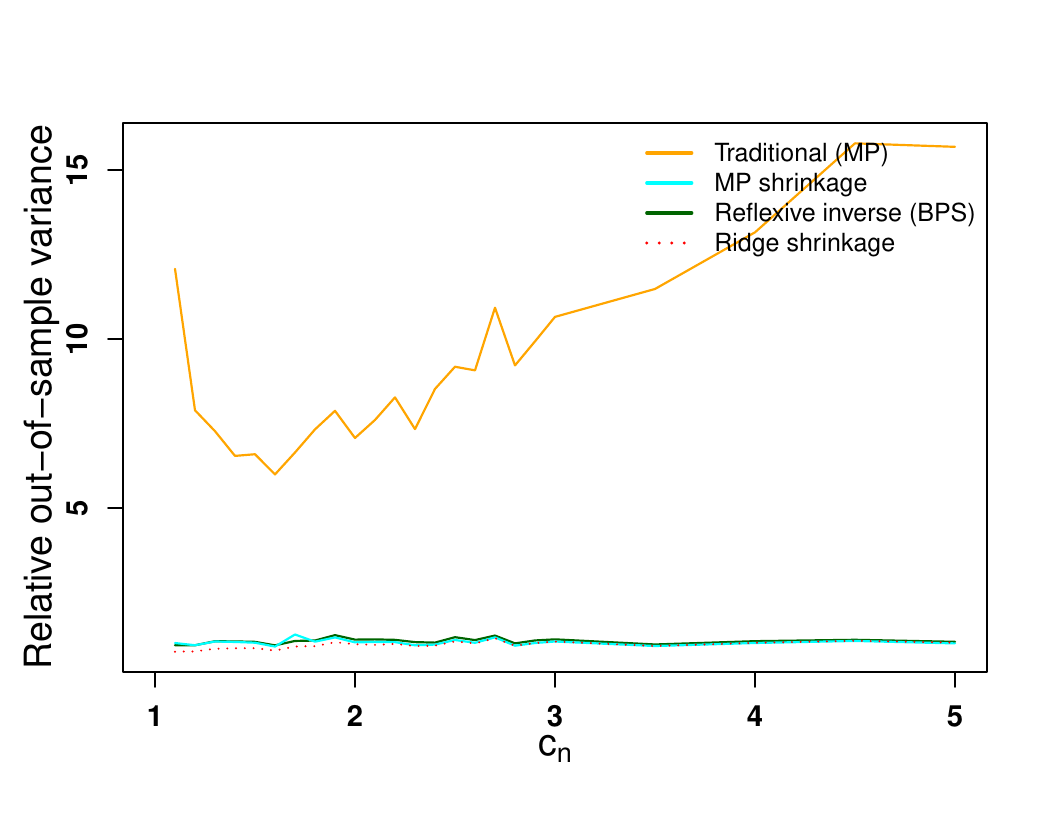}&
\hspace{-0.5cm}\includegraphics[width=7cm]{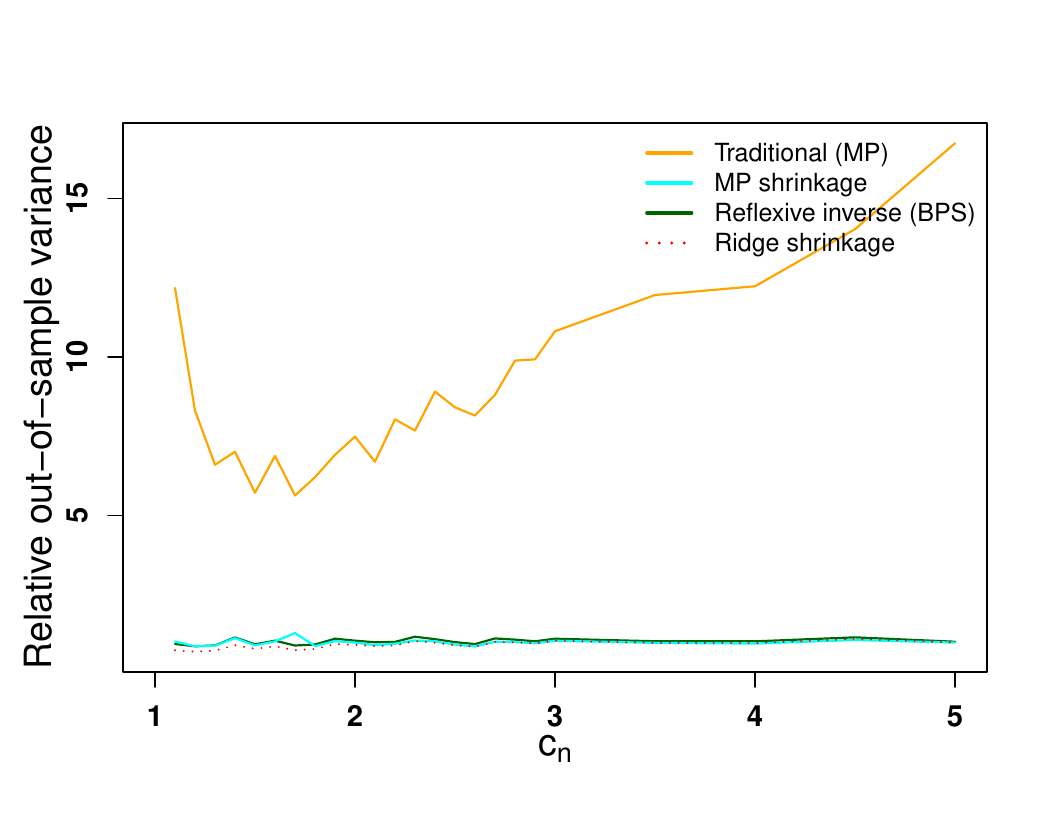}\\[-0.7cm]
\end{tabular}
 \caption{rOSV for $c_n \in (1,5]$, $n=100$ (first row) and $n=250$ (second row) when the elements of $\bX_n$ are drawn from the normal distribution (first column) and scale $t$-distribution (second column).}
\label{fig:shrinkage-gmv-add}
 \end{figure}
 \begin{figure}[h!t]
\centering
\begin{tabular}{cc}
\hspace{-0.5cm}\includegraphics[scale=0.35]{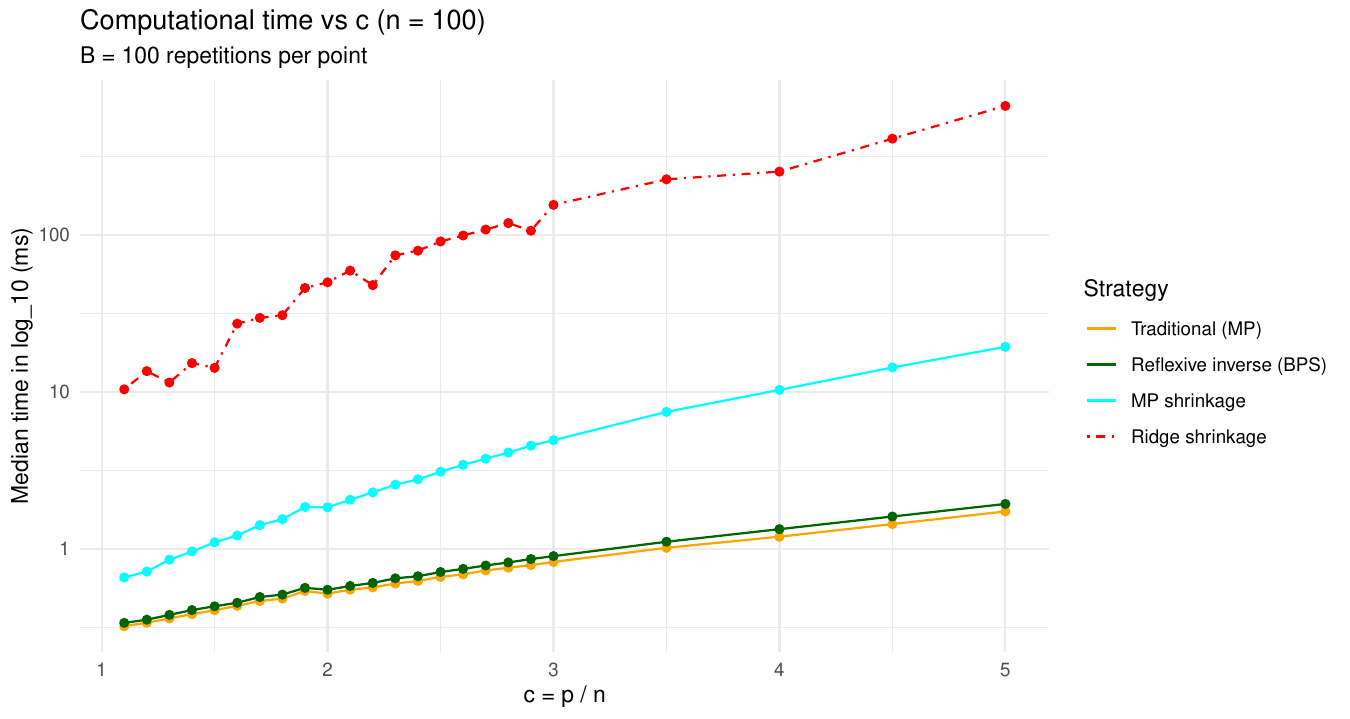}&
\hspace{-0.5cm}\includegraphics[scale=0.35]{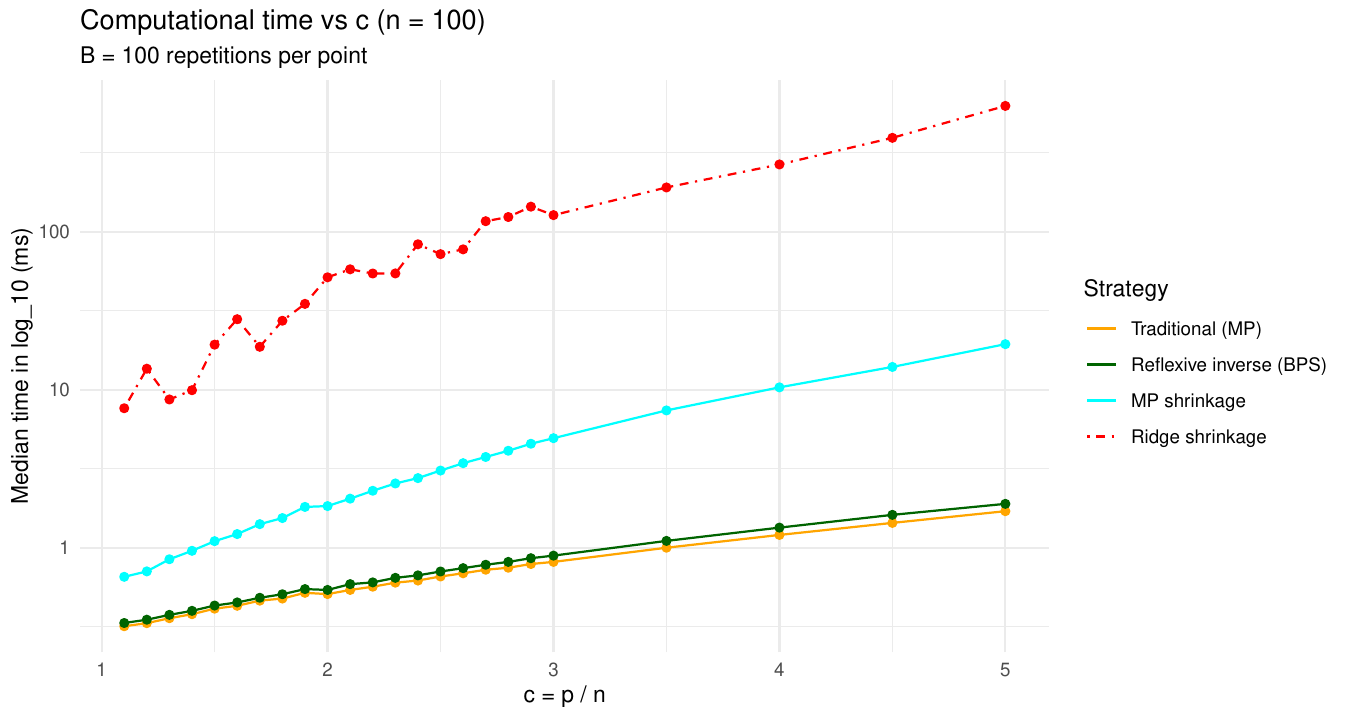}\\[-0.5cm]
\hspace{-0.5cm}\includegraphics[scale=0.35]{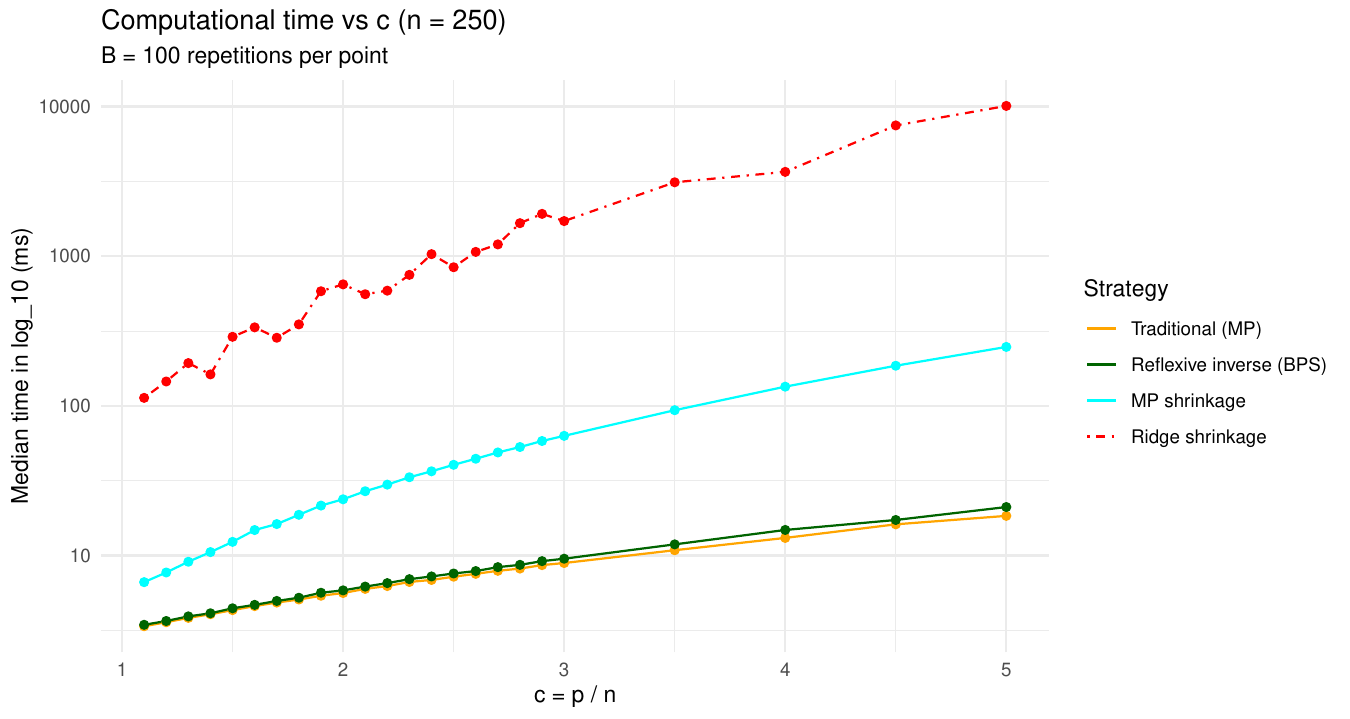}&
\hspace{-0.5cm}\includegraphics[scale=0.35]{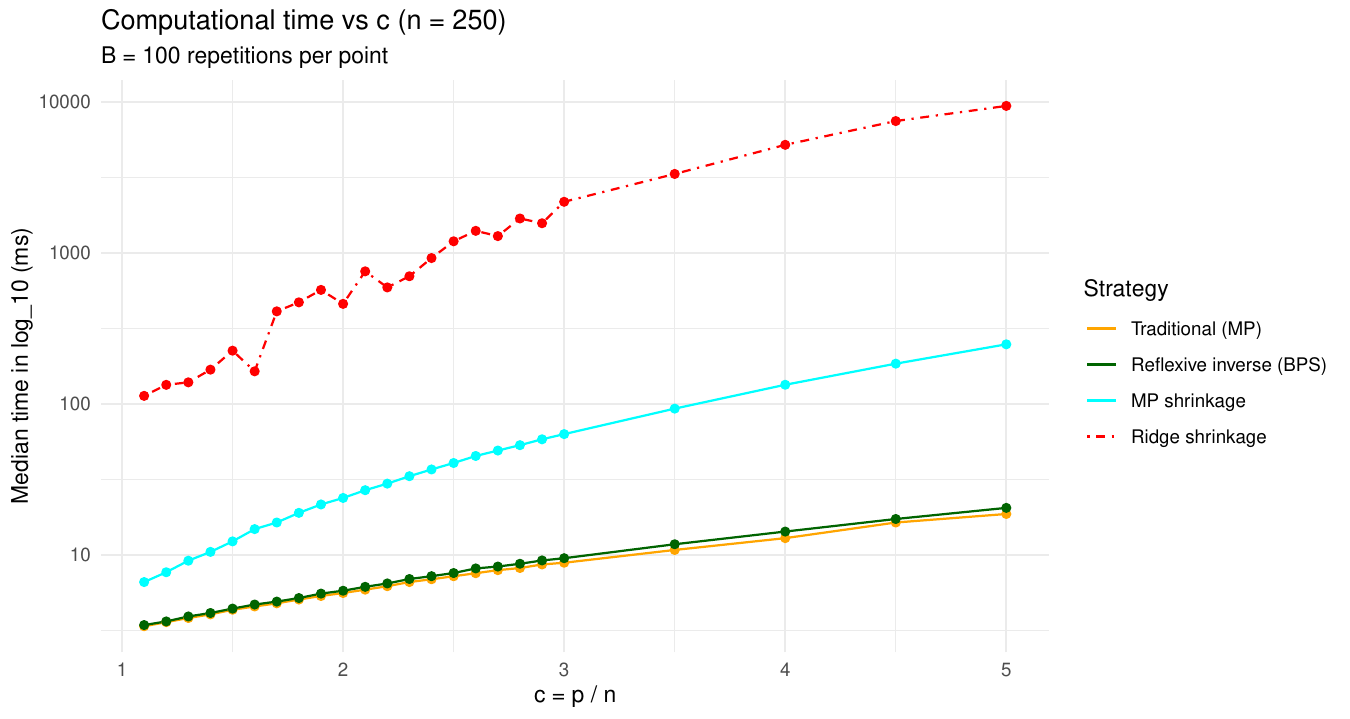}\\[-0.5cm]
\end{tabular}
 \caption{Computational time (in logarithmic scale) for GMV estimators for $c_n \in (1,5]$, $n=100$ (first row) and $n=250$ (second row) when the elements of $\bX_n$ are drawn from the normal distribution (first column) and scale $t$-distribution (second column).}
\label{fig:shrinkage-gmv-time}
 \end{figure}

\bibliography{ref}

\end{document}